\documentclass{mcom-l}

\usepackage{lipsum}
\usepackage{amsfonts}
\usepackage{graphicx}
\usepackage{epstopdf}
\usepackage{algorithmic}
\usepackage{booktabs} % For professional looking tables
\usepackage{ulem}
\usepackage{amsmath}
\allowdisplaybreaks[4]
\usepackage{cancel}
\usepackage{multirow}
\usepackage{graphicx}
\usepackage{amssymb}
\usepackage{color}
\usepackage{mathtools}
\usepackage{tikz}
\usepackage[hidelinks]{hyperref}
\usepackage{amsrefs}
\usepackage{float}
\usepackage{caption}
\usepackage[toc,page]{appendix}
\usepackage{mathrsfs}
\usepackage{mathtools}
 \DeclarePairedDelimiter{\norm}{\lVert}{\rVert}
\usepackage{latexsym, bm}
\ifpdf
  \DeclareGraphicsExtensions{.eps,.pdf,.png,.jpg}
\else
  \DeclareGraphicsExtensions{.eps}
\fi

\def\vT{\vartheta}

\def\no{{\nonumber}}

\newcommand{\mNQ}{\mathbf{N}_{h}}
\newcommand{\mNu}{\mathcal{N}_{h}}
\newcommand{\mNQz}{\mathbf{N}}
\newcommand{\mNuz}{\mathcal{N}}
\newcommand{\muQ}{\mathbf{H}_h}
\newcommand{\muu}{\mu_h}
\newcommand{\muQz}{\mathbf{H}}
\newcommand{\muuz}{\mu}
\newcommand{\muQa}{\mathbf{H}_{\rm p}}
\newcommand{\muua}{\mu_{p}}
\newcommand{\muQb}{\mathbf{H}_{\rm d}}

\newcommand{\muQaz}{\mathbf{H}_{\rm p}}
\newcommand{\muuaz}{\mu_{p}}
\newcommand{\muQbz}{\mathbf{H}_{\rm d}}

\newcommand{\mLQ}{\mathcal{L}_{h}}
\newcommand{\mLu}{\mathcal{D}_{h}}

\newcommand{\id}[2]{\left\langle #1,#2\right\rangle_h}
\newcommand{\ia}[1]{ \norm{#1 }_{h}^2}

\newcommand{\ilQ}[1]{\norm{#1 }_{\mathcal{L}_h}^2}
\newcommand{\ilu}[1]{\norm{#1 }_{\mathcal{D}_h}^2}

\newcommand{\iaQ}[1]{\norm{#1 }^2_{\mathcal{Q}_{\mathcal{L}}}}
\newcommand{\ibQ}[1]{\norm{#1 }^2_{\mathcal{Q}_{\mathcal{L}}^1}}
\newcommand{\icQ}[1]{\norm{#1 }^2_{\mathcal{Q}_{\mathcal{L}}^2}}
\newcommand{\iau}[1]{\norm{#1 }^2_{\mathcal{Q}_{\mathcal{D}}}}
\newcommand{\ibu}[1]{\norm{#1 }^2_{\mathcal{Q}_{\mathcal{D}}^1}}

 \newcommand{\mQ}{\mathcal{Q}(\tau \mLQ)}
\newcommand{\dtau}{\delta_\tau}
\newcommand{\be}{\boldsymbol{e}}

\newcommand{\ep}{\epsilon}

\newcommand{\bM}{\mathbf{M}}

 \newcommand{\ea}{\delta_\tau \be_{\mathbf{Q}}^{n+1}}
\newcommand{\eb}{\delta_\tau e_{u}^{n+1}}
\newcommand{\ec}{\delta_\tau e_{s}^{n+1}}
\newcommand{\vertiii}[1]{{\left\vert\kern-0.23ex\left\vert\kern-0.23ex\left\vert #1
		\right\vert\kern-0.23ex\right\vert\kern-0.23ex\right\vert}}

 \newcommand{\da}{\delta_\tau \mathbf{Q}(t_{n+1})}
 \newcommand{\db}{\delta_\tau u(t_{n+1})}
\def\ba{\mathbf{a}}

\def\be{\mathbf{e}}

\def\A{\alpha}
\def\B{\beta}
\def\C{\gamma}

\def\bu{u}

\def\bQ{\mathbf{Q}_h}
\def\bu{u_h}

\newcommand{\mq}{\mathcal{Q}}

\def\ma{\mathcal{A}}

\def\ts{{s}}
\def\bX{\mathbf{X}}

 \def\ba{\begin{equation}\begin{aligned}}
 \def\ed{\end{aligned}\end{equation}}
\newtheorem{theorem}{Theorem}[section]
\newtheorem{lemma}[theorem]{Lemma}

\theoremstyle{definition}

\theoremstyle{remark}
\newtheorem{remark}[theorem]{Remark}

\numberwithin{equation}{section}

\begin{document}

\title[GSAV-EI scheme for smectic-A phase]{A Structure-Preserving GSAV Exponential Integrator for Smectic-A Liquid Crystals}


%    author one information
\newcommand{\BNUaddress}{%
Laboratory of Mathematics and Complex Systems, Ministry of Education,
School of Mathematical Sciences, Beijing Normal University,
Beijing 100875, China}

\author{Wenshuai Hu}
\address{\BNUaddress}
\email{202431130052@mail.bnu.edu.cn}

\author{Guanghua Ji}
\address{\BNUaddress}
\email{ghji@bnu.edu.cn}
\thanks{Corresponding author. G. Ji is partially supported by the National Natural Science Foundation of China (Grant No. 12471363)}

\author{Xiao Li}
\address{\BNUaddress}
\email{lixiao@bnu.edu.cn}
\thanks{X. Li is partially supported by  Beijing Normal University (Grant No. 312200502508)}
%    \subjclass is required.
\date{}

\dedicatory{}

%    Abstract is required.
\begin{abstract}
The modified Landau--de Gennes (mLdG) theory provides a powerful continuum framework for modeling smectic-A (SmA) liquid crystals by coupling the tensorial orientational order parameter $\mathbf{Q}$ with the scalar positional order parameter $u$.
    In this paper, we develop and analyze a structure-preserving generalized scalar auxiliary variable exponential integrator (GSAV-EI) scheme for the fully coupled mLdG system.
   The key ingredient is a backward Euler-type reformulation of the exponential integrator, which reveals a coercive discrete structure suitable for energy estimates and error analysis.
Moreover, it eliminates the mesh-dependent time-step restriction ($\tau \lesssim h^2$) required in the existing GSAV-EI error analysis and provides a transparent connection between exponential time-differencing and implicit time-stepping methods.
Building on this reconstructed structure, we prove unconditional modified-energy stability and establish optimal-order fully discrete error estimates.
   Numerical experiments in two and three dimensions validate the theoretical convergence rates, verify the discrete energy-dissipation law, and illustrate the self-assembly dynamics of the SmA phase.
\end{abstract}

\maketitle
\section{Introduction}
Liquid crystals are mesophases that exhibit properties intermediate between those of conventional solids and liquids %states and are generally known as the fourth state of matter
\cite{han2015microscopic,liu2007dynamics}.
%According to the orientational and the positional order, liquid crystals can exhibit different phases, such as nematic, cholesteric, and smectic phases
Depending on the degree of orientational and positional order, they can form nematic, cholesteric, and smectic phases
\cite{ball2011orientability,wang2021modelling,biscari2007landau}.
%The nematic phase exhibits only long-range orientational order, whereas the smectic phase develops an additional one-dimensional positional order, organizing molecules into layered structures \cite{wang2021modelling}.
Among smectic phases, Smectic-A (SmA) liquid crystals are characterized by a dual-order structure: the molecules form one-dimensional density layers, and the orientational director is aligned with the layer normal \cite{xia2021structural,xia2023variational,E1997modeling}. Although the classical Landau-de Gennes (LdG) theory has been highly successful in describing isotropic-nematic transitions and nematic defect structures \cite{fei2018isotropic,majumdar2010equilibrium}, it does not directly capture the positional ordering that is essential for the SmA phase \cite{shi2025modified}. The mLdG theory addresses this limitation by coupling the orientational Q-tensor with a scalar positional order parameter \cite{Ball03052015,xia2021structural,xia2024simple}, thereby providing a thermodynamically consistent framework for the isotropic-nematic-smectic transition and resulting defect dynamics \cite{xia2023variational,shi2025modified}.

For the SmA liquid crystal system considered in this paper, the mLdG energy proposed by Xia et al. \cite{xia2021structural,xia2023variational} is given by
\begin{equation}
    E[\mathbf{Q}, u] = \int_{\Omega} \left( f_{\mathrm{LdG}}(\mathbf{Q}, \nabla \mathbf{Q}) + f_{\mathrm{bs}}(u) + f_{\mathrm{int}}(\mathbf{Q}, u) \right) \, \mathrm{d}\mathbf{x},
\end{equation}
where $\Omega\subset\mathbb{R}^d$ ($d=2,3$) is a smooth, bounded domain.
The LdG order parameter $\mathbf{Q}(\mathbf{x})$ is a symmetric and traceless
$d\times d$ tensor, while $u(\mathbf{x})$ is a scalar order parameter
representing the relative deviation of the molecular density from its mean
\cite{shi2025modified}.
% More precisely, the $\mathbf{Q}$-tensor is defined as
% \begin{equation}
%     \mathbf{Q} = s_+ \left( \mathbf{n} \otimes \mathbf{n} - \frac{\mathbf{I}_d}{d} \right),
% \end{equation}
% where $s_+$ is the scalar order parameter and $\mathbf{n}$ is the unit vector along the director \cite{wang2021modelling}.

In the classical LdG theory \cite{deGennes1974}, the free-energy density
$f_{\mathrm{LdG}}$ is given by
% The LdG energy density $f_{\mathrm{LdG}}$ is given by the classical LdG theory \cite{deGennes1974} as follows:
\begin{equation}
    f_{\mathrm{LdG}}(\mathbf{Q}, \nabla \mathbf{Q}) = \frac{K}{2} |\nabla \mathbf{Q}|^2 +  \frac{\A}{2}\operatorname{tr}(\mathbf{Q}^2) - \frac{\B}{3}\operatorname{tr}(\mathbf{Q}^3) + \frac{\C}{4}(\operatorname{tr}(\mathbf{Q}^2))^2,
\end{equation}
where $|\nabla \mathbf{Q}|^2=\sum_{i,j,k}(\partial_i Q_{jk})^2$
and $\operatorname{tr}(\Phi)=\sum_i\Phi_{ii}$ for any
$\Phi\in\mathbb{R}^{d\times d}$.
Here, $K>0$ denotes the elastic constant, while
$\A=a_1(\vT-\vT_1^*)$ is a temperature-dependent coefficient with
$a_1>0$, where $\vT$ is the temperature and $\vT_1^*$ is the
characteristic temperature for nematic ordering. The constants $\B\ge 0$ and
$\C>0$ are  material-dependent bulk constants.

Following the Landau phase transition theory \cite{pevnyi2014modeling,izzo2020landau}, the smectic bulk energy density $f_{\mathrm{bs}}(u)$ is a polynomial function of $u$ given by
% The smectic bulk energy density $f_{\mathrm{bs}}(u)$, derived from the Landau phase transition theory \cite{pevnyi2014modeling,izzo2020landau}, is a polynomial function of $u$ given by
\begin{equation}
    f_{\mathrm{bs}}(u) = \frac{a}{2}u^2 + \frac{b}{3}u^3 + \frac{c}{4}u^4,
\end{equation}
where $a=a_2(\vT-\vT_2^*)$  is temperature dependent with $a_2>0$,
$\vT_2^*<\vT_1^*$ is the characteristic temperature for the nematic-smectic
transition, while $b,c$ are material constants with $c>0$; in particular,
$b=0$ for symmetric molecules \cite{pevnyi2014modeling}.

Following the mLdG theory \cite{Ball03052015,shi2025modified,xia2021structural}, the coupling term $f_{\mathrm{int}}(\mathbf{Q},u)$ captures the interaction
between $\mathbf{Q}$ and $u$ and is given by
%
% The coupling term $f_{\mathrm{int}}(\mathbf{Q}, u)$  is designed to ensure that the energy is minimized when the orientational order $\mathbf{Q}$ and the positional order $u$ are compatible with the layered structure of SmA phase. Following \cite{Ball03052015,shi2025modified}, the explicit formulation for $f_{\mathrm{int}}$ is given as
\begin{align*}
    f_{\mathrm{int}}(\mathbf{Q}, u) & = \begin{cases}
                                    B_0 \left| D^2 u + q^2 \left( \frac{\mathbf{Q}}{s_+} + \frac{\mathbf{I}_d}{d} \right) u \right|_{F}^2,                                                                                    & (d=3,\ \A < \frac{\B^2}{27\C})
~\text{or}~
(d=2,\ \A < 0), \\
                                   B_0 |D^2 u|_{F}^2, & \text{otherwise},
                               \end{cases}\\
                   s_+=&\dfrac{\B+\sqrt{\B^2-24\A\C}}{4\C}~(d=3,\ \A<\frac{\B^2}{27\C}),\quad
s_+=\sqrt{-2\A/\C}~(d=2,\ \A<0),
\end{align*}
where $|\Phi|_F=\sqrt{\Phi : \Phi}=\sqrt{\sum_{i,j} \Phi_{ij}^2}$ is the Frobenius norm for any  $\Phi$ in $\mathbb{R}^{d\times d}$,
$B_0 > 0$ is the coupling strength, $q = 2\pi/l$ is the wave number for the smectic layer thickness $l$, $\mathbf{I}_d$ is the $d \times d$ identity matrix, and $D^2 u$  denotes the Hessian matrix of $u$.
% Subject to the symmetry and traceless constraints on $\mathbf{Q}$,

The $L^2$-gradient flow associated with $E[\mathbf{Q},u]$ leads to the following coupled nonlinear PDE system:
\ba \label{1.3a}
\frac{\partial \mathbf{Q}}{\partial t} = - \left( \frac{\delta E}{\delta \mathbf{Q}} \right) + \lambda \mathbf{I}_d + \nu - \nu^{T},\qquad
\frac{\partial u}{\partial t} = -  \frac{\delta E}{\delta u},
\ed
where  $\lambda$ and $\nu = (\nu^{ij})_{d\times d}$ are Lagrange multipliers enforcing the tracelessness and symmetry constraints on $\mathbf{Q}$.
When $(d=3,\ \A < \frac{\B^2}{27\C})$
or
$(d=2,\ \A < 0)$, the resulting system  of \eqref{1.3a} with periodic boundary conditions   can be explicitly expressed as
% \begin{small}
    \begin{equation}\label{eq1_9}
    \begin{aligned}
        \frac{\partial \mathbf{Q}}{\partial t} & = K\Delta\mathbf{Q} - \left[ \A\mathbf{Q} - \B \left( \mathbf{Q}^2 - \frac{\operatorname{tr}(\mathbf{Q}^2)}{d}\mathbf{I}_d \right) + \C \operatorname{tr}(\mathbf{Q}^2)\mathbf{Q} \right] \\
                                               & \quad -  \frac{2B_0q^2}{s_+}   \left( u D^2 u - \frac{\operatorname{tr}(u D^2 u)}{d}\mathbf{I}_d \right) -  \frac{2B_0q^4}{s_+^2} \mathbf{Q}u^2,                   \\
        \frac{\partial u}{\partial t}          & = - 2B_0\Delta^2 u - au - bu^2 - cu^3  - 2B_0 D^2 u : (q^2\bM )                                                                                                              \\
                                               & \quad - 2B_0 q^2\nabla \cdot \left( \nabla \cdot \left( \bM u \right) \right)  - 2B_0  \left| q^2 \bM \right|_F^2 u,
    \end{aligned}
\end{equation}
% \end{small}
where $\bM=\frac{\mathbf{Q}}{s_+} + \frac{\mathbf{I}_d}{d}$.
In the complementary case, however, system \eqref{eq1_9} reduces to a
degenerate SmA model, which can be obtained by setting $q = 0$.

Building upon the LdG energy, the mathematical analysis and numerical simulation of nematic liquid crystals have been extensively studied in the past decades \cite{izzo2020landau,pevnyi2014modeling,chen1976landau,huangGlobalWellposednessDynamical2015,huang2024errorestimateorderenergy,Ji2020BDF2}.
Recently, the formulation and analysis of the  mLdG theory have been systematically advanced \cite{Ball03052015,xia2021structural}.  Xia et al. \cite{xia2023variational}   established the existence of global minimizers for the mLdG energy  and derived a priori error estimates for its FEM scheme in
the decoupled case ($q=0$).
 Furthermore, Shi et al. \cite{shi2025modified} rigorously established the existence of global weak solutions to this system, and analytically proved that the mLdG model can capture the isotropic-nematic-smectic (I-N-S) phase transition as a function of temperature. However,   there is a critical need  to design an  energy-stable, decoupled numerical scheme for  \eqref{eq1_9}.

Energy-stable time discretizations have become a central tool for the numerical simulation of gradient flow systems. Representative approaches include convex splitting  \cite{eyre1998unconditionally,baskaran2013convergence}, stabilized implicit-explicit (IMEX) schemes \cite{xu2006stability,shen2010numerical,feng2013stabilized}, discrete gradient methods \cite{du1991numerical,furihata2001stable}, exponential time differencing (ETD)  \cite{du2018stabilized,du2019,du2021,liu2025maximum}, invariant energy quadratization (IEQ) \cite{yang2017linearly,xu2019efficient,yang2020convergence}, and scalar auxiliary variable (SAV) methods \cite{jiang2022improving,shen2019new,shen2018scalar}. ETD methods are attractive because they integrate stiff linear operators exactly, while SAV-type methods provide a flexible mechanism for constructing linear and unconditionally energy-stable schemes for nonlinear gradient flows \cite{du2021,shen2019new}. A natural idea is therefore to combine these two advantages. Ju et al. \cite{ju2022generalized} proposed a GSAV-EI framework for Allen-Cahn-type gradient flows. However, the original analysis relies on a strong time step restriction, which limits its applicability to fully discrete error analysis and to more complicated coupled systems.

%In recent years, numerical schemes preserving the energy dissipation law have attracted a lot of attention for time integration of the  gradient flows, including convex splitting schemes \cite{eyre1998unconditionally,baskaran2013convergence}, stabilized implicit-explicit (IMEX) schemes \cite{xu2006stability,shen2010numerical,feng2013stabilized}, discrete gradient schemes \cite{du1991numerical,furihata2001stable}, exponential time differencing (ETD) schemes \cite{du2018stabilized,du2019,du2021,liu2025maximum}, invariant energy quadratization (IEQ) schemes \cite{yang2017linearly,xu2019efficient,yang2020convergence}, and scalar auxiliary variable (SAV) schemes \cite{shen2019new,shen2018scalar,zhang2022generalized,jiang2022improving}.
%In particular, the ETD method exactly integrates stiff linear operators, while the SAV method guarantees unconditional energy stability \cite{du2021,shen2019new}.
%To combine these excellent structural properties, Ju et al. \cite{ju2022generalized} merged these two approaches and proposed a linear GSAV-EI scheme for Allen-Cahn type gradient flows.
%Nevertheless, the theoretical analysis of this  approach inevitably requires a severe time-step restriction, which strictly precludes a fully discrete error estimate and limits the scheme's practical applications to complex coupled systems.

The present work removes this obstacle by reformulating the GSAV-EI discretization into an equivalent backward-Euler-type structure. This reformulation preserves the exponential integrator character of the method while making the stability and error analysis resemble that of an implicit time-stepping scheme. We apply this strategy to the mLdG SmA system, which is a highly nonlinear tensor--scalar coupled model involving fourth-order diffusion and derivative coupling terms in the potential energy. The resulting method is linear, decoupled, and unconditionally energy stable. Moreover, the new
formulation allows us to prove optimal fully discrete error estimates without the severe time-step restriction required by the original GSAV-EI analysis.
%In this paper, we systematically address this  theoretical obstacle by reformulating the exponential integrator structure thereby broadening the applicability of the GSAV-EI approach.
%In particular, we successfully apply the approach to the SmA system \eqref{eq1_9}, which is a fourth-order, highly nonlinear, tensor-scalar coupled model with derivative coupling terms  in the potential energy.
%Moreover, we rigorously establish the unconditional energy stability, uniform boundedness and optimal error estimates under our reformulated exponential integration framework. This indicates that our reformulated scheme can reproduce nearly all the theoretical analytical results associated with the traditional Duhamel-based ETD structure, while retaining the analytical simplicity of the backward Euler method. By bridging these two frameworks, this reformulation opens up entirely new possibilities for the theoretical analysis of ETD-type schemes.

%While the SAV approach has revolutionized the design of linear, unconditionally energy-stable schemes for gradient flows, it fundamentally preserves a modified discrete energy rather than the original physical free energy. Recently, a series of SAV-type or SAV-based methods have been proposed to solve this issue, such as MSAV \cite{cheng2018multiple}, RSAV \cite{jiang2022improving}, R-GSAV \cite{zhang2022generalized} and  EOP-SAV \cite{liu2024novel,zhang2025novel}. Notably, introducing relaxation techniques has emerged as a highly effective strategy to dynamically pull the modified energy significantly closer to the original free energy.

  The main contributions of this work are summarized as follows:
\begin{itemize}
    \item We develop a linear, decoupled, and structure-preserving GSAV-EI finite difference scheme for the fully coupled mLdG SmA system.
    %and analyze a novel generalized scalar auxiliary variable exponential integrator (GSAV-EI) scheme for system \eqref{eq1_9}, which is linear, decoupled, structure-preserving, and unconditionally energy-stable.
    \item  We reformulate the exponential integrator into an equivalent backward-Euler-type form. This removes the severe time-step restriction in the previous GSAV-EI analysis and makes a rigorous fully discrete error estimate possible.
    \item We prove unconditional modified-energy stability of the proposed scheme.
    \item We derive optimal fully discrete error estimates by combining the reformulated GSAV-EI structure with a bootstrap argument and higher-order regularity estimates for the numerical solution.
    % \item {\color{red}We introduce an exponential relaxation strategy, inspired by relaxed SAV(RSAV) \cite{jiang2022improving}  and energy-optimization SAV(EOP-SAV) \cite{liu2024novel,zhang2025novel}  ideas, to improve the consistency between the modified auxiliary energy and the originaal physical energy in long-time simulations.}
    %By integrating the RSAV \cite{jiang2022improving} and EOP-SAV \cite{liu2024novel,zhang2025novel} concepts, we propose a novel exponential relaxation technique, which better preserves the physical consistency of the system.
\end{itemize}

The rest of this paper is organized as follows.
In Section \ref{se2}, we
construct the fully discrete GSAV-EI finite difference scheme for the SmA model and  prove its unconditional modified-energy stability.
In Section \ref{se3}, we establish optimal fully discrete error estimates by a bootstrap argument based on  higher-order regularity bounds for the numerical solution.
% {\color{red}Section 4 presents a relaxation strategy that improve the long-time consistency between the modified and original energies.}
Section \ref{section6} reports numerical experiments that verify the theoretical results  and demonstrate the effectiveness of the proposed method in simulating  SmA self-assembly dynamics. We conclude with some
perspectives in Section \ref{section7}.
\section{Fully discrete GSAV-EI scheme and energy stability analysis}\label{se2}
%In this section, we construct the fully discrete GSAV-EI finite difference scheme for the SmA model \eqref{eq1_9} and present the reformulation strategy of the original GSAV-EI scheme.
First, we give some notation.
The space $\mathcal{S}^{(d)}$ is defined as the set of $d \times d$ symmetric and traceless matrices:
\begin{equation}
    \mathcal{S}^{(d)} \overset{\mathrm{def}}{=} \left\{ \bX \in \mathbb{R}^{d \times d} \;\middle|\; \operatorname{tr}(\bX) = 0, \bX^{ij} = \bX^{ji} \in \mathbb{R} \ \forall i,j = 1, \dots, d \right\}.
\end{equation}
%where $\operatorname{tr}(\bX) := \sum_{i=1}^d \bX^{ii}$.
The admissible space for  $\mathbf{Q}$ is denoted by $\mathbf{H}^1(\Omega, \mathcal{S}^{(d)})$, defined as
\begin{equation*}
    \mathbf{H}^1(\Omega, \mathcal{S}^{(d)}) := \left\{ \mathbf{Q}: \Omega \to \mathcal{S}^{(d)} \mid \mathbf{Q}_{ij} \in H^1(\Omega) \text{ for } i,j = 1, \dots, d \right\}.
\end{equation*}
For $\mathbf{Q}\in \mathbf{H}^1(\Omega,\mathcal{S}^{(d)})$, we define its standard Frobenius-type
$L^2$ norm  by
\begin{equation*}
    \|\mathbf{Q}\|_{L^2}^2
    :=
    \int_{\Omega} |\mathbf{Q}|_F^2\,d\mathbf{x},\quad
    \|\nabla \mathbf{Q}\|_{L^2}^2
    :=
    \int_{\Omega} |\nabla \mathbf{Q}|^2\,d\mathbf{x}.
\end{equation*}
The admissible space for $u$ is $H^2(\Omega)$.
%, which is defined as:
%\begin{equation*}
%    H^2(\Omega) := \left\{ u: \Omega \to \mathbb{R} \mid u \in L^2(\Omega), \nabla u \in (L^2(\Omega))^d, D^2 u \in (L^2(\Omega))^{d\times d} \right\}.
%\end{equation*}

 The $L^2$ gradient flow equations \eqref{1.3a} directly yield the energy dissipation law of the coupled system:
\begin{equation}\label{eq_dissipation}
    \frac{dE}{dt} = - \int_{\Omega} \left| \frac{\delta E}{\delta \mathbf{Q}} - \lambda \mathbf{I}_d - \nu + \nu^{T} \right|_F^2 d\mathbf{x} - \int_{\Omega} \left| \frac{\delta E}{\delta u} \right|^2 d\mathbf{x} \le 0.
\end{equation}
As a consequence of \eqref{eq_dissipation}, we obtain a priori  estimates for system \eqref{eq1_9}.
\begin{lemma}[A priori estimates]\label{lem:continuous_regularity}
  Assume that the initial data
$\mathbf{Q}_0\in\mathbf{H}^1(\Omega,\mathcal{S}^{(d)})$ and
$u_0\in H^2(\Omega)$ satisfy
$E(\mathbf{Q}_0,u_0)<+\infty$. Then, by the energy dissipation law \eqref{eq_dissipation}  and
the coercivity of $E$, the nonlinear energy functional
$E_1[\mathbf{Q},u]$ remains uniformly bounded on $[0,T]$; namely,
    \begin{align*}
        % \mathbf{Q}& \in L^\infty(0, T; \mathbf{H}^1(\Omega, \mathcal{S}^{(d)})), \quad u \in L^\infty(0, T; H^2(\Omega)),\\
          -C_*
    &\le
    E_1[\mathbf{Q},u]
    \coloneqq
    \int_{\Omega}
    2B_0(D^2u:q^2\mathbf{M}u)
    \,d\mathbf{x}
    +
    B_0
    \left\|
    q^2\mathbf{M}u
    \right\|_{L^2}^2
    \\&\qquad\qquad
    +
    \int_{\Omega}
    f_{\mathrm{bn}}(\mathbf{Q})
    +
    f_{\mathrm{bs}}(u)
    -
    \frac \kappa2(|\mathbf{Q}|_F^2+|u|^2)
    d\mathbf{x}
    \le
    C^* ,
    \end{align*}
 where $f_{\mathrm{bn}}(\mathbf{Q})=\frac{\A}{2}\operatorname{tr}(\mathbf{Q}^2)-\frac{\B}{3}\operatorname{tr}(\mathbf{Q}^3)+\frac{\C}{4}(\operatorname{tr}(\mathbf{Q}^2))^2$,
 $\kappa>0$ is a stabilization parameter, and $C_*,C^*>0$ are
constants depending only on the initial energy
$E(\mathbf{Q}_0,u_0)$, the domain $\Omega$, and the model parameters.
\end{lemma}
\begin{proof}
    % By the energy dissipation law \eqref{eq_dissipation}, we have
    % \ba \label{eqa1}
    % E[\mathbf{Q}(t), u(t)] \le E[\mathbf{Q}_0, u_0] , \quad \forall t > 0.
    % \ed
    As established in \cite{zhao2017novel} and implied by the inherent
coercivity of $f_{\mathrm{bn}}(\mathbf{Q})$ and $f_{\mathrm{bs}}(u)$,
% both bulk
% potentials
% $
% f_{\mathrm{bn}}(\mathbf{Q})
% =
% \frac{\A}{2}\operatorname{tr}(\mathbf{Q}^2)
% -
% \frac{\B}{3}\operatorname{tr}(\mathbf{Q}^3)
% +
% \frac{\C}{4}\bigl(\operatorname{tr}(\mathbf{Q}^2)\bigr)^2
% $
% and $f_{\mathrm{bs}}(u)$ are bounded from below.
% ince the leading quartic terms in
% $f_{\rm bn}$ and $f_{\rm bs}$ have positive coefficients,
there exists a
constant $C_f>0$ such that
\ba\label{eqc6}
f_{\mathrm{bn}}(\mathbf{Q})
\ge \frac \C 8 |\mathbf{Q}|_F^4 -C_f,
\quad
f_{\mathrm{bs}}(u)\ge \frac c 8 u^4-C_f.
\ed
Thus, by the energy dissipation law \eqref{eq_dissipation} and \eqref{eqc6}, we have
\begin{equation*}
    \frac{K}{2}\|\nabla \mathbf{Q}\|_{L^2}^2
    +\frac \C 8 \|\mathbf{Q}\|_{L^4}^4+
    B_0
    \left\|
    D^2 u + q^2\mathbf{M}u
    \right\|_{L^2}^2
    +
    \frac{c}{8}\|u\|_{L^4}^4
    \le
    E[\mathbf{Q}_0,u_0]
    +
    C_f|\Omega| ,
\end{equation*}
which implies
$\mathbf{Q}\in L^\infty(0,T;\mathbf{H}^1(\Omega,\mathcal{S}^{(d)}))$ and
 $u\in L^\infty(0,T;L^4(\Omega))$.
 Then we obtain
\begin{align*}
    B_0\|D^2u\|_{L^2}^2
    &\le
    2
    B_0\left\|
    D^2u+q^2\mathbf{M}u
    \right\|_{L^2}^2
    +
    2 B_0
    \left\|
    q^2\mathbf{M}u
    \right\|_{L^2}^2 \\& \leq
    2
    (E[\mathbf{Q}_0,u_0]
    +
    C_f|\Omega|)+2B_0 q^4
    \left\|
    \mathbf{M}
    \right\|_{L^4}^2
     \left\|
   u
    \right\|_{L^4}^2
    \\& \leq C_E ,
\end{align*}
which implies
$u\in L^\infty(0,T;H^2(\Omega))$.
% The proof is complete.
Consequently, we have the following
 bound for the coupling term:
\begin{align*}
    -C_E\le\int_{\Omega}
    2B_0(D^2u:q^2\mathbf{M}u)
    \,d\mathbf{x}&+B_0
    \left\|
    q^2\mathbf{M}u
    \right\|_{L^2}^2
    \\=
   B_0\left\|
    D^2u+q^2\mathbf{M}u
    \right\|_{L^2}^2
    &-
    B_0
    \|D^2u\|_{L^2}^2 \le E(\mathbf{Q}_0,u_0)
    +
    C_f|\Omega| .
\end{align*}
Since the leading quartic terms in
$f_{\rm bn}$ and $f_{\rm bs}$ have positive coefficients, there exists a
constant $C_\kappa>0$ such that
\[
f_{\rm bn}(\mathbf Q)-\frac{\kappa}{2}|\mathbf Q|_F^2\ge -C_\kappa,
\qquad
f_{\rm bs}(u)-\frac{\kappa}{2}u^2\ge -C_\kappa .
\]
Along the exact solution trajectory, there exist positive constants $C_*$ and $C^*$ such that
\begin{equation*}
\begin{aligned}
    -C_*
    \le&
    E_1[\mathbf{Q},u]
    % \coloneqq
    % \int_{\Omega}
    % 2B_0(D^2u:q^2\mathbf{M}u)
    % \,d\mathbf{x}
    % +
    % B_0
    % \left\|
    % q^2\mathbf{M}u
    % \right\|_{L^2}^2
    % \\&\qquad\qquad
    % +
    % \int_{\Omega}
    % f_{\mathrm{bn}}(\mathbf{Q})
    % +
    % f_{\mathrm{bs}}(u)
    % -
    % \frac \kappa2(|\mathbf{Q}|_F^2+|u|^2)
    % d\mathbf{x}
    \le
    C^* .
\end{aligned}
\end{equation*}
\end{proof}
Let $\zeta\in C^1(\mathbb R;(0,\infty))$ satisfy $\zeta'\ge0$ on
$\mathbb R$. By introducing the scalar auxiliary variable
$s(t)=E_1[\mathbf{Q}(t),u(t)]$, we define the modified energy
$\mathcal E$ and the scalar factor $g$ by
\begin{equation*}
    \mathcal{E}[\mathbf{Q},u,s]
    :=
    \frac{K}{2}\|\nabla\mathbf{Q}\|_{L^2}^2
    +
    B_0\|\Delta u\|_{L^2}^2
    +
    \frac{\kappa}{2}
    (\|u\|_{L^2}^2+\|\mathbf{Q}\|_{L^2}^2)
    +
    s,
\end{equation*}
and
$g(\mathbf{Q},u,s)
    :=
    \zeta(s)/
    \zeta(E_1[\mathbf{Q},u])$.
Here $\|\Delta u\|_{L^2}^2=\|D^2u\|_{L^2}^2$ under periodic boundary
conditions.
Then the original system \eqref{eq1_9} can be equivalently reformulated
into the following coupled system for the unknowns $(\mathbf{Q},u,s)$:
\begin{align}
\label{eq_sys_Q}
\frac{\partial \mathbf{Q}}{\partial t}
&=
K\Delta\mathbf{Q}-\kappa\mathbf{Q}
-
g(\mathbf{Q},u,s)\mathbf{H}(\mathbf{Q},u),
\\
\frac{\partial u}{\partial t}
&=
-2B_0\Delta^2u-\kappa u
-
g(\mathbf{Q},u,s)\mu(\mathbf{Q},u),
\\
\frac{d s}{dt}
&=
g
\left(
\int_{\Omega}
\mathbf{H}(\mathbf{Q},u):\frac{\partial\mathbf{Q}}{\partial t}
\,d\mathbf{x}
+
\int_{\Omega}
\mu(\mathbf{Q},u)\frac{\partial u}{\partial t}
\,d\mathbf{x}
\right),
\end{align}
where $\mathbf H=\mathbf H_{\rm p}(\mathbf Q,u)+\mathbf H_{\rm d}( u)$ and $\mu=\mu_{\rm p}(\mathbf Q,u)+\mu_{\rm d}(\mathbf Q,u)$ with the constituent terms defined by
\begin{small}
\ba \label{eqd1}
\mathbf H_{\rm p}
={}&
(\A-\kappa)\mathbf Q
-
\B
\left(
\mathbf Q^2
-
\frac{\operatorname{tr}(\mathbf Q^2)}{d}\mathbf I_d
\right)
+
\C\operatorname{tr}(\mathbf Q^2)\mathbf Q
+
2B_0q^4
\frac{\mathbf Q}{s_+^2}
u^2,
\\
\mathbf H_{\rm d}
={}&
\frac{2B_0q^2}{s_+}
\left(
uD^2u
-
\frac{\operatorname{tr}(uD^2u)}{d}\mathbf I_d
\right),
\mu_{\rm p}
={}
(a-\kappa)u
+
bu^2
+
cu^3
+
2B_0
|q^2\mathbf M|_F^2u,
\\
\mu_{\rm d}
={}&
2B_0D^2u:(q^2\mathbf M)
+
2B_0q^2
\nabla\cdot
\left(
\nabla\cdot(\mathbf M u)
\right).
\ed
\end{small}
% \begin{remark}
% The function $\zeta$ can be chosen flexibly within the GSAV framework.
% Typical admissible choices include
% \[
% \zeta(x)=e^x,\qquad
% \zeta(x)=1+\tanh x,\qquad
% \zeta(x)=\frac{\pi}{2}+\arctan x.
% \]
% \end{remark}
\subsection{Spatial Discretization and Discrete Function Spaces}
% The details of the spatial discretization are provided in Supplementary
% Material~\ref{se1}.
To simplify the notation, we consider the problem \eqref{eq1_9} in the domain $\Omega = [0, L_d]^3$. Given a positive integer $J$, the uniform mesh size for each spatial direction is set to be $h = L_d/J$. For a regular grid, all variables are stored at the primary mesh points. We denote by $\mathbf{E}$ the set of mesh points, defined by
\begin{equation*}
    \mathbf{E} = \left\{ (x_p, y_q, z_r) = (ph, qh, rh) \mid p, q, r = 0, 1, \dots, J \right\}.
\end{equation*}
The corresponding periodic grid function space is given by
\begin{equation*}
    E_h^{\text{per}} \!= \{ U : \mathbf{E} \to \mathbb{R} \mid U_{0,q,r} \!= U_{J,q,r}, \, U_{p,0,r} = U_{p,J,r}, \, U_{p,q,0} \!= U_{p,q,J}, 0 \le p, q, r \le J \}.
\end{equation*}
For ease of notation, we handle the periodic boundary conditions by assigning the values of the corresponding interior nodes to the ghost nodes. Specifically, for any $U \in E_h^{\text{per}}$, we have
\begin{equation*}
    U_{-1,q,r} = U_{J-1,q,r}, \quad U_{J+1,q,r} = U_{1,q,r}, \quad \forall 0 \le q, r \le J.
\end{equation*}
Analogous periodic extensions apply to the $q$ and $r$ directions.

Then, we define the forward, backward, and central difference operators for the grid function $U \in E_h^{\text{per}}$ along the $x$-direction as follows:
% \begin{small}
    \begin{align*}
    (D_1^+ U)_{p,q,r}\! =\! \frac{U_{p+1,q,r} \!-\! U_{p,q,r}}{h},
    (D_1^- U)_{p,q,r} \!= \!\frac{U_{p,q,r} \!- \!U_{p-1,q,r}}{h},
    D_1^c U\!=\!\frac{1}{2}(D_1^+ U \!+\! D_1^- U\!).
\end{align*}
% \end{small}
where the difference operators $D_2^{\pm,c}$ and $D_3^{\pm,c}$ along the $y$- and $z$-directions are defined analogously. Moreover, we denote the mixed central difference as $D_{k,l}^c U = D_k^c D_l^c U$ for $k,l \in \{1,2,3\}$.

In a collocated grid, to maintain a compact stencil for the Laplacian, we define the discrete gradient operator $\nabla_h : E_h^{\text{per}} \to (E_h^{\text{per}})^3$ using forward differences:
\begin{equation*}
    (\nabla_h U)_{p,q,r} = ( (D_1^+ U)_{p,q,r}, (D_2^+ U)_{p,q,r}, (D_3^+ U)_{p,q,r} )^T,
\end{equation*}
and the corresponding discrete divergence operator $\nabla_h \cdot : (E_h^{\text{per}})^3 \to E_h^{\text{per}}$ using backward differences:
\begin{equation*}
    \nabla_h \cdot (U^{(1)}, U^{(2)}, U^{(3)})^T = D_1^- U^{(1)} + D_2^- U^{(2)} + D_3^- U^{(3)}.
\end{equation*}
In addition, for any scalar-valued grid function
$U\in E_h^{\mathrm{per}}$, we define its discrete Hessian as the
matrix-valued grid function
\begin{equation*}
    D_h^2U
    :=
    \left(D_{i,j}^cU\right)_{i,j=1}^3,
    \qquad
    D_{i,j}^cU:=D_i^cD_j^cU.
\end{equation*}
Because all variables are collocated, the discrete inner products no longer require averaging operators. They are simply given by:
\begin{align*}
    \langle U, V \rangle_h &= h^3 \sum_{p,q,r=1}^{J} U_{p,q,r} V_{p,q,r} \quad \forall U, V \in E_h^{\text{per}},\\
    [\mathbf{U}, \mathbf{V}]_h &= \sum_{k=1}^3 \langle U^{(k)}, V^{(k)} \rangle_h \quad \forall \mathbf{U}, \mathbf{V} \in (E_h^{\text{per}})^3.
\end{align*}
Then, for any $U \in E_h^{\text{per}}$, the corresponding discrete
$L^2$, $L^\infty$, $H^1$, and $H^2$ norms are defined by
% \begin{small}
    \begin{align*}
    \|U\|_h^2       & = \langle U, U \rangle_h, \quad
    \|\nabla_h U\|_h^2 = [\nabla_h U, \nabla_h U]_h ,\quad
        \|U\|_{H_h^1}^2  = \|U\|_h^2 + \|\nabla_h U\|_h^2, \\
    % = \sum_{k=1}^3 \langle D_k^+ U, D_k^+ U \rangle_h, \\
    \|U\|_{H_h^2}^2 &= \!\|U\|_{H_h^1}^2 + \|\Delta_h U\|_h^2,
    \|U\|_{H_h^4}^2 \!= \!\|U\|_{H_h^2}^2 + \|\Delta_h^2 U\|_h^2,
    \|U\|_\infty \!= \!\max_{0 \le p,q,r \le J} |U_{p,q,r}|.
\end{align*}
% \end{small}
The discrete Laplacian is defined as $\Delta_h U = \sum_{k=1}^3 D_k^+ D_k^- U$, and the discrete biharmonic operator is defined as $\Delta_h^2 U = \Delta_h (\Delta_h U)$.
Using summation-by-parts on the collocated grid,  for any $U, V \in E_h^{\text{per}}$, we have
\begin{equation*}
    \begin{aligned}
        \langle D_k^- D_k^+ U, V \rangle_h & = - \langle D_k^+ U, D_k^+ V \rangle_h = \langle U, D_k^- D_k^+ V \rangle_h,                                      \\
        \langle D_{k,l}^c U, V \rangle_h   & = - \langle D_l^c U, D_k^c V \rangle_h = - \langle D_k^c U, D_l^c V \rangle_h = \langle U, D_{k,l}^c V \rangle_h,
    \end{aligned}
\end{equation*}
which also implies that
\begin{equation}\label{eq_c1}
    \langle -\Delta_h U, V \rangle_h =  [\nabla_h U, \nabla_h V]_h = \langle U, -\Delta_h V \rangle_h,\quad
    \langle \Delta_h^2 U, V \rangle_h = \langle U, \Delta_h^2 V \rangle_h.
\end{equation}

For tensor-valued grid functions, we first introduce the general
matrix-valued grid function space
\begin{equation*}
    \mathbb{M}_h^{(3)}
    :=
    \left\{
    \bm{\Phi}_h=(\Phi_h^{ij})_{i,j=1}^3
    \ \middle|\
    \Phi_h^{ij}\in E_h^{\mathrm{per}},
    \quad i,j=1,2,3
    \right\}.
\end{equation*}
The discrete function space for symmetric and traceless
$\mathbf{Q}$-tensor fields is then defined as
\begin{equation*}
    \mathcal{S}_h^{(3)}
    :=
    \left\{
    \bm{\Phi}_h\in\mathbb{M}_h^{(3)}
    \ \middle|\
    \bm{\Phi}_h^{T}=\bm{\Phi}_h,\quad
    \operatorname{tr}(\bm{\Phi}_h)=0
    \right\}.
\end{equation*}
The discrete Frobenius product  and the discrete $L^4$ norm for any $\bm{\Phi}_h, \bm{\Psi}_h \in  \mathbb{M}_h^{(3)}$ are defined as follows:
\begin{equation*}
    \langle \bm{\Phi}_h, \bm{\Psi}_h \rangle_h =  \sum_{i,j=1}^3 \langle {\Phi}_h^{i,j}, {\Psi}_h^{i,j} \rangle_h,\quad  \|\bm{\Phi}_h\|_{L_h^4}^4 = \big\langle |\bm{\Phi}_h|_F^2, |\bm{\Phi}_h|_F^2  \big\rangle_h.
\end{equation*}
For any $\bm{\Phi}_h \in  \mathbb{M}_h^{(3)}$, the discrete $L^2$,
$L^\infty$, $H^1$, and $H^2$ norms are defined by
\begin{equation*}
\begin{aligned}
\|\bm{\Phi}_h\|_h^2 &:= {\sum_{i,j=1}^3 \| {\Phi}_h^{i,j}\|_h^2}, \quad
\|\nabla_h \bm{\Phi}_h\|_h^2 := {\sum_{i,j=1}^3 \| \nabla_h{\Phi}_h^{i,j}\|_h^2},\\ \|\Delta_h \bm{\Phi}_h\|_h^2 &:= {\sum_{i,j=1}^3 \|\Delta_h {\Phi}_h^{i,j}\|_h^2},\quad
\|\bm{\Phi}_h\|_\infty := \max_{0 \le p,q,r \le J} |(\bm{\Phi}_{h})_{p,q,r}|_F    , \\
\|\bm{\Phi}_h\|_{H_h^1}^2 &:= {\|\bm{\Phi}_h\|_h^2 + \|\nabla_h \bm{\Phi}_h\|_h^2}
, \quad
\|\bm{\Phi}_h\|_{H_h^2}^2 := {\|\bm{\Phi}_h\|_{H_h^1}^2 + \|\Delta_h \bm{\Phi}_h\|_h^2}.
\end{aligned}
\end{equation*}

For any
$\bm{\Phi}_h\in\mathbb{M}_h^{(3)}$, we define the discrete central
double-divergence operator directly by
\begin{equation*}
    \operatorname{div}_{h,c}^{\,2}\bm{\Phi}_h
    :=
    \sum_{i,j=1}^3D_i^cD_j^c\Phi_h^{ij}.
\end{equation*}
By the periodic summation-by-parts formula, the discrete Hessian and
the discrete double-divergence operator satisfy
\begin{equation}\label{eq_discrete_hessian_adjoint}
    \left\langle D_h^2U,\bm{\Phi}_h\right\rangle_h
    =
    \left\langle
    U,\operatorname{div}_{h,c}^{\,2}\bm{\Phi}_h
    \right\rangle_h,
    \qquad
    U\in E_h^{\mathrm{per}},
    \quad
    \bm{\Phi}_h\in\mathbb{M}_h^{(3)}.
\end{equation}
Moreover, by the discrete Fourier transform and Parseval's identity,
we have
\ba\label{eq_discrete_hessian_bound}
    \|D_h^2U\|_h
    &\le
    \|\Delta_hU\|_h,
    \qquad U\in E_h^{\mathrm{per}}.
\ed

Finally, the discrete Sobolev spaces for $\mathbf{Q}_h$ and $u_h$  are defined as
\begin{equation*}
\begin{aligned}
\mathbf{H}_h^2(\Omega, \mathbb{M}_h^{(3)}) &= \{\bm{\Phi}_h \in \mathbb{M}_h^{(3)} \mid \|\bm{\Phi}_h\|_{H_h^2} < \infty\}, \quad \mathbf{H}_h^2(\Omega,\mathcal{S}_h^{(3)})
:=
\mathbf{H}_h^2(\Omega,\mathbb{M}_h^{(3)})
\cap\mathcal{S}_h^{(3)}.\\
H_h^4(\Omega) &= \{U \in E_h^{\text{per}} \mid \|U\|_{H_h^4} < \infty\}.
\end{aligned}
\end{equation*}

The two-dimensional discretization is obtained analogously by
omitting the $z$-direction and restricting all spatial and tensor
indices to $\{1,2\}$, with the discrete operators and norms understood
accordingly. Hereafter, $\mathbb{M}_h^{(d)}$ and
$\mathcal{S}_h^{(d)}$ denote the corresponding spaces for
$d\in\{2,3\}$.
\subsection{Fully discrete GSAV-EI scheme}
Let $\tau > 0$ be the uniform time step size and $t_n = n\tau$ for $n \ge 0$.
Let $\mathbf{Q}_h^n$, $u_h^n$, and $s^n$ denote the fully discrete numerical approximations of the exact solutions $\mathbf{Q}(t_n)$, $u(t_n)$, and $s(t_n)$, respectively.
The discrete approximations for $E_1(t_n)$ and $g(t_n)$, denoted by $E_{1h}^n$ and $g^n$ respectively, are given by
\begin{equation}\label{eq_s_ex_n}
    \begin{aligned}
        E_{1h}^n
&:=
E_{1h}[\mathbf{Q}_h^n,u_h^n]
:=
         2 B_0 q^2 \langle D_h^2 u_h^n , \bM_h^n u_h^n \rangle_h + B_0 q^4 \| \bM_h^n u_h^n \|_h^2 \\& + \langle f_{\mathrm{bn}}(\mathbf{Q}_h^n), 1 \rangle_h+ \langle f_{\mathrm{bs}}(u_h^n), 1 \rangle_h-\frac \kappa2(\|u_h^n\|_h^2+\|\mathbf{Q}_h^n\|_h^2),\\
         g^n &=  \zeta(s^n)/\zeta(E_{1h}^n),
    \end{aligned}
\end{equation}
where $\bM_h^n = \frac{\mathbf{Q}_h^n}{s_+} + \frac{\mathbf{I}_d}{d}$ is the discrete evaluation of  $\bM$.
% and $g^n= e^{s^n}/e^{E_{1h}^n}$ is the discrete evaluation of the exponential stabilization factor $g$ at time $t_n$.

Next, we define the following $\kappa$-stabilized discrete operators $\mLQ: \mathcal{S}_h^{(d)} \to \mathcal{S}_h^{(d)}$ and $\mLu: E_h^{\text{per}} \to E_h^{\text{per}}$ as follows:
\begin{align} \label{eq_L}
    \mLQ\mathbf{Q}_h^n = -K\Delta_h \mathbf{Q}_h^n +  \kappa \mathbf{Q}_h^n,\quad
    \mLu u_h^n = 2B_0\Delta_h^2 u_h^n +  \kappa  u_h^n.
\end{align}
The  nonlinear operators $\mNQ: \mathcal{S}_h^{(d)} \times E_h^{\text{per}} \times \mathbb{R} \to \mathcal{S}_h^{(d)}$ and $\mNu: \mathcal{S}_h^{(d)} \times E_h^{\text{per}} \times \mathbb{R} \to E_h^{\text{per}}$ are defined as follows:
\ba \label{eq_N_Q_def}
\mathbf{N}_{h}^n = \mathbf{N}_{h}(\mathbf{Q}_h^n, u_h^n, s^n) = -g^n \muQ^n , \quad
\mathcal{N}_{h}^n = \mathcal{N}_{h}(\mathbf{Q}_h^n, u_h^n, s^n) = -g^n \mu_{h}^n ,
\ed
where $\muQ^n$ and $\mu_{h}^n$ are defined as
{
\ba\label{eqc8}
\muQ^n
&=
\mathbf H_{\rm p}(\mathbf Q_h^n,u_h^n)
+
\mathbf H_{{\rm d},h}(u_h^n),~
\mathbf H_{{\rm d},h}(u_h^n)
:=
\frac{2B_0q^2}{s_+}
\left(
u_h^nD_h^2u_h^n
-
\frac{\operatorname{tr}(u_h^nD_h^2u_h^n)}{d}\mathbf I_d
\right),
\\
\mu_h^n
&=
\mu_{\rm p}(\mathbf Q_h^n,u_h^n)
+
\mu_{{\rm d},h}^n,~
\mu_{{\rm d},h}^n
:=
2B_0q^2\left(
\mathbf M_h^n:D_h^2u_h^n
+
\operatorname{div}_{h,c}^{\,2}
\left(\mathbf M_h^n u_h^n\right)
\right).
\ed}

Let $ \mathcal{I}_h$ denote the standard grid restriction  operator for continuous fields. We initialize the numerical solution by
\begin{equation*}
    \mathbf{Q}_h^0=\mathcal{I}_h\mathbf{Q}(0),
    \qquad
    u_h^0=\mathcal{I}_h u(0),
    \qquad
    s^0=E_{1h}^0=E_{1h}[\mathbf{Q}_h^0,u_h^0].
\end{equation*}
so that $g^0=1$.
Following  \cite{ju2022generalized}, we construct a first-order fully discrete GSAV-EI scheme for system \eqref{eq_sys_Q}  as follows: given the initial data $(\bQ^0, u_h^0, s^0)$, for each time step $n \ge 0$, we update the numerical solution to the next time level $t_{n+1}$ by
\begin{align}
    \bQ^{n+1} & = e^{-\tau \mLQ} \bQ^{n} + \tau \varphi(\tau \mLQ) \mNQ^n, \label{eq4_6}                                                                                       \\
    \bu^{n+1} & = e^{-\tau \mLu} \bu^{n} + \tau \varphi(\tau \mLu) \mNu^n,\label{eq4_7}                                                                                        \\
    {s}^{n+1} & = s^n -  \left\langle \mNQ^n, \bQ^{n+1} - \bQ^{n} \right\rangle_h -\left\langle \mNu^n, \bu^{n+1} - \bu^{n} \right\rangle_h ,\label{eq4_8}
\end{align}
where $\varphi(z) = (1-e^{-z})/z$.

To address the severe time-step restriction required by the original GSAV-EI scheme \cite{ju2022generalized}, we  reformulate \eqref{eq4_6}-\eqref{eq4_8} into a  quasi-implicit backward Euler-type structure that eliminates this restriction and admits a rigorous fully discrete error analysis.
Applying the operator $\frac{1}{\tau}(\mq(\tau \mLQ)+\tau \mLQ)$ to both sides of \eqref{eq4_6}, where $\mq(z) = z / (e^z - 1)$, we have
\begin{align}\label{eq4_11}
    \frac{1}{\tau}(\mq(\tau \mLQ)+\tau \mLQ) (\bQ^{n+1} - e^{-\tau \mLQ} \bQ^{n}) = & \frac{1}{\tau}(\mq(\tau \mLQ)+\tau \mLQ) \tau \varphi(\tau \mLQ) \mNQ^n.
\end{align}
Noting that $\mq(\tau \mLQ)+\tau \mLQ=  \frac{\tau \mLQ}{e^{\tau \mLQ} - 1} + \tau \mLQ =  \frac{\tau \mLQ e^{\tau \mLQ}}{e^{\tau \mLQ} - 1} $,
the identity \eqref{eq4_11} can be rewritten as
\ba
\mq(\tau \mLQ) \frac{\bQ^{n+1} - \bQ^n}{\tau} + \mLQ \bQ^{n+1} = \mNQ^n.
\ed
% Similarly, applying the operator $\frac{1}{\tau}(\mq(\tau \mLu)+\tau \mLu)$ to both sides of \eqref{eq4_7}, we have
% \ba
% \mq(\tau \mLu) \frac{\bu^{n+1} - \bu^{n}}{\tau} + \mLu \bu^{n+1} = \mNu^n.
% \ed
Taking the same approach for \eqref{eq4_7},  the   GSAV-EI scheme for  system \eqref{eq_sys_Q} reads as follows: given the initial data $(\bQ^0, u_h^0, s^0)$, for $n \ge 0$, the numerical solution $(\bQ^{n+1}, u_h^{n+1}, s^{n+1}) $  is recursively updated via
\begin{align}
   & \mq(\tau \mLQ)\dtau \bQ^{n+1} + \mLQ \bQ^{n+1} = \mNQ^n, \label{eq_scheme_Q}  \\
    &\mq(\tau \mLu) \dtau \bu^{n+1} + \mLu \bu^{n+1} = \mNu^n, \label{eq_scheme_u} \\
    &\ts^{n+1} = s^n -  \left\langle \mNQ^n, \delta \bQ^{n+1} \right\rangle_h - \left\langle \mNu^n, \delta \bu^{n+1} \right\rangle_h .\label{eq_scheme_s}
\end{align}
where  the notation $\delta \phi^{n+1} = \phi^{n+1} - \phi^n$ and $\dtau \phi^{n+1} = \frac{\delta \phi^{n+1}}{\tau}$ are used for any $\phi \in \{\bQ, \bu,\ts\}$.
\iffalse
\begin{remark}
    The standard implicit backward Euler scheme is renowned for its well-established and elegant theoretical frameworks, particularly concerning the construction of appropriate test functions for error estimates and the rigorous derivation of energy dissipation laws. By employing the aforementioned system reformulation, our proposed scheme seamlessly inherits these analytical advantages. The profound mathematical benefits of this transformation will become persistently evident throughout the subsequent rigorous proofs.
\end{remark}
\fi
% To avoid excessive deviation of $s^{n+1}$ from the discrete physical energy $E_{1h}^{n+1}$, we introduce a relaxation step to update $s^{n+1}$ as a convex combination. Specifically, this relaxation step is designed by combining the efficient evaluation criterion of the EOP-SAV method and the robust update strategy of the R-GSAV scheme.
\subsection{Unconditional energy stability analysis}
%In this section, we establish the unconditional energy stability of the GSAV-EI scheme \eqref{eq_scheme_Q}-\eqref{eq_scheme_s}.
For ease of presentation, we introduce some notation for the modified discrete energy norms.
Utilizing the fact that  $\mLQ$ and $ \mLu$ are symmetric positive semidefinite,   we define the following inner products and the corresponding energy norms:
\begin{equation*}
\begin{aligned}
\ilQ{ \cdot }  := \langle \mathcal{L}_h \cdot,\cdot  \rangle_h= K  \norm{\nabla_h \cdot }_h^2 +  \kappa  \norm{\cdot}_h^2, \quad \ilu{\cdot} :=\langle \mathcal{D}_h \cdot,\cdot  \rangle_h= 2B_0  \norm{\Delta_h \cdot }_h^2 + \kappa \norm{\cdot}_h^2.
\end{aligned}
\end{equation*}
Furthermore,  we construct the  weighted energy norms  for any $\bm{\Phi}_h \in \mathbf{H}_h^2(\Omega,  \mathcal{S}_h^{(d)})$:
    \begin{equation*}
\begin{aligned}
\iaQ{\bm{\Phi}_h} &:= \langle \mathcal{Q}(\tau \mathcal{L}_h) \bm{\Phi}_h, \bm{\Phi}_h \rangle_h,\quad
\ibQ{\bm{\Phi}_h} := \langle \mathcal{Q}_1(\tau \mathcal{L}_h) \bm{\Phi}_h, \bm{\Phi}_h \rangle_h ,\\
\icQ{\bm{\Phi}_h} &:=  \langle \mathcal{Q}(\tau \mathcal{L}_h)  \mathcal{L}_h   \bm{\Phi}_h, \bm{\Phi}_h \rangle_h,
% \quad\idQ{\mathbf{U}} :=  \langle \mathcal{Q}_1(\tau \mathcal{L}_h)  \mathcal{L}_h   \mathbf{U}, \mathbf{U} \rangle_h,
\end{aligned}
\end{equation*}
where the  function $\mathcal{Q}_1$ is defined as
\begin{equation}\label{eq_inner_Q}
    \mathcal{Q}_1(z) =\mathcal{Q}(z) + \frac{z}{2} = \frac{z}{e^z - 1}+ \frac{z}{2}= \frac{z}{2}\text{coth}\left( \frac{z}{2} \right)
    \geq 1, \quad   z\geq 0.
\end{equation}
Analogous norms of $\mathcal{D}_h$ are defined by  replacing $\mathcal{L}_h$ with $\mathcal{D}_h$ in the expressions above.
\begin{lemma}\label{lem_norm_equivalence}
    For any  $\bm{\Phi}_h \in \mathbf{H}_h^2(\Omega, \mathcal{S}_h^{(d)})$ and $v_h \in H_h^2(\Omega)$, we have
    \ba\label{eq_norm_bounds_1}
    &\|\mathcal{Q}(\tau \mathcal{L}_h) \bm{\Phi}_h\|_h \le \|\bm{\Phi}_h\|_{\mathcal{Q}_\mathcal{L}} \le \|\bm{\Phi}_h\|_h \le \|\bm{\Phi}_h\|_{\mathcal{Q}_\mathcal{L}^1},\quad\|\bm{\Phi}_h\|_{\mathcal{Q}_{\mathcal L}^2}
\le
\|\bm{\Phi}_h\|_{\mathcal L_h},\\
    &\|\mathcal{Q}(\tau \mathcal{D}_h) v_h\|_h \le \|v_h\|_{\mathcal{Q}_\mathcal{D}} \le \|v_h\|_h \le \|v_h\|_{\mathcal{Q}_\mathcal{D}^1}.
%     \\
%      & \|\mathbf{U}\|_{\mathcal{Q}_{\mathcal L}}
% \le
% \|\mathbf{U}\|_{\mathcal L_h}
% \le
% \|\mathbf{U}\|_{\mathcal{Q}_{\mathcal L}^{3}},
% \quad
% \|v\|_{\mathcal{Q}_{\mathcal D}}
% \le
% \|v\|_{\mathcal D_h}
% \le
% \|v\|_{\mathcal{Q}_{\mathcal D}^{3}} .
    \ed
\end{lemma}
\begin{proof}
    By the definitions of  $\mathcal{Q}$  and $\mathcal{Q}_1$, we can derive the following inequalities:
    \begin{align}\label{eq_a1}
       0\le \mathcal{Q}(z) \le 1,\quad \mathcal{Q}_1(z)  \geq 1 , \quad   z\geq 0.
    \end{align}
    Then we can further derive the following inequalities:
    \begin{equation}\label{eq_a2}
        0 \le \mathcal{Q}^2(z) \le \mathcal{Q}(z) \le 1 \le \mathcal{Q}_1(z),  \quad 0 \le \mathcal{Q}(z)z \le z\le \mathcal{Q}_1(z)z.
    \end{equation}
    Since $\mathcal{L}_h, \mathcal{D}_h$ are  symmetric positive semi-definite operators, we can apply the spectral mapping theorem to extend \eqref{eq_a1}-\eqref{eq_a2} to the  inequalities \eqref{eq_norm_bounds_1}.
    % Applying the spectral mapping theorem to the operator $\tau\mathcal{L}_h$ and $\tau\mathcal{D}_h$, these scalar inequalities naturally extend to the corresponding operator inequalities. For any $\mathbf{U} \in \mathbf{H}_h^2(\Omega, \mathcal{S}^{(d)})$, we have
    % \begin{align*}
    %  \ibQ{\mathbf{U}} &= \langle \mathcal{Q}(\tau \mathcal{L}_h) \mathbf{U}, \mathbf{U} \rangle_h + \frac{\tau}{2} \langle \mathcal{L}_h \mathbf{U}, \mathbf{U} \rangle_h \\
    % & \leq \langle  \mathbf{U}, \mathbf{U} \rangle_h + \frac{\tau}{2} \langle \mathcal{L}_h \mathbf{U}, \mathbf{U} \rangle_h\leq \frac{1}{2}g^n \kappa \langle  \mathbf{U}, \mathbf{U} \rangle_h +  \frac{1}{2}\langle \mathcal{L}_h \mathbf{U}, \mathbf{U} \rangle_h\leq \langle \mathcal{L}_h \mathbf{U}, \mathbf{U} \rangle_h.
    % \end{align*}
    % The  inequalities for $V\in E_h^{\text{per}}$ can be derived similarly, which completes the proof.
\end{proof}

To proceed, we define the modified discrete energy $\mathcal{E}_h$  as follows:
\ba\label{energy_discrete}
\mathcal{E}_h[\mathbf{Q}_h, u_h,s] = \frac{K}{2} \|\nabla_h \mathbf{Q}_h\|_h^2  + B_0 \left\| \Delta_h u_h  \right\|_h^2
+\frac \kappa2(\|\mathbf{Q}_h\|_h^2+\|u_h\|_h^2)
+ s.
\ed

\begin{theorem}[Unconditional Energy Stability]\label{them3_1}
Denote $\mathcal{E}_h^n = \mathcal{E}_h[\bQ^n, \bu^n, s^n]$ as the modified discrete energy at time $t_n$.
 For any $\kappa > 0$ and  $\tau > 0$, the GSAV-EI scheme \eqref{eq_scheme_Q}--\eqref{eq_scheme_s} satisfies the following discrete energy dissipation law:
    \ba
    \mathcal{E}_h^{n+1} \le \mathcal{E}_h^n,\quad \forall n \ge 0.
    \ed
\end{theorem}
\begin{proof}
     Taking the inner product of \eqref{eq_scheme_Q} with $\delta \bQ^{n+1}$ and \eqref{eq_scheme_u} with $\delta \bu^{n+1}$, we obtain
    \ba \label{eq34}
    \langle \mLQ\bQ^{n+1}, \delta \bQ^{n+1} \rangle_h &=-\frac{1}{\tau}\langle \mq(\tau \mLQ) \delta \bQ^{n+1}, \delta \bQ^{n+1} \rangle_h+\id{\mNQ^n}{\delta \bQ^{n+1}}
    \\&= -\frac{1}{\tau}\iaQ{\delta \bQ^{n+1}}  +\id{\mNQ^n}{\delta \bQ^{n+1}}, \\
    \langle \mLu\bu^{n+1}, \delta \bu^{n+1} \rangle_h &= -\frac{1}{\tau}\langle \mq(\tau \mLu) \delta \bu^{n+1}, \delta \bu^{n+1} \rangle_h + \langle \mNu^n, \delta \bu^{n+1} \rangle_h \\
    &=-\frac{1}{\tau}\iau{\delta \bu^{n+1}} + \langle \mNu^n, \delta \bu^{n+1} \rangle_h,
    \ed
    Evaluating the difference between $\mathcal{E}_h^{n+1}$ and $\mathcal{E}_h^n$ yields the following expression:
    \ba \label{eq_diff_E}
        \mathcal{E}_h^{n+1} - {\mathcal{E}}_h^n & = \frac{K}{2}\ia{\nabla_h  \bQ^{n+1} }- \frac{K}{2}\ia{\nabla_h \bQ^{n}} + B_0\ia{\Delta_h\bu^{n+1}} - B_0\ia{\Delta_h\bu^{n}}    \\&
        \quad+ s^{n+1} - {s}^n + \frac{\kappa }{2} \left( \|\bQ^{n+1}\|_h^2 - \|\bQ^n\|_h^2 + \|\bu^{n+1}\|_h^2 - \|\bu^n\|_h^2 \right)  \\
                                            & =  \id{\mLQ\bQ^{n+1}}{\delta \bQ^{n+1}}  - \frac{1}{2} \ilQ{\delta \bQ^{n+1}} + \langle \mLu\bu^{n+1}, \delta \bu^{n+1} \rangle_h  \\
                                            &
        \quad- \frac{1}{2} \ilu{ \delta \bu^{n+1}}  -\id{\mNQ^n}{\delta \bQ^{n+1}}-\langle \mNu^n, \delta \bu^{n+1} \rangle_h.
    \ed
    Substituting \eqref{eq34} back into (\ref{eq_diff_E}) and using Lemma~\ref{lem_norm_equivalence}, we obtain the following inequality:
    \begin{equation*}
        \begin{aligned}
            {\mathcal{E}}_h^{n+1} - \mathcal{E}_h^n & = -\frac{1}{\tau}\iaQ{\delta \bQ^{n+1}}   - \frac{1}{2} \ilQ{\delta \bQ^{n+1}}
                                                  {-\frac{1}{\tau}\iau{\delta \bu^{n+1}}
            - \frac{1}{2} \ilu{ \delta \bu^{n+1}}}                 \le 0,
                  \\
            %                                     & \leq -\frac{1}{\tau}\ibQ{\delta \bQ^{n+1}}-\frac{1}{\tau}\ibu{\delta \bu^{n+1}} \le 0,
        \end{aligned}
    \end{equation*}
    which completes the proof.
\end{proof}
 \begin{remark}
 In the decoupled case \(q=0\), a discrete maximum bound principle for the
Q-tensor can be obtained by a standard comparison argument. In the fully coupled
case, however, the terms involving second-order derivatives of \(u\) prevent the comparison argument from closing. Nevertheless, the numerical results in Section \ref{section6} show that
\(\|{\mathbf Q}\|_{\infty}\) remains uniformly bounded in all tested regimes.
 \end{remark}
\section{Convergence analysis of the GSAV-EI scheme}\label{se3}
Due to the linear appearance of the auxiliary variable $s$ in
\eqref{energy_discrete}, the discrete energy law for the modified energy
$\mathcal E_h$ does not directly provide the high-order uniform bounds. For this reason, we employ a bootstrap argument in the convergence analysis.

Throughout this section, let $(\mathbf Q_h^n,u_h^n,s^n)$ denote the numerical
solution generated by the GSAV-EI scheme \eqref{eq_scheme_Q}--\eqref{eq_scheme_s}.
Let $(\mathbf Q(t),u(t),s(t))$ denote the exact solution to the continuous system
\eqref{eq_sys_Q} with the following regularity:
         \ba \label{eqa_19}
&\mathbf{Q}(t) \in L^{\infty}(0,T; \mathbf{H}^4(\Omega, \mathcal{S}^{(d)}))\cap W^{1,\infty}(0,T; \mathbf{H}^2(\Omega, \mathcal{S}^{(d)})),\\& u(t) \in L^\infty(0,T; H^6(\Omega)) \cap W^{1,\infty}(0,T; H^4(\Omega)).
            \ed
For notational simplicity, we shall omit $\mathcal I_h$ when the
continuous fields $\mathbf{Q}(t_n)$ and $u(t_n)$ appear in discrete
operators or norms.
To facilitate the error analysis, we introduce the following notation for the error functions:
\ba
\be_{\mathbf{Q}}^{n} &= \mathbf{Q}_h^n -  \mathbf{Q}(t_n),\quad
e_{u}^{n} = u_h^n -  u(t_n),\quad
e_{s}^{n} = s^n -  s(t_n),\\
\ea &= \frac{\be_{\mathbf{Q}}^{n+1} - \be_{\mathbf{Q}}^{n}}{\tau},\quad
\eb = \frac{e_{u}^{n+1} - e_{u}^{n}}{\tau},\quad
\ec = \frac{e_{s}^{n+1} - e_{s}^{n}}{\tau},\\
\delta_\tau \mathbf{Q}(t_{n+1}) &= \frac{\mathbf{Q}(t_{n+1}) - \mathbf{Q}(t_n)}{\tau},\quad
\delta_\tau u(t_{n+1}) = \frac{u(t_{n+1}) - u(t_n)}{\tau}.
\ed
 The initial errors are defined by
\begin{equation}\label{eqcv}
\begin{aligned}
    \be_{\mathbf Q}^0
    &:=\mathbf Q_h^0-\mathbf Q(0),\qquad
    e_u^0:=u_h^0-u(0),\\
    e_s^0
    &:=s^0-s(0)
    =s^0-E_1[\mathbf Q(0),u(0)],
\end{aligned}
\end{equation}
which satisfy $\be_{\mathbf{Q}}^{0}=0,
e_u^{0}=0.$
Moreover, under the regularity assumptions in \eqref{eqa_19} and the
second-order consistency of the spatial discretization,
\begin{equation}\label{eq_initial_error}
    |e_s^0|
    =
    \left|
    E_{1h}[\mathcal I_h\mathbf{Q}(0), \mathcal I_h u(0)]
    -
    E_1[\mathbf{Q}(0),u(0)]
    \right|
    \le Ch^2.
\end{equation}
By the initial error definition \eqref{eqcv}, along with \eqref{eq_initial_error},  we have
\begin{equation}\label{eq_initial_error_bound}
    \|\be_{\mathbf{Q}}^{0}\|_{H_h^1}
    +\|e_u^{0}\|_{H_h^2}
    +|e_s^{0}|
    \le Ch^2,
\end{equation}
where $C$ is a positive constant independent of $\tau$ and $h$.
\begin{theorem}[Error estimate]\label{thm_global_error}
%  Under the regularity assumptions in \eqref{eqa_19},  when $\tau$ and $h$ are sufficiently small,
% the bootstrap assumption \eqref{eq:bootstrap_discrete} holds for all
% $0\le n\le N_T$. Moreover, the fully discrete solution
% satisfies
Under the regularity assumptions in \eqref{eqa_19} and the
initialization above, the fully discrete solution $(\mathbf{Q}_h^n, u_h^n, s^n)$ of the GSAV-EI scheme satisfies the following error estimate, provided that  $\tau$ and $h$ are sufficiently small:
\[
\|\be_{\mathbf Q}^{n}\|_{H_h^1}
+\|e_u^{n}\|_{H_h^2}
+|e_s^{n}|
\le  C_{err}(\tau+h^2),
\qquad 0\le n\le N_T:=\lfloor T/\tau\rfloor.
\]
\end{theorem}
Let $M_s=\|s(t)\|_{L^\infty(0,T)}$ and choose $C_s=M_s+1$.
We first make the following temporary bootstrap
assumption:
\begin{equation}
\label{eq:bootstrap_discrete}
s^n\ge -C_s,\qquad 0\le n\le m,\qquad \text{for some}\quad 0\le m\le N_T.
\end{equation}
This is valid at $m=0$, since $s^0=E_1[\mathbf{Q}(0),u(0)]+e_s^0
\ge -M_s-Ch^2>-M_s-1=-C_s$ for sufficiently small $h$.
Once the
desired error estimate is established under
\eqref{eq:bootstrap_discrete}, the estimate extends to all
$0\le n\le N_T$ by a standard induction argument.
% \subsection{Preliminary estimates}
\subsection{Uniform boundedness of the scalar factor \texorpdfstring{$g^n$}{g^n}}
\begin{lemma}\label{gn_bound0}
    For the scalar factor $g^n$ defined in \eqref{eq_s_ex_n}, under the bootstrap assumption \eqref{eq:bootstrap_discrete},
    there exist  positive constants $g_*$ and $g^*$ independent of $\tau$ and $h$ such that  the following uniform bound holds:
    \begin{equation}
        0 <g_*\le g^n \le g^*< \infty, \quad 0\le n\le m.
    \end{equation}
\end{lemma}
\begin{proof}
   By Theorem \ref{them3_1} and the bootstrap assumption
\eqref{eq:bootstrap_discrete}, we obtain the following uniform bound
on $s^n$:
    \ba\label{eq_a5}
   -C_s \le  s^n \le \mathcal{E}_h^n \le \mathcal{E}_h^0, \quad 0\le n\le m.
    \ed
    % Recall from \eqref{eq_s_ex_n} that $E_{1h}^n$ is expressed as follows:
    % \begin{equation*}
    %     E_{1h}^n = 2 B_0 q^2 \langle D_h^2 u_h^n , \bM_h^n u_h^n \rangle_h + B_0 q^4 \| \bM_h^n u_h^n \|_h^2 + \langle f_{\mathrm{bn}}(\mathbf{Q}_h^n), 1 \rangle_h + \langle f_{\mathrm{bs}}(u_h^n), 1 \rangle_h.
    % \end{equation*}
By the definition of $\mathcal{E}_h^n$ in \eqref{energy_discrete} and \eqref{eq_a5}, we have
    \begin{equation*}
      \frac{K}{2} \|\nabla_h \mathbf{Q}_h\|_h^2  + B_0 \left\| \Delta_h u_h  \right\|_h^2+\frac \kappa2 (\|\mathbf{Q}_h\|_h^2+\|u_h\|_h^2) \le \mathcal{E}_h^0+C_s, \quad 0\le n\le m.
    \end{equation*}
 Then the following discrete $H_h^1$ and $H_h^2$ uniform bounds hold:
    \begin{equation}\label{eqc7}
        \|\mathbf{Q}_h^n\|_{H_h^1}^2 + \|u_h^n\|_{H_h^2}^2 \le C_1(\mathcal{E}_h^0,C_s), \quad 0\le n\le m,
    \end{equation}
    where $C_1$ is a positive constant independent of $\tau$ and $h$.
By applying the Cauchy-Schwarz inequality and invoking \eqref{eqc7}, we obtain the following bounds for the  terms of $E_{1h}^n$ defined in \eqref{eq_s_ex_n}:
    \begin{align*}
        \left| \langle D_h^2 u_h^n , \bM_h^n u_h^n \rangle_h \right| & \le \|D_h^2 u_h^n\|_h \|\bM_h^n\|_{h} \|u_h^n\|_{\infty} \le C, ~~
        \| \bM_h^n u_h^n \|_h^2                                       \le \|\bM_h^n\|_{h}^2 \|u_h^n\|_{\infty}^2 \le C,\\
        \left| \langle f_{\mathrm{bn}}(\mathbf{Q}_h^n), 1 \rangle_h \right| & \le C (1+\|\mathbf{Q}_h^n\|_{L_h^4}^4  ) \le C, \quad
        \left| \langle f_{\mathrm{bs}}(u_h^n), 1 \rangle_h \right| \le C (1+\|u_h^n\|_{\infty}^4) \le C.
    \end{align*}
Then, there exist positive constants $\widetilde{C}_*$ and $\widetilde{C}^*$ independent of $\tau$ and $h$ such that
    \begin{equation}\label{eq_E1_bounds}
        -\widetilde{C}_* \le E_{1h}^n \le \widetilde{C}^*, \quad 0\le n\le m.
    \end{equation}
    Combining  \eqref{eq_a5} and \eqref{eq_E1_bounds}, for $g^n = \zeta(s^n)/\zeta(E_{1h}^{n})$, we have
    \begin{equation*}
        0< g_*:= {\frac{\zeta(-C_s)}{\zeta(\widetilde{C}^*)}}\le g^n  \le {\frac{\zeta(\mathcal{E}_h^0)}{\zeta(-\widetilde{C}_*)}} =: g^* , \quad 0\le n\le m,
    \end{equation*} which completes the proof.
\end{proof}
\subsection{Higher regularity estimates for the numerical solution}
Next, we establish a direct regularity bootstrapping estimate for the GSAV-EI  scheme. The key idea is to exploit the smoothing properties of the linear semigroup generated by $\mLQ$ and $\mLu$ to obtain higher-order regularity estimates for  $\bQ^n$ and $\bu^n$.
\begin{lemma} \label{lem_analytic_semigroup}
For any self-adjoint, positive semidefinite discrete operator $A$, the
analytic semigroup generated by $-A$ satisfies the following smoothing
estimate \cite{pazy2012semigroups}:
    \begin{equation} \label{eq_smoothing_estimate}
        \vertiii{A^r e^{-tA}}_h\le \sup_{\lambda \ge 0} \lambda^r e^{-t\lambda} = C_r t^{-r}, \quad \forall t > 0, \; r \geq 0,
    \end{equation}
  where $\vertiii{\cdot}_h$ denotes the operator norm induced by the
 norm $\|\cdot\|_h$, and $C_r = (r/e)^r$.
\end{lemma}
\begin{theorem} \label{th4_1}
    % Assume the initial data $\mathbf{Q}^0 \in \mathbf{H}^2(\Omega, \mathcal{S}^{(d)})$ and $u^0 \in H^4(\Omega)$.
    % Let ${(\bQ^n, \bu^n, s^n)}$ be the numerical solution generated by the GSAV-EI scheme \eqref{eq_scheme_Q}-\eqref{eq_scheme_s}.
    Under the regularity assumptions in \eqref{eqa_19} and the bootstrap assumption \eqref{eq:bootstrap_discrete},
     there exist  positive constants $\mathcal{M}_Q$, $\mathcal{M}_u$, and $\mathcal{M}_1$, independent of $\tau$ and $h$, such that the following uniform bound holds:
    \ba
         \|\Delta_h\bQ^n\|_{h} &\le \mathcal{M}_Q,~   \norm{(-\Delta_h)^{2-\epsilon} \bu^n}_{h} \le \mathcal{M}_u, \quad 0\le n\le m,\\
           \norm{\mNu^n}_h&+\norm{\mNQ^n}_h
          \le \mathcal{M}_1, \quad 0\le n\le m,
    \ed
     where $\epsilon \in (0, 1/2)$ is a fixed sufficiently small parameter.
\end{theorem}
\begin{proof}
Let $C_0 = \max\{\norm{ \bQ^0}_{H_h^2}, \norm{ \bu^0}_{H_h^4}\}$. By the stability of the grid restriction operator and
\eqref{eqa_19}, $C_0$ is bounded independently of $h$.
 Iteratively applying \eqref{eq4_6} from $n=0$ up to $n-1$, we
    can express $\bQ^n$ in terms of the initial data $\bQ^0$ and the nonlinear terms $\mNQ^k$ for $k=0, 1, \ldots, n-1$ as follows:
    \ba \label{eq_discrete_sum}
        \bQ^n = e^{-t_n \mLQ} \bQ^0 + \sum_{k=0}^{n-1} e^{- (t_n - t_{k+1}) \mLQ} \int_0^\tau e^{-(\tau-\varsigma)\mLQ} \mNQ^k \, d\varsigma.
    \ed
   We then define a piecewise constant extension of $\mNQ^k$ by
    \begin{equation} \no
        \widetilde{\mNQ}(\sigma) := \mNQ^k, \quad \text{for } \sigma \in [t_k, t_{k+1}), \quad k = 0, 1, \ldots, n-1.
    \end{equation}
    % and rewrite the time difference in the semigroup exponent as follows:
    % \begin{equation*}
    %     (t_n - t_{k+1}) + (\tau-s) = t_n - t_{k+1} + (t_{k+1} - t_k) - (\sigma - t_k) = t_n - \sigma.
    % \end{equation*}
By setting $\sigma=t_k+\varsigma$ in each integral, we can rewrite \eqref{eq_discrete_sum} as
    \begin{equation} \label{eq_discrete_duhamel}
        \bQ^n = {e^{-t_n \mLQ} \bQ^0} + {\int_0^{t_n} e^{- (t_n - \sigma)\mLQ} \widetilde{\mNQ}(\sigma) \, d\sigma}.
    \end{equation}
    % where $\widetilde{\mNQ}(\sigma) := \mNQ^k,  \text{for } \sigma \in [t_k, t_{k+1}),  k = 0, 1, \ldots, n-1.$
%    Replacing the time index $n$ with $k$ in \eqref{eq4_6}, we have
% \begin{equation}
% \label{eq_one_step_etd_Q}
% \bQ^{k+1}
% =
% e^{-\tau\mLQ}\bQ^k
% +
% \int_0^\tau e^{-(\tau-s)\mLQ}\mNQ^k\,ds .
% \end{equation}
% We introduce the discrete evolution operator
% \[
% \mathcal U_Q^{n,k}
% :=
% e^{-(t_n-t_k)\mLQ},
% \quad
% % \]
% % so that
% % \[
% \mathcal U_Q^{n,k}
% =
% \mathcal U_Q^{n,k+1}e^{-\tau\mLQ},\qquad
% 0\le k\le n.
% \]
% Multiplying \eqref{eq_one_step_etd_Q} by $\mathcal U_Q^{n,k+1}$ and summing over
% $k=0,\ldots,n-1$, we obtain
% \begin{equation}
% \label{eq_discrete_sum}
% \bQ^n
% =
% \mathcal U_Q^{n,0}\bQ^0
% +
% \sum_{k=0}^{n-1}
% \mathcal U_Q^{n,k+1}
% \int_0^\tau e^{-(\tau-s)\mLQ}\mNQ^k\,ds .
% \end{equation}
% Since $\mLQ$ is independent of $k$, by setting $\sigma=t_k+s$ in each integral, we have
% \[
% \mathcal U_Q^{n,k+1}e^{-(\tau-s)\mLQ}
% =
% e^{-(t_n-t_{k+1})\mLQ}e^{-(t_{k+1}-\sigma)\mLQ}
% =
% e^{-(t_n-\sigma)\mLQ}.
% \]
% Then we introduce a piecewise constant extension of $\mNQ^k$ over the time interval $[0,t_n]$:
% \[
% \widetilde{\mNQ}(\sigma)=\mNQ^k,
% \qquad
% \sigma\in[t_k,t_{k+1}),\quad k=0,\ldots,n-1,
% \]
% and we can rewrite \eqref{eq_discrete_sum} as
% \begin{equation}
% \label{eq_discrete_duhamel}
% \bQ^n
% =
% e^{-t_n\mLQ}\bQ^0
% +
% \int_0^{t_n}
% e^{-(t_n-\sigma)\mLQ}\widetilde{\mNQ}(\sigma)\,d\sigma .
% \end{equation}

    By  Lemma~\ref{gn_bound0} and \eqref{eq:bootstrap_discrete}, there exists a positive constant $C_{\mathbf{N}}=C_{\mathbf{N}}(g^*,C_1)$, independent of $\tau$ and $h$, such that
    \begin{align}
        \norm{\mNQ^n}_h =\norm{-g^n \muQ^n}_h \leq C_{\mathbf{N}}. \label{align7_28}
    \end{align}
    Define  $\mathcal{A} = -K\Delta_h+\frac{1 }{2}\kappa I $ as a self-adjoint and positive definite operator.
   Applying $\mathcal{A}^{1-\frac{\epsilon}{2}}$ to both sides of
\eqref{eq_discrete_duhamel} and then taking the norm $\|\cdot\|_h$, we can deduce from Lemma~\ref{lem_analytic_semigroup} and \eqref{align7_28} the following estimate:
    \begin{equation}\label{eq_triangle_combined}
        \begin{aligned}
              \norm{\mathcal{A}^{1 - \frac{\epsilon}{2}} \bQ^n}_h
              & \le \norm{  \mathcal{A}^{1 - \frac{\epsilon}{2}} \bQ^0}_h + \int_0^{t_n} \norm{\mathcal{A}^{1-\frac{\epsilon}{2}} e^{- (t_n - \sigma)(-K\Delta_h+ \kappa I)}  \widetilde{\mNQ}(\sigma)}_h \, d\sigma                             \\
             & \le \norm{ \mathcal{A}^{1 - \frac{\epsilon}{2}} \bQ^0}_h + \int_0^{t_n} \norm{\mathcal{A}^{1-\frac{\epsilon}{2}} e^{- (t_n - \sigma)\mathcal{A}} e^{-\frac{1 }{2}\kappa (t_n - \sigma)} \widetilde{\mNQ}(\sigma)}_h \, d\sigma                             \\
             & \le \norm{ \bQ^0}_{H_h^2} + C_\epsilon \int_0^{t_n} (t_n - \sigma)^{-(1-\frac{\epsilon}{2})} e^{-\frac{1 }{2}\kappa (t_n - \sigma)} \norm{\widetilde{\mNQ}(\sigma)}_h \, d\sigma \\
            %  & \le C_0  + C_\epsilon C_{\mathbf{N}} \int_0^{t_n} (t_n - \sigma)^{-1+\frac{\epsilon}{2}} e^{-\omega_1 (t_n - \sigma)} \, d\sigma                                                                      \\
             & \le C_0 + C_\epsilon C_{\mathbf{N}} (\frac{1 }{2}\kappa)^{-\frac{\epsilon}{2}} \int_0^{\infty} \omega^{-1+\frac{\epsilon}{2}} e^{-\omega} \, d \omega                                                    \\
             & \le C_0 + C_\epsilon C_{\mathbf{N}} \, (\frac{1 }{2}\kappa)^{-\frac{\epsilon}{2}} \Gamma\left(\frac{\epsilon}{2}\right)=: C_2(C_0, C_{\mathbf{N}}, \epsilon),
        \end{aligned}
    \end{equation}
    where $\omega = \frac{1 }{2}\kappa (t_n - \sigma)$ and  $\Gamma(\cdot)$ represents the Gamma function.

    We now estimate the smoothed  nonlinear term $\ma^{-\frac{\ep}{2}}\mNu^n$. Most of its components are
lower-order terms and can be bounded straightforwardly in the norm $\|\cdot\|_h$ by \eqref{eqc7}, except for the term
$\operatorname{div}_{h,c}^{\,2}(\bM_h\bu)$.
 By employing the discrete Fourier transform and Parseval's identity under periodic boundary conditions, the action of the double divergence operator is dominated by the discrete Laplacian \cite{shen2011spectral,wise2009energy,wang2011energy}. This directly yields the following estimate
for any
$\boldsymbol{\Phi}_h\in
\mathbf{H}_h^2(\Omega,\mathbb{M}_h^{(d)})$:
\begin{equation}\label{eq_double_div_bound}
    \left\|
    \operatorname{div}_{h,c}^{\,2}\boldsymbol{\Phi}_h
    \right\|_h^2
    \le
    C
    \left\|
    \Delta_h\boldsymbol{\Phi}_h
    \right\|_h^2.
\end{equation}
    Since the fractional operator $\ma^{-\frac{\ep}{2}}$ commutes with $\Delta_h$ and
$\operatorname{div}_{h,c}^{\,2}$, we can use \eqref{eqc7},  \eqref{eq_double_div_bound} and \eqref{eq_triangle_combined} to directly establish the following estimate:
\begin{equation}
    \begin{aligned}
      \norm{
\ma^{-\frac{\ep}{2}}
\operatorname{div}_{h,c}^{\,2}
(\bM_h\bu)
}_h
         &\le  C\norm{ -\Delta_h( \ma^{-\frac{\ep}{2}} (\bM_h \bu))}_h \\
         &\le  C\norm{\ma^{1-\frac{\ep}{2}} (\bM_h \bu)}_h
         \\
         &\le C\left(\norm{\bM_h}_{\infty}\norm{\ma^{1-\frac{\ep}{2}}\bu}_h + \norm{\ma^{1-\frac{\ep}{2}}\bM_h}_h\norm{\bu}_{\infty}\right)\\
         &\le C\norm{\ma^{1-\frac{\ep}{2}} \bM_h }_h \norm{\ma^{1-\frac{\ep}{2}} \bu}_h\le C(C_1,C_2).
    \end{aligned}
\end{equation}
  Then there exists a constant $C_\mathcal{N} > 0$, independent of $\tau$ and $h$, such that
%    the following uniform bound holds for the smoothed nonlinear term $\ma^{-\frac{\ep}{2}}\mNu^n$:
    \begin{equation} \label{eq_Nu_bound}
        \norm{\ma^{-\frac{\ep}{2}} \mNu^n}_{h}  \le C_\mathcal{N}\left(g^*,C_1,C_2 \right).
    \end{equation}

    Next, we apply the same procedure to the equation for $\bu^n$. By iterating \eqref{eq4_7} for $\bu^n$ from $n=0$ up to $m$, we can express $\bu^n$ in terms of the initial data  and the nonlinear terms  as follows:
%     \begin{equation}
% \label{eq_v}
% u_h^n
% \end{equation}
    \ba
        \bu^n = {e^{- t_n \mLu} \bu^0} + {\int_0^{t_n} e^{- (t_n - \sigma)\mLu} \widetilde{\mNu}(\sigma) \, d\sigma},\quad 0\le n\le m, \label{eq_v}
    \ed
    where
     $\widetilde{\mNu}(\sigma) = \mNu^k$ for $\sigma \in [t_k, t_{k+1}), k = 0, 1, \ldots, n-1$.
    Define $\mathcal{B} =2B_0 \Delta^2_h$ as a self-adjoint and positive semi-definite operator.
    Applying  $\ma^{-\frac{\ep}{2}}\mathcal{B}^{1-\frac{\epsilon}{4}}$ to both sides of \eqref{eq_v} and  then taking the norm $\|\cdot\|_h$, we can deduce from Lemma~\ref{lem_analytic_semigroup} and \eqref{eq_Nu_bound} the following estimate:
    \begin{align}
            \norm{\ma^{-\frac{\ep}{2}}\mathcal{B}^{1-\frac{\epsilon}{4}} \bu^n}_{h}
             & \!\le\! \norm{\ma^{-\frac{\ep}{2}}\mathcal{B}^{1-\frac{\epsilon}{4}} \bu^0}_{h} \!+ \!\int_0^{t_n} \norm{ \mathcal{B}^{1-\frac{\epsilon}{4}} e^{- (t_n - \sigma)\mathcal{B}}e^{-\kappa  (t_n-\sigma)} \ma^{-\frac{\ep}{2}} \widetilde{\mNu}(\sigma) }_{h} \, d\sigma      \no\\
             & \le  \norm{ \bu^0}_{H_h^4} + C_\epsilon \int_0^{t_n} (t_n-\sigma)^{-\left(1 - \frac{\epsilon}{4}\right)} e^{-\kappa(t_n-\sigma)} \norm{\ma^{-\frac{\ep}{2}}\widetilde{\mNu}(\sigma)}_{h} \, d\sigma\no \\
            %  & \le C_0  + C_\epsilon C_\mathcal{N} \int_0^{t_n} (t_n-\sigma)^{-1 + \frac{\epsilon}{4}} e^{-\kappa(t_n-\sigma)} \, d\sigma                                                          \\
             & \le C_0 +  C_\epsilon C_\mathcal{N} \int_0^{\infty} \omega_1^{-1 + \frac{\epsilon}{4}} e^{- \omega_1} \, d\omega_1                                                                                                                                 \no \\
             &\le C_0 +  C_\epsilon C_\mathcal{N}\kappa^{-\frac{\epsilon}{4}} \Gamma\left(\frac{\epsilon}{4}\right) ,\label{eq_combined_estimate}
    \end{align}
    where $\omega_1 = \kappa(t_n-\sigma)$.
    Combining the estimate \eqref{eq_combined_estimate} with \eqref{eq_Nu_bound}, there exists a constant $\mathcal{M}_u > 0$, independent of $\tau$ and $h$, such that the following uniform bound holds for $\bu^n$:
    \begin{equation}\label{eq4_13}
          \norm{(-\Delta_h)^{2-\epsilon}\bu^n}_{h}\le C\norm{\ma^{-\frac{\ep}{2}}\mathcal{B}^{1-\frac{\epsilon}{4}}\bu^n}_{h}  \le \mathcal{M}_u\left(C_0,  C_\mathcal{N}, \epsilon \right),\quad 0\le n\le m.
    \end{equation}

 Combining \eqref{eq_triangle_combined} and \eqref{eq4_13} with Lemma \ref{gn_bound0} yields the following bound for the nonlinear term:
    \ba\label{eqa4_13}
         \norm{\ma^{\frac{\ep}{2}}\mNQ^n}_{h} &\le  g^n\left( \norm{\ma^{\frac{\ep}{2}}\muQa^n}_{h}+\norm{\ma^{\frac{\ep}{2}}\muQb^n}_{h}\right)
         \\&\le C(\norm{\bQ^n}_{\infty}\norm{\ma^{\frac{\ep}{2}}\bQ^n}_{h}+\norm{\bu^n}_{\infty}\norm{(-\Delta_h)^{2-\epsilon}\bu^n}_{h}) \le C(C_2,\mathcal{M}_u).
    \ed
    Applying the smoothing estimate \eqref{eq_triangle_combined} in conjunction with \eqref{eqa4_13}, there exists a constant $\mathcal{M}_Q > 0$, independent of $\tau$ and $h$, such that the following uniform bound holds for $\bQ^n$:
    \begin{equation}\label{eqa4_14}
        \norm{\ma\bQ^n}_{h} = \norm{\ma^{1 - \frac{\epsilon}{2}}(\ma^{\frac{\ep}{2}}\bQ^n)}_{h} \le \mathcal{M}_Q\left(C_0,  C_2,\mathcal{M}_u, \epsilon \right),\quad 0\le n\le m.
    \end{equation}
    Finally, combining the estimates \eqref{eq4_13} and \eqref{eqa4_14}, there exists a constant $\mathcal{M}_1 > 0$, independent of $\tau$ and $h$, such that
    \ba
          \norm{\mNu^n}_h+\norm{\mNQ^n}_h
          \le \mathcal{M}_1\left(\mathcal{M}_Q ,C_1\right),\quad 0\le n\le m.
    \ed
\end{proof}
\subsection{Error estimates for the GSAV-EI scheme}
We first define the continuous counterparts of $\mNQ$ and $\mNu$,
denoted by $\mNQz(t)$ and $\mNuz(t)$, respectively, as
\begin{equation*}
\begin{aligned}
    \mNQz(t)&:=-g(t)\mathbf H(t),\qquad
    \mNuz(t):=-g(t)\mu(t),
    \qquad 0\le t\le T.
\end{aligned}
\end{equation*}
By substituting the exact solutions into the numerical scheme \eqref{eq4_6}-\eqref{eq4_8} and rearranging the terms,  we can express the exact solutions at time $t_{n+1}$ as:
\begin{subequations}\label{eq:system}
\begin{align}
\mathbf{Q}(t_{n+1}) & = e^{-\tau \mLQ} \mathbf{Q}(t_n) + \int_0^\tau e^{-(\tau-\varsigma)\mLQ} \left( \mNQz(t_n) + \mathbf{T}_{\mathbf{Q}}^n(\varsigma) \right) d\varsigma,\label{eq:system-a}\\
u(t_{n+1}) & = e^{-\tau \mLu} u(t_n) + \int_0^\tau e^{-(\tau-\varsigma)\mLu} \left( \mNuz(t_n) + T_u^n(\varsigma) \right) d\varsigma,\label{eq:system-b}\\
\delta_\tau s(t_{n+1}) & =-  \left\langle \mNQz(t_n), \delta_\tau \mathbf{Q}(t_{n+1}) \right\rangle_h - \left\langle \mNuz(t_n), \delta_\tau u(t_{n+1}) \right\rangle_h  + T_s^n, \label{eq:system-c}
\end{align}
\end{subequations}
where the truncation errors $\mathbf{T}_{\mathbf{Q}}^n(\varsigma)$, $T_u^n(\varsigma)$, and $T_s^n$  are defined as
\begin{equation}\label{eq_truncation_errors}
\begin{aligned}
\mathbf T_{\mathbf Q}^n(\varsigma)
&:=
\mNQz(t_n+\varsigma)-\mNQz(t_n)
+K\left(
\Delta\mathbf Q(t_n+\varsigma)
-\Delta_h\mathbf Q(t_n+\varsigma)
\right),\\
T_u^n(\varsigma)
&:=
\mNuz(t_n+\varsigma)-\mNuz(t_n)
+2B_0\left(
\Delta_h^2u(t_n+\varsigma)
-\Delta^2u(t_n+\varsigma)
\right),\\
T_s^n
&:=
\delta_\tau s(t_{n+1})
+\left\langle
\mNQz(t_n),\delta_\tau\mathbf Q(t_{n+1})
\right\rangle_h
+\left\langle
\mNuz(t_n),\delta_\tau u(t_{n+1})
\right\rangle_h .
\end{aligned}
\end{equation}
By the regularity assumptions in \eqref{eqa_19} and Taylor's theorem, the truncation errors satisfy
\begin{equation*}
\begin{aligned}
\left\| \mathbf{T}_{\mathbf{Q}}^n(\varsigma) \right\|_h &\le C (\tau + h^2), ~
\left\| T_u^n(\varsigma) \right\|_h \le C (\tau + h^2), ~
\left| T_s^n \right| \le C (\tau + h^2), \quad 0\le \varsigma \le \tau.
\end{aligned}
\end{equation*}
Applying  $\frac{1}{\tau}(\mq(\tau \mLQ)+\tau \mLQ)$ to both sides of \eqref{eq:system-a} and $\frac{1}{\tau}(\mq(\tau \mLu)+\tau \mLu)$ to \eqref{eq:system-b}, we can derive the following equivalent transformed consistency relations:
% \ba \label{eqa_14}
% \mq(\tau \mLQ) \delta_\tau \mathbf{Q}(t_{n+1})+\mLQ \mathbf{Q}(t_{n+1})  & =  \mNQz(t_n) +\widetilde{ \mathbf{T}}_{\mathbf{Q}}^n,\\
% \mq(\tau \mLu) \delta_\tau u(t_{n+1}) +\mLu u(t_{n+1}) & =  \mNuz(t_n) + \widetilde{T}_u^n,
% \ed
% where we have
% \begin{small}
    \ba \label{eqa_14}
\mq(\tau \mLQ) \delta_\tau \mathbf{Q}(t_{n+1})+\mLQ \mathbf{Q}(t_{n+1})  & =  \mNQz(t_n) +\widetilde{ \mathbf{T}}_{\mathbf{Q}}^n,\\
\mq(\tau \mLu) \delta_\tau u(t_{n+1}) +\mLu u(t_{n+1}) & =  \mNuz(t_n) + \widetilde{T}_u^n,
\ed
% \end{small}
 where  $\widetilde{ \mathbf{T}}_{\mathbf{Q}}^n= \int_0^\tau \mathcal K_{\mLQ}(\varsigma)  \mathbf{T}_{\mathbf{Q}}^n(\varsigma) d\varsigma$ and $\widetilde{T}_u^n= \int_0^\tau \mathcal K_{\mLu}(\varsigma)  T_u^n(\varsigma) d\varsigma$ with
\[
\mathcal K_{\Lambda_h}(\varsigma)
:=
\frac{1}{\tau}
\bigl(\mq(\tau\Lambda_h)+\tau\Lambda_h\bigr)
e^{-(\tau-\varsigma)\Lambda_h},\quad \Lambda_h\in\{\mLQ,\mLu\}.
\]
Since $\Lambda_h$ is self-adjoint and positive definite, we can calculate
\[
\int_0^\tau\mathcal K_{\Lambda_h}(\varsigma)\,d\varsigma=I,
\qquad
\vertiii{\int_r^\tau\mathcal K_{\Lambda_h}(\varsigma)\,d\varsigma}_h
\le1,\quad 0  \le r \le \tau.
\]
Using
$\mathbf{T}_{\mathbf{Q}}^n(\varsigma)=\mathbf{T}_{\mathbf{Q}}^n(0)+\int_0^\varsigma\partial_r\mathbf{T}_{\mathbf{Q}}^n(r)\,dr$
and Fubini's theorem, we obtain
\ba
\widetilde{ \mathbf{T}}_{\mathbf{Q}}^n
={}&
\left(
\int_0^\tau
\mathcal K_{\mLQ}(\varsigma)\,d\varsigma
\right)\mathbf{T}_{\mathbf{Q}}^n(0)
\nonumber\\
&+
\int_0^\tau
\mathcal K_{\mLQ}(\varsigma)
\left(
\int_0^\varsigma
\partial_r\mathbf{T}_{\mathbf{Q}}^n(r)\,dr
\right)d\varsigma
\nonumber\\
={}&
\mathbf{T}_{\mathbf{Q}}^n(0)
+
\int_0^\tau
\left(
\int_r^\tau
\mathcal K_{\mLQ}(\varsigma)\,d\varsigma
\right)
\partial_r\mathbf{T}_{\mathbf{Q}}^n(r)\,dr.
\ed
A similar estimate holds for $ \widetilde{T}_u^n$.
By  the
regularity assumptions \eqref{eqa_19}, we have
\[
\|\mathbf T_{\mathbf Q}^n(0)\|_h+\|T_u^n(0)\|_h\le Ch^2,
\qquad
\max_{0\le\varsigma\le\tau}
\left(
\|\partial_\varsigma\mathbf T_{\mathbf Q}^n(\varsigma)\|_h
+\|\partial_\varsigma T_u^n(\varsigma)\|_h
\right)\le C.
\]
Consequently, we can establish the following bounds for the  truncation errors:
\ba \label{eqft}
\left\|\widetilde{\mathbf T}_{\mathbf Q}^n\right\|_h
+\left\|\widetilde T_u^n\right\|_h
\le C(\tau+h^2).
\ed

Subtracting the consistency relations \eqref{eqa_14} from the
numerical schemes \eqref{eq_scheme_Q} and \eqref{eq_scheme_u},
respectively, and subtracting the $s$-equation in
\eqref{eq:system-c} from \eqref{eq_scheme_s}, we obtain the following
error equations:
\begin{align}
    \mq(\tau \mLQ)                           & \ea + \mLQ \be_{\mathbf{Q}}^{n+1} = \mNQ^n-\mNQz(t_n)  - \widetilde{ \mathbf{T}}_{\mathbf{Q}}^n,\label{eq_error_s_final}                                     \\
    \mq(\tau \mLu)                           & \eb + \mLu e_{u}^{n+1} =\mNu^n-\mNuz(t_n)  - \widetilde{T}_u^n,\label{eq_error_s_final2}                                                                     \\
    \frac{1}{\tau} ({e}_{s}^{n+1}-e_{s}^{n}) & = - \left\langle \mNQ^n, \ea \right\rangle_h - \left\langle \mNu^n, \eb \right\rangle_h  \no                                     \\
                                               - \langle  \mNQ^n&-\mNQz(t_n), \da  \rangle_h - \left\langle \mNu^n-\mNuz(t_n), \db \right\rangle_h - T_s^n.\label{eq_error_s_final3}
    % \\
    % {e}_s^{n+1} &
    % = \xi^{n+1}  \tilde{e}_s^{n+1} + (1 - \xi^{n+1}) {\left( E_{1h}^{n+1} - s(t_{n+1}) \right)}.\label{eq_error_s_final4}
\end{align}
% Here and in what follows, the analogous notation convention is applied to other operators and variables.
% The corresponding truncation errors $\widetilde{ \mathbf{T}}_{\mathbf{Q}}^n$, $T_u^n$, and $T_s^n$, arising from the time derivative and Laplace operator discretizations of $\mathbf{Q}$ and $u$, and the time derivative discretizations of   $s$, respectively, can be estimated as follows:
% \begin{equation} \label{eq_truncation_estimate}
%     \|\widetilde{ \mathbf{T}}_{\mathbf{Q}}^n\|_h + \|T_u^n\|_h  \le C (\tau + h^2), \quad |T_s^n| \le C \tau.
% \end{equation}
% \begin{proposition} \label{thm_global_error}
%  Under the regularity assumptions in \eqref{eqa_19} and Theorem \ref{th4_1}, the fully discrete solution $(\mathbf{Q}_h^n, u_h^n, s^n)$ of the GSAV-EI scheme satisfies the following error estimate, provided that  $\tau$ and $h$ are sufficiently small:
%     \begin{equation} \label{eq_final_total_error}
%         \| \be_{\mathbf{Q}}^{n} \|_{H^1_h} + \| e_{u}^{n} \|_{H^2_h} + |e_s^n| \le C (\tau + h^2), \quad 0 \le n \le  m+1,
%     \end{equation}
%     where $C$ is a positive constant independent of $\tau$  and $h$.
% \end{proposition}
% \begin{proof}
The following lemma establishes the Lipschitz continuity of the nonlinear terms $\mNQ$ and $\mNu$ with respect to the error variables $\be_{\mathbf{Q}}^{n}$, $e_{u}^{n}$, and $e_s^n$.
    % Let $g^n$ be the numerical multiplier defined in \eqref{eq_scheme_s} and $g(t_n)$ be the exact continuous multiplier defined in \eqref{eq_modified_system}.
\begin{lemma}\label{thm_full_lipschitz}
  Under the regularity assumptions in \eqref{eqa_19} and Theorem \ref{th4_1},
%   the numerical multiplier $g^n$ satisfies the following Lipschitz continuity estimate:
%     \begin{equation}
%         |g^n - g(t_n)| \le C \left( |e_s^n| + \|\be_{\mathbf{Q}}^n\|_{H^1_h} + \|e_u^{n}\|_{H^2_h} + h^2 \right), \quad 0 \le n \le m.
%     \end{equation}
%     Consequently,
    the nonlinear terms $\mNQ$ and $\mNu$  satisfy the following Lipschitz continuity estimates:
    \ba\label{eq_full_Lip_Q}
    \| \mNQ^n -\mNQz(t_n)\|_h \le C_5 \left( |e_s^n| + \|\be_{\mathbf{Q}}^{n}\|_{H^1_h} + \|e_{u}^{n}\|_{H^2_h}+ h^2 \right),\quad 0 \le n \le m,\\
    \|  \mNu^n- \mNuz(t_n)\|_h \le C_6 \left( |e_s^n| + \|\be_{\mathbf{Q}}^{n}\|_{H^2_h} + \|e_{u}^{n}\|_{H^2_h}+ h^2 \right),\quad 0 \le n \le m,
    \ed
    where $C_5$ and $C_6$ are positive constants independent of $\tau$ and $h$.
\end{lemma}
\begin{proof}
       % We first decompose the nonlinear term $\muQ$ into two components: the polynomial component $\muQa$ and the derivative coupling component $\muQb$, as follows:
    % \begin{equation}\label{eqa_28}
    %     \begin{aligned}
    %         \muQ(\mathbf{Q}_h^n, u_h^n) =
    %          & \underbrace{ \A \mathbf{Q}_h^n - \B \left( (\mathbf{Q}_h^n)^2 - \frac{\operatorname{tr}((\mathbf{Q}_h^n)^2)}{3}\mathbf{I} \right) + \C \operatorname{tr}((\mathbf{Q}_h^n)^2)\mathbf{Q}_h^n +\left[  \frac{2B_0 q^4}{s_+^2} \mathbf{Q}_h^n (u_h^n)^2 \right] }_{\muQa^n} \\
    %          & + \underbrace{ \left[  \frac{2B_0 q^2}{s_+} \left( u_h^n D_h^2 u_h^n - \frac{\operatorname{tr}(u_h^n D_h^2 u_h^n)}{3}\mathbf{I} \right) \right] }_{\muQb^n}.
    %     \end{aligned}
    % \end{equation}
  Using the definition of $\muQaz$ in \eqref{eqd1} and
Theorem~\ref{th4_1}, we estimate the discrete $L^2$ norm of the difference as follows:
\begin{align}
\label{eq_muQb_discrete_bound1}
&\| \muQa(\bQ^n,u_h^n)
-\muQaz(\mathbf{Q}(t_n),u(t_n))\|_h
\nonumber\\
&\quad \le
|\A|\|\be_{\mathbf{Q}}^{n}\|_h
+\B
\left(
\|\mathbf{Q}_h^n\|_{\infty}
+\|\mathbf{Q}(t_n)\|_{\infty}
\right)
\|\be_{\mathbf{Q}}^{n}\|_h
\nonumber\\
&\qquad
+\C
\left(
\|\bQ^n\|_{\infty}^2
+\|\mathbf{Q}(t_n)\|_{\infty}^2
\right)
\|\be_{\mathbf{Q}}^{n}\|_h
+\kappa\|\be_{\mathbf{Q}}^{n}\|_h
\nonumber\\
&\qquad
+2B_0q^4s_+^{-2}
(\|\bQ^n\|_{\infty}
\|u_h^n+u(t_n)\|_{\infty}
\|e_u^{n}\|_h+\|(u_h^n)^2\|_{\infty}
\|\be_{\mathbf{Q}}^{n}\|_h)
% \nonumber\\
% &\qquad
% +2B_0q^4s_+^{-2}
% \|(u_h^n)^2\|_{\infty}
% \|\be_{\mathbf{Q}}^{n}\|_h
\nonumber\\
&\quad \le
C(\mathcal{M}_Q,C_1)\left(
\|\be_{\mathbf{Q}}^{n}\|_h
+\|e_u^{n}\|_h
\right).
\end{align}
    % Evaluating the difference between $\muQb$ and $\muQbz(t_n)$ in the $\ell^2$-norm, we can  isolate the spatial truncation error from the purely discrete difference:
    % \begin{align} \label{eq_muQb_triangle}
    %     \| \muQb^n - \muQbz(t_n) \|_h
    %      & \le \underbrace{ \left\| \muQbz(t_n) - \muQb^n(  u(t_n)) \right\|_h }_{\text{Spatial truncation error} \le C h^2}  + \left\| \muQb^n(  u(t_n)) - \muQb^n(u_h^n) \right\|_h.
    % \end{align}
 Using the definition of $\mathbf H_{{\rm d},h}$ in \eqref{eqc8} and
Theorem~\ref{th4_1}, we estimate the discrete $L^2$ norm of the difference as follows:
    \ba \label{eq_muQb_discrete_bound}
    &\left\| \mathbf H_{{\rm d},h}(  u(t_n)) - \mathbf H_{{\rm d},h}(u_h^n) \right\|_h\\
    &= C\left\| \left(   u(t_n) - u_h^n \right) D_h^2   u(t_n) + u_h^n \left( D_h^2   u(t_n) - D_h^2 u_h^n \right) \right\|_h \\
    % &\le C\left\| e_u^n D_h^2   u(t_n) \right\|_h + \left\| u_h^n D_h^2 e_u^n \right\|_h \\
    &\le C\|D_h^2   u(t_n)\|_{\infty} \|e_u^n\|_h + \|u_h^n\|_{\infty} \|D_h^2 e_u^n\|_h \\
    &\le C \left( \|e_u^n\|_h + \|D_h^2 e_u^n\|_h \right) \le C(C_1)  \|e_u^n\|_{H^2_h}.
    \ed
 Combining the bounds in \eqref{eq_muQb_discrete_bound1}--\eqref{eq_muQb_discrete_bound}, we obtain the desired Lipschitz continuity estimate for $\muQ$:
    \ba \label{eqr1}
        \| \muQ^n-\muQz(t_n)\|_h & \le \| \muQa(\bQ^n,u_h^n)
-\muQaz(\mathbf{Q}(t_n),u(t_n))\|_h  \\
&+ \left\| \mathbf H_{{\rm d},h}(  u(t_n)) - \mathbf H_{{\rm d},h}(u_h^n) \right\|_h + \left\| \muQbz(t_n) - \mathbf H_{{\rm d},h}(  u(t_n)) \right\|_h       \\
                                 & \le C_3(\mathcal{M}_Q,C_1) \left( \|\be_{\mathbf{Q}}^{n}\|_h + \|e_{u}^{n}\|_{H^2_h} + h^2 \right),
    \ed
    where the last term is bounded by $C h^2$ due to the spatial truncation error.

    Next, for  $\muu$ defined in \eqref{eqc8},
    evaluating the difference between $\muua(\mathbf Q_h^n,u_h^n)$ and  $\muuaz(t_n)$ in the norm $\|\cdot\|_h$ and using Theorem \ref{th4_1}, we can derive the following estimate:
    \begin{align}
        \| \muua(\mathbf Q_h^n,u_h^n)-\muuaz(t_n)\|_h \le C(\mathcal{M}_Q,C_1) \left( \|\be_{\mathbf{Q}}^{n}\|_h + \|e_{u}^{n}\|_h \right).
    \end{align}
    For the term $\mu_{{\rm d},h}$ defined in \eqref{eqc8},  under the  regularity assumptions in \eqref{eqa_19} and Theorem \ref{th4_1}, along with  \eqref{eq_double_div_bound}, we can derive the following estimate:
 \begin{align*}
&\|\mu_{{\rm d},h}(\mathbf Q_h^n,u_h^n)-\mu_{{\rm d},h}(t_n)\|_h\\
&\le
C
\left\|
D_h^2u(t_n):\bM(t_n)
-
D_h^2u_h^n:\bM_h^n
\right\|_h
+
C
\left\|
\operatorname{div}_{h,c}^{\,2}
\left(
\bM_h^n u_h^n-\bM(t_n)u(t_n)
\right)
\right\|_h
\nonumber\\
&\le
C\left(
\|\be_{\mathbf Q}^n\|_h
+
\|e_u^n\|_{H_h^2}
\right)
+
C
\left\|
\bM_h^n u_h^n-\bM(t_n)u(t_n)
\right\|_{H_h^2}
\nonumber\\
&\le
C\left(
\|\be_{\mathbf Q}^n\|_h
+
\|e_u^n\|_{H_h^2}
\right)
+
C\left(
\|\bM_h^n-\bM(t_n)\|_{H_h^2}\|u_h^n\|_{H_h^2}
+
\|\bM(t_n)\|_{H_h^2}\|e_u^n\|_{H_h^2}
\right)
\nonumber\\
&\le
C(\mathcal{M}_Q,C_1)\left(
\|\be_{\mathbf Q}^n\|_{H_h^2}
+
\|e_u^n\|_{H_h^2}
\right),
% \label{eq_muub_error_bound}
\end{align*}
where we have used the discrete Sobolev embedding
$H_h^2\hookrightarrow L_h^\infty$ for $d\le3$,
the product estimate
\[
\|vw\|_{H_h^2}
\le C\|v\|_{H_h^2}\|w\|_{H_h^2}.
\]
Combining the above bounds, we obtain the desired Lipschitz continuity estimate for $\muu^n$:
    \ba \label{eqr2}
        \| \muu^n-\muuz(t_n)\|_h & \le  \| \muua(\mathbf Q_h^n,u_h^n)-\muuaz(t_n)\|_h \\&+ \|\mu_{{\rm d},h}(\mathbf Q_h^n,u_h^n)-\mu_{{\rm d},h}(t_n)\|_h + \| \mu_{{\rm d},h}(t_n) - \mu_{{\rm d}}( t_n) \|_h     \\
                                 & \le C_4(\mathcal{M}_Q,C_1) \left( \|\be_{\mathbf{Q}}^{n}\|_{H^2_h} + \|e_{u}^{n}\|_{H^2_h} + h^2 \right).
    \ed

    Evaluating the difference between the numerical multiplier $g^n$ and the exact multiplier $g(t_n)$ and applying the mean value theorem and the condition that $\zeta\in C^1(\mathbb R;(0,\infty))$, we can express it as follows:
\ba \label{eq_14}
|g^n - g(t_n)|
&\le
\left|
\frac{\zeta(s^n)}{\zeta(E_{1h}^n)}
-
\frac{\zeta(s(t_n))}{\zeta(E_{1h}^n)}
\right|
+
\left|
\frac{\zeta(s(t_n))}{\zeta(E_{1h}^n)}
-
\frac{\zeta(s(t_n))}{\zeta(E_1(t_n))}
\right|
\\
&\le
\frac{\zeta'(\varpi_1)}{\zeta(E_{1h}^n)}
|s^n-s(t_n)|
+
\zeta(s(t_n))
\frac{\zeta'(\varpi_2)}
{\zeta(E_{1h}^n)\zeta(E_1(t_n))}
\left|E_1(t_n)-E_{1h}^n\right|
\\
&\le
\frac{\zeta'(\varpi_1)}{\zeta(-\widetilde{C}_*)}
|s^n-s(t_n)|
+
\zeta(C^*)
\frac{\zeta'(\varpi_2)}
{\zeta(-\widetilde{C}_*)\zeta(-C_*)}
\left|E_1(t_n)-E_{1h}^n\right|
\\
&\le
C_g(\mathcal{E}_h^0,\widetilde{C}_*,C_*,C^*,C_1,\zeta)
\left(
|e_s^n|
+
\|\be_{\mathbf{Q}}^n\|_{H^1_h}
+
\|e_u^n\|_{H^2_h}
+
h^2
\right),
\ed
    where $s(t_n)=E_1(t_n)$, $\varpi_1$ is between $s^n$ and $s(t_n)$,  ${\varpi_2}$ is between $E_{1h}^n$ and $E_1(t_n)$, and $\zeta'(\cdot)$ is bounded on the closed  interval.
    % Applying the Mean Value Theorem to $\mathcal{T}_1$ with $\xi$  between $s^n$ and $s(t_n)$ yields the following bound:
    % \begin{equation*}
    %     \mathcal{T}_1 \le \frac{\zeta(\xi)}{\zeta(E_{1h}^n)} |s^n - s(t_n)| \le C |e_s^n|, \xi
    % \end{equation*}
    % For $\mathcal{T}_2$, applying the Mean Value Theorem  with $\widetilde{\eta}$ between $E_{1h}^n$ and $E_1(t_n)$ and the triangle inequality yields the following bound:
    % \begin{align}
    %     \mathcal{T}_2 & = \zeta(s(t_n)) \left| \zeta(-E_{1h}^n) - \zeta(-E_1(t_n)) \right| \notag             \\
    %                   & \le \zeta(s(t_n)) \zeta(-\widetilde{\eta}) \left| E_1(t_n) - E_{1h}^n \right| \notag \\
    %                   & \le C \left({\left| E_{1h}(\mathbf{Q}(t_n), u(t_n)) - E_{1h}^n \right|}
    %     + {\left| E_1(t_n) - E_{1h}(\mathbf{Q}(t_n), u(t_n)) \right|} \right)
    %     \notag                                                                                                                                      \\
    %                   & \le C \left( \|\be_{\mathbf{Q}}^n\|_{H^1_h} + \|e_u^n\|_{H^2_h} + h^2 \right). \notag
    % \end{align}
    % Combining the estimates for $\mathcal{T}_1$ and $\mathcal{T}_2$ in \eqref{eq_g_diff_split_2step} yields the desired Lipschitz continuity estimate for $g^n$:
    % \begin{equation}\label{eq_14}
    %     |g^n - g(t_n)| \le C \left( |e_s^n| + \|\be_{\mathbf{Q}}^n\|_{H^1_h} + \|e_u^{n}\|_{H^2_h} + h^2 \right).
    % \end{equation}
    Based on \eqref{eqr1}, \eqref{eqr2} and  \eqref{eq_14}, we can derive the following estimates for $\mNQ^n -\mNQz(t_n)$ and $\mNu^n - \mNuz(t_n)$:
    \begin{equation} \label{eq_J_Q_bound}
        \begin{aligned}
            \| \mNQ^n -\mNQz(t_n)\|_h
             & \le \| \muQz(t_n) \|_h |g^n - g(t_n)| + g^n \|\muQ^n- \muQz(t_n)\|_h                                                                                                       \\
            %  & \le  CC_g \left( |e_s^n| + \|\be_{\mathbf{Q}}^n\|_{H^1_h} + \|e_u^{n}\|_{H^2_h} +h^2\right) + C_3 g^* \left( \|\be_{\mathbf{Q}}^{n}\|_{H^1_h} + \|e_{u}^{n}\|_{H^2_h} + h^2 \right) \\
             & \le C_5(C_g,C_3,g^*) \left( |e_s^n| + \|\be_{\mathbf{Q}}^{n}\|_{H^1_h} + \|e_{u}^{n}\|_{H^2_h}+h^2 \right),                                                                               \\
            \| \mNu^n - \mNuz(t_n)\|_h
             & \le    \| \muuz(t_n) \|_h |g^n - g(t_n)| + g^n \|\muu^n- \muuz(t_n)\|_h
                                                                                            \\
            %  & \le  C \left( |e_s^n| + \|\be_{\mathbf{Q}}^n\|_{H^1_h} + \|e_u^{n}\|_{H^2_h} +h^2\right) + C \left( \|\be_{\mathbf{Q}}^{n}\|_{H^2_h} + \|e_{u}^{n}\|_{H^2_h} + h^2 \right) \\
             & \le C_6(C_g,C_4,g^*) \left( |e_s^n| + \|\be_{\mathbf{Q}}^{n}\|_{H^2_h} + \|e_{u}^{n}\|_{H^2_h}+h^2 \right),
        \end{aligned}
    \end{equation}
    which completes the proof.
\end{proof}

Then we can establish the following error energy estimate for the GSAV-EI scheme.
    Taking the discrete inner products of (\ref{eq_error_s_final}) with $\ea$ and (\ref{eq_error_s_final2}) with $\eb$, and summing the results yields the following error energy identity:
    \begin{small}
           \ba \label{eqa_17}
    &\left\langle  \mathcal{Q}(\tau \mLQ) \ea, \ea \right\rangle_h + \left\langle \mLQ \be_{\mathbf{Q}}^{n+1}, \ea \right\rangle_h
    + \left\langle  \mathcal{Q}(\tau \mLu) \eb, \eb \right\rangle_h  \\
    &+ \left\langle \mLu e_{u}^{n+1}, \eb \right\rangle_h
    = \left\langle   \mNQ^n -\mNQz(t_n) - \widetilde{ \mathbf{T}}_{\mathbf{Q}}^n, \ea \right\rangle_h + \left\langle \mNu^n-\mNuz(t_n)- \widetilde{T}_u^n, \eb \right\rangle_h.
    \ed
    \end{small}
    Applying the algebraic identity to the LHS of \eqref{eqa_17} together with \eqref{eq_inner_Q} yields the following expansion:
    \begin{equation} \no
        \begin{aligned}
            \text{LHS} & = \iaQ{\ea} + \frac{1}{2\tau} \left[  \ilQ{\be_{\mathbf{Q}}^{n+1}} - \ilQ{\be_{\mathbf{Q}}^{n}}  + \left\langle \tau^2 \mLQ \ea, \ea \right\rangle_h \right]      \\
                       & \quad + \iau{\eb} + \frac{1}{2\tau} \left[ \ilu{e_{u}^{n+1}}  - \ilu{e_{u}^{n}}  +\left\langle \tau^2 \mLu \delta_\tau  e_{u}^{n+1}, \eb \right\rangle_h  \right] \\
                       & = \! \frac{1}{2\tau} \left[  \ilQ{\be_{\mathbf{Q}}^{n+1}}\! - \!\ilQ{\be_{\mathbf{Q}}^{n}} + \!\ilu{e_{u}^{n+1}}  \!-\! \ilu{e_{u}^{n}}  \right]
            + \ibQ{\ea} \!+\!\ibu{\eb} .
        \end{aligned}
    \end{equation}
    Applying the Cauchy-Schwarz and Young's inequalities to the RHS of \eqref{eqa_17}, along with Lemma \ref{thm_full_lipschitz} and \eqref{eqft}, yields the following estimate:
    \begin{equation} \no
        \begin{aligned}
            \text{RHS} & \leq \!\frac{1}{4}\left( \ia{\ea} +\! \ia{\eb} \right) +    \ia{\mNQ^n -\!\mNQz(t_n)-\!\widetilde{ \mathbf{T}}_{\mathbf{Q}}^n} +\!\ia{\mNu^n-\!\mNuz(t_n)-\!\widetilde{T}_u^n}     \\
                       & \leq \frac{1}{4} \left( \|\ea\|_h^2 + \|\eb\|_h^2 \right) +   C (\tau+h^2)^2                                                                                                                    \\
                       & \quad +   C_5 \left( |e_s^n|^2 + \|\be_{\mathbf{Q}}^{n}\|_{H_h^1}^2 + \|e_{u}^{n}\|_{H_h^2}^2 \right) + C_6 \left( |e_s^n|^2 + \|\be_{\mathbf{Q}}^{n}\|_{H_h^2}^2 + \|e_{u}^{n}\|_{H_h^2}^2 \right) \\
                       & \leq \frac{1}{4} \left( \|\ea\|_h^2 + \|\eb\|_h^2 \right)
            + C \left( |e_s^n|^2 + \|\be_{\mathbf{Q}}^{n}\|_{H_h^2}^2 + \|e_{u}^{n}\|_{H_h^2}^2 \right)+ C (\tau+h^2)^2.
        \end{aligned}
    \end{equation}
    Substituting the above bounds  into  \eqref{eqa_17} and using  Lemma \ref{lem_norm_equivalence} yields the following  error energy inequality:
    \begin{align}
        % \begin{aligned}
              \frac{3}{4} ( \ibQ{\ea} &+ \ibu{\eb} ) + \frac{1}{2\tau} \left[  \ilQ{\be_{\mathbf{Q}}^{n+1}} - \ilQ{\be_{\mathbf{Q}}^{n}}   + \ilu{e_{u}^{n+1}}  - \ilu{e_{u}^{n}}  \right] \no\\
              \leq& C \left( |e_s^n|^2 + \|\be_{\mathbf{Q}}^{n}\|_{H_h^2}^2 + \|e_{u}^{n}\|_{H_h^2}^2 \right)+ C (\tau+h^2)^2.\label{eq_error_energy_recursive}
        % \end{aligned}
    \end{align}

    Next, we return to the error equation \eqref{eq_error_s_final3} for ${s}$.
    Multiplying  \eqref{eq_error_s_final3}  by $2{e}_s^{n+1}$, we can derive the following identity:
    \begin{equation} \label{eq_es_identity}
        \begin{aligned}
             & \frac{1}{\tau} \left( |{e}_s^{n+1}|^2 - |e_s^n|^2 + |{e}_s^{n+1} - e_s^n|^2 \right)
            = -2 {e}_s^{n+1} \Big[  \langle \mNQ^n, \ea \rangle_h + \langle \mNu^n, \eb \rangle_h \Big]                                                                \\
             & -\! 2{e}_s^{n+1} \Big[ \left\langle  \mNQ^n-\!\mNQz(t_n), \da  \right\rangle_h\! + \!\left\langle  \mNu^n-\mNuz(t_n), \db \right\rangle_h +\!T_{s}^n  \Big]=: \!\mathcal{J}_1 + \!\mathcal{J}_2.
        \end{aligned}
    \end{equation}
    For the term $\mathcal{J}_1$, applying the Cauchy-Schwarz and Young's inequalities, along with Lemma \ref{gn_bound0} and Theorem \ref{th4_1}, yields the following estimate:
    \begin{equation} \label{eq_I2_bound}
        \begin{aligned}
            \mathcal{J}_1 & \le 2 \mathcal{M}_1 |{e}_s^{n+1}| \left( \|\ea\|_h + \|\eb\|_h \right)                           \\
                          & \le 16 \mathcal{M}_1^2 |{e}_s^{n+1}|^2 + \frac{1}{8} \|\ea\|_h^2 + \frac{1}{8} \|\eb\|_h^2.
        \end{aligned}
    \end{equation}
    For the term $\mathcal{J}_2$, applying the Cauchy-Schwarz and Young's inequalities, along with Lemma \ref{thm_full_lipschitz}, yields the following estimate:
    \begin{equation} \label{eq_I1_bound}
        \begin{aligned}
            \mathcal{J}_2 & \le  2|{e}_s^{n+1}|\cdot   \left(C_5 \left( |e_s^n| + \|\be_{\mathbf{Q}}^{n}\|_{H_h^1} + \|e_{u}^{n}\|_{H_h^2}+ h^2 \right)\right.                                      \\
                          & \qquad \left.
            +C_6 \left( |e_s^n| + \|\be_{\mathbf{Q}}^{n}\|_{H_h^2} + \|e_{u}^{n}\|_{H_h^2}+ h^2 \right) +C(\tau+h^2) \right)                                                                        \\
                          & \le |{e}_s^{n+1}|^2+ C \left(   |e_s^n|^2 + \|\be_{\mathbf{Q}}^{n}\|_{H_h^1}^2 + \|\be_{\mathbf{Q}}^{n}\|_{H_h^2}^2+ \|e_{u}^{n}\|_{H_h^2}^2 + (\tau+h^2)^2  \right).
        \end{aligned}
    \end{equation}
    Substituting  \eqref{eq_I2_bound}--\eqref{eq_I1_bound} back into \eqref{eq_es_identity} and rearranging the equation yields the following estimate:
    % \begin{equation} \no
    %     \begin{aligned}
    %         \frac{1}{\tau} & \left( |{e}_s^{n+1}|^2 - |e_s^n|^2 \right) \le \left( 16  \mathcal{M}_1^2 + 1 \right) |{e}_s^{n+1}|^2 + C \left(   |e_s^n|^2 + \|\be_{\mathbf{Q}}^{n}\|_{H_h^1}^2 + \|\be_{\mathbf{Q}}^{n}\|_{H_h^2}^2+ \|e_{u}^{n}\|_{H_h^2}^2 + (\tau+h^2)^2  \right).
    %     \end{aligned}
    % \end{equation}
    % Multiplying both sides by $\tau$ and rearranging the equation, we obtain
    \begin{equation*}
\begin{aligned}
        &\left( 1 - \tau \left( 16  \mathcal{M}_1^2 + 1 \right) \right) |{e}_s^{n+1}|^2
        \le |e_s^n|^2 + \frac{1}{8} \tau\|\ea\|_h^2 + \frac{1}{8} \tau\|\eb\|_h^2\\&+ C \tau \left( |e_s^n|^2 + \|\be_{\mathbf{Q}}^{n}\|_{H_h^1}^2 + \|\be_{\mathbf{Q}}^{n}\|_{H_h^2}^2+ \|e_{u}^{n}\|_{H_h^2}^2 + (\tau+h^2)^2 \right).
  \end{aligned}
\end{equation*}
    Assuming that  $\tau$ is sufficiently small to satisfy $\tau \left( 16  \mathcal{M}_1^2 + 1 \right) \le \frac{1}{2}$ and dividing both sides by this coefficient yields the following recursive inequality for $|e_s^{n+1}|^2$:
    \begin{equation*}
        \begin{aligned}
            |{e}_s^{n+1}|^2 &\le|e_s^n|^2+ \frac{1}{4}\tau \|\ea\|_h^2 + \frac{1}{4}\tau \|\eb\|_h^2
            % & \left( 1 + 2\tau \left( 16  \mathcal{M}_1^2 + 1 \right)\right) |e_s^n|^2 + {C} \tau \left( |e_s^n|^2 +\|\be_{\mathbf{Q}}^{n}\|_{H_h^1}^2 + \|\be_{\mathbf{Q}}^{n}\|_{H_h^2}^2+ \|e_{u}^{n}\|_{H_h^2}^2 + (\tau+h^2)^2 \right),
            \\
              &+ C \tau \left( |e_s^n|^2 +\|\be_{\mathbf{Q}}^{n}\|_{H_h^1}^2 + \|\be_{\mathbf{Q}}^{n}\|_{H_h^2}^2+ \|e_{u}^{n}\|_{H_h^2}^2 + (\tau+h^2)^2 \right),
        \end{aligned}
    \end{equation*}
    where we have used the inequality  $\frac{1}{1-x} \le 1+2x$ for $x \in [0, 1/2]$.
    Combining the above inequality with \eqref{eq_error_energy_recursive}, we obtain the following unified recursive inequality:
    \begin{equation} \label{eq_unified_single}
        \begin{aligned}
             & \frac{1}{\tau} \left( \ilQ{\be_{\mathbf{Q}}^{n+1}} - \ilQ{\be_{\mathbf{Q}}^{n}}   + \ilu{e_{u}^{n+1}}  - \ilu{e_{u}^{n}}  + |e_s^{n+1}|^2 - |e_s^n|^2 \right) \\
             & \le C \left( |e_s^n|^2 + \|\be_{\mathbf{Q}}^{n}\|_{H_h^1}^2 + \|\be_{\mathbf{Q}}^{n}\|_{H_h^2}^2 + \|e_{u}^{n}\|_{H_h^2}^2 + (\tau + h^2)^2  \right).
        \end{aligned}
    \end{equation}

    Taking the discrete inner product of  \eqref{eq_error_s_final} with $\mLQ \be_{\mathbf{Q}}^{n}$
    with  $n$ replaced by $n-1$ and rearranging the terms
    yields the following equation:
    \begin{equation} \label{eq_H2_bound_shifted_detailed}
        \begin{aligned}
             & \left\langle  \mQ \delta_\tau \be_{\mathbf{Q}}^{n}, \mLQ \be_{\mathbf{Q}}^{n} \right\rangle_h +\! \|\mLQ \be_{\mathbf{Q}}^{n}\|_h^2 =\! \left\langle  \mNQ^{n-1} -\!\mNQz(t_{n-1}) -\widetilde{\mathbf{T}}_{\mathbf{Q}}^{n-1}, \mLQ \be_{\mathbf{Q}}^{n} \right\rangle_h.
        \end{aligned}
    \end{equation}
    Applying the self-adjointness of $\mQ$ and $\mLQ$ together with \eqref{eq_inner_Q}, we rewrite the first term on the LHS of \eqref{eq_H2_bound_shifted_detailed} as follows:
    \begin{align}
        % \begin{aligned}
            % \left\langle  \mQ\delta_\tau \be_{\mathbf{Q}}^{n}, \mLQ\be_{\mathbf{Q}}^{n} \right\rangle_h & =
            \left\langle  \mQ\mLQ\delta_\tau \be_{\mathbf{Q}}^{n}, \be_{\mathbf{Q}}^{n} \right\rangle_h
            &=\frac{1}{2\tau}\left(\left\langle \mQ \mLQ  \be_{\mathbf{Q}}^{n}, \be_{\mathbf{Q}}^{n} \right\rangle_h - \left\langle \mQ \mLQ  \be_{\mathbf{Q}}^{n-1}, \be_{\mathbf{Q}}^{n-1} \right\rangle_h\right.\no\\
                                & \quad \left. + \left\langle \mQ \mLQ  (\be_{\mathbf{Q}}^{n}-\be_{\mathbf{Q}}^{n-1}), \be_{\mathbf{Q}}^{n}-\be_{\mathbf{Q}}^{n-1} \right\rangle_h\right)\no
            \\&=\frac{1}{2\tau} \left( \icQ{\be_{\mathbf{Q}}^{n}} - \icQ{\be_{\mathbf{Q}}^{n-1}} \right) + \frac{\tau}{2} \icQ{\delta_\tau \be_{\mathbf{Q}}^{n}}.\label{eq_source_prev_step}
            % \tau \left(  \left\langle \mathcal{Q} (\tau \mLQ)\mLQ  (\delta_\tau \be_{\mathbf{Q}}^{n}),\delta_\tau \be_{\mathbf{Q}}^{n}\right\rangle_h\right).
        % \end{aligned}
    \end{align}
    Applying Young's inequality to the RHS of \eqref{eq_H2_bound_shifted_detailed}, along with Lemma \ref{thm_full_lipschitz} and \eqref{eqft}, yields the following bound:
    \begin{align}
        &\left\langle \mNQ^{n-1} -\mNQz(t_{n-1}) -\widetilde{\mathbf{T}}_{\mathbf{Q}}^{n-1}, \mLQ \be_{\mathbf{Q}}^{n} \right\rangle_h
        % \no\\& \le \frac{1}{2} \|\mNQ^{n-1} -\mNQz(t_{n-1}) -\widetilde{\mathbf{T}}_{\mathbf{Q}}^{n-1}\|_h^2 + \frac{1}{2} \|\mLQ \be_{\mathbf{Q}}^{n}\|_h^2
        \no \\&\le \frac{1}{2}C_5 \left( |e_s^{n-1}|^2 + \|\be_{\mathbf{Q}}^{n-1}\|_{H_h^1}^2 + \|e_{u}^{n-1}\|_{H_h^2}^2 \right) + C (\tau+h^2)^2+ \frac{1}{2} \|\mLQ \be_{\mathbf{Q}}^{n}\|_h^2.\label{eq_source_prev_step2}
    \end{align}
Substituting \eqref{eq_source_prev_step} and
\eqref{eq_source_prev_step2} into
\eqref{eq_H2_bound_shifted_detailed}, and using the norm-equivalence
estimate
$\|{\be_{\mathbf Q}^{n}}\|_{H_h^2}
\le
C_L\|\mathcal L_h \be_{\mathbf{Q}}^{n}\|_h$, we can obtain the following  inequality:
    \begin{align}
       & \frac{C_L^2}{\tau} \left( \icQ{\be_{\mathbf{Q}}^{n}} - \icQ{\be_{\mathbf{Q}}^{n-1}} \right)+\|\be_{\mathbf{Q}}^{n}\|_{H_h^2}^2 \no\\& \le   \frac{1}{2}C_5^2C_L^2 \left( |e_s^{n-1}|^2 + \|\be_{\mathbf{Q}}^{n-1}\|_{H_h^1}^2 + \|e_{u}^{n-1}\|_{H_h^2}^2 \right) + C(\tau+h^2)^2.\label{eq_H2_full_bound_equiv}
    \end{align}
    % yields the following recursive inequality:
    % \begin{equation} \label{eq_H2_full_bound}
    %     \frac{1}{\tau} \left( \icQ{\be_{\mathbf{Q}}^{n}} - \icQ{\be_{\mathbf{Q}}^{n-1}} \right)+\|\mLQ \be_{\mathbf{Q}}^{n}\|_h^2 \le   \frac{1}{2}C_5^2 \left( |e_s^{n-1}|^2 + \|\be_{\mathbf{Q}}^{n-1}\|_{H_h^1}^2 + \|e_{u}^{n-1}\|_{H_h^2}^2 \right) + C (\tau+h^2)^2.
    % \end{equation}
    Combining \eqref{eq_unified_single}  with \eqref{eq_H2_full_bound_equiv} and balancing the coefficients of the error terms, we obtain the final error evolution inequality:
    \begin{align*}
        % \begin{aligned}
             & \frac{1}{\tau} \left( |{e}_s^{n+1}|^2 - |e_s^n|^2 \right) + \frac{CC_L^2}{\tau} \left[ \icQ{\be_{\mathbf{Q}}^{n}} - \icQ{\be_{\mathbf{Q}}^{n-1}} \right]                                        \\
             & + \frac{1}{\tau} \left[ \|\be_{\mathbf{Q}}^{n+1}\|_{\mathcal{L}_h}^2 - \|\be_{\mathbf{Q}}^{n}\|_{\mathcal{L}_h}^2 \right]
            + \frac{1}{\tau} \left[ \|e_{u}^{n+1}\|_{\mathcal{D}_h}^2 - \|e_{u}^{n}\|_{\mathcal{D}_h}^2 \right]                                                         \no                                          \\
             & \leq C (\tau + h^2)^2 + C \left( |e_s^n|^2 + |e_s^{n-1}|^2 +\|\be_{\mathbf{Q}}^{n}\|_{H_h^1}^2+ \|\be_{\mathbf{Q}}^{n-1}\|_{H_h^1}^2 + \|e_{u}^{n}\|_{H_h^2}^2 + \|e_{u}^{n-1}\|_{H_h^2}^2 \right).\no
        % \end{aligned}
    \end{align*}

    For the sake of brevity, we define the following discrete energy functional $\mathcal{Y}^n$ as
    \ba \label{eqa_20}
    \mathcal{Y}^n \coloneqq  \|\be_{\mathbf{Q}}^{n}\|_{\mathcal{L}_h}^2  + \|e_{u}^{n}\|_{\mathcal{D}_h}^2+|e_s^n|^2.
    \ed
    Then the final error inequality can be  rewritten in terms of  $\mathcal{Y}^n$ as follows:
    \begin{equation*}
        \frac{\left(\mathcal{Y}^{n+1}+CC_L^2\icQ{\be_{\mathbf{Q}}^{n}}\right) - \left(\mathcal{Y}^n+CC_L^2\icQ{\be_{\mathbf{Q}}^{n-1}}\right)}{\tau} \le C (\tau + h^2)^2 + C (\mathcal{Y}^n + \mathcal{Y}^{n-1}).
    \end{equation*}
By the initial error definition \eqref{eqcv}, along with \eqref{eq_initial_error_bound},  we have
\begin{equation}\label{eqcf}
    \mathcal Y^0
    =|e_s^0|^2
    \le Ch^4,
    \qquad
    \|\be_{\mathbf Q}^0\|_{\mathcal Q_{\mathcal L}^2}^2=0.
\end{equation}

    Using \eqref{eqcf} and summing the above inequality over $n$ from $1$ to $m$ yields the following  energy inequality:
    \ba\label{eqa_21}
    \mathcal{Y}^{m+1} +CC_L^2\icQ{\be_{\mathbf{Q}}^{m}}&\le \mathcal{Y}^1 +C \tau \sum_{k=1}^{m} (\tau + h^2)^2 + C \tau \sum_{k=1}^{m} (\mathcal{Y}^k + \mathcal{Y}^{k-1})\\
    &\le \mathcal{Y}^1 + C T (\tau + h^2)^2 + 2C \tau \sum_{k=0}^{m} \mathcal{Y}^k.
    \ed
    Taking $n=0$ in \eqref{eq_unified_single} yields the following bound for the  error $\mathcal{Y}^1$:
    \ba \label{eqa_22}
    \mathcal{Y}^1 \le C(\tau (\tau + h^2)^2+h^4).
    \ed
    Substituting \eqref{eqa_22} in \eqref{eqa_21} and applying the discrete Gronwall lemma yields the following bound for $\mathcal{Y}^m$:
    % \begin{small}
            \ba \label{eqa_24}
    \mathcal{Y}^{m+1} &\le \left( \mathcal{Y}^1 + C T (\tau + h^2)^2 \right) \exp\left( \sum_{k=0}^{m} 2C\tau \right) \\
    &\le \left( C(\tau+  T) (\tau + h^2)^2+Ch^4 \right) e^{2C(T+\tau)}\\
    &\le C(T,C_5,C_6,C_L,\mathcal{M}_1) (\tau^2 + h^4).
    \ed
    % \end{small}
% \end{proof}

By virtue of the norm equivalence, we can obtain the following error estimate:
\[
\|\be_{\mathbf Q}^{n}\|_{H_h^1}
+\|e_u^n\|_{H_h^2}
+|e_s^n|
\le C_{err}(T,C_5,C_6,\mathcal{M}_1,C_L)(\tau+h^2),
\quad 0\le n\le m+1.
\]
Taking $h$ and $\tau$ sufficiently small, we have
\[
s^n=s(t_n)+e_s^n
\ge -M_s-C_{err}(\tau+h^2)\ge -M_s-\frac12>-M_s-1=-C_s,
\quad 0\le n\le m+1.
\]
% Taking $h$ and $\tau$ sufficiently small such that
% \[
% C(\tau+h^2)\le \frac12,
% \]
% we obtain
% \[
% s^n\ge -M_s-\frac12>-M_s-1=-C_s,
% \qquad 0\le n\le m+1.
% \]
Therefore, the bootstrap assumption is improved and can be continued to the
next time level. By induction, it holds for all
$0\le n\le N_T$.
\section{Numerical experiments}\label{section6}
In this section, we present numerical experiments in two and three dimensions to verify the theoretical results. The tests are designed to examine the temporal convergence rate, the modified energy dissipation law, structure-preserving properties,  and the ability of the method to capture the self-assembly dynamics of SmA liquid crystals.

Following \cite{shi2025modified}, we consider
$\vT<\vT_2^*<\vT_1^*$, corresponding to a deep
quench into the stable SmA regime.
Specifically, we set $a_1=a_2=1$, $\vT=-1$, $\vT_1^*=1$, and $\vT_2^*=0$,
which yields $\A=-2$ and $a=-1$.
The other parameters are set to
\ba \no
L_d=\!2\pi,~ K=\!0.1,~ \B=\!0,~ \C=\!4,~ b=\!0,~ c=\!3,~\! q=\!4,~ B_0 \!= \!1 \times 10^{-3},~ \!s_+ \!=\! \sqrt{-2\A/\C}.
\ed
We then take $\zeta(x)=e^x$ and set $\kappa=3$. We use uniform spatial grids with mesh sizes $h=2\pi/128$ in two dimensions (2D) and $h=2\pi/64$ in three dimensions (3D).

Throughout this section, $\{\xi_i(\mathbf{x})\}_{i=1}^5$ denote
mutually independent random variables uniformly distributed on
$[-0.1,0.1]$ at each grid point $\mathbf{x}$, where
$\mathbf{x}=(x,y)$ in 2D  and $\mathbf{x}=(x,y,z)$ in 3D.
\subsection{Convergence tests}
% According to \cite{shi2025modified}, the temperature-driven states of the smectic liquid crystal are characterized by two nondimensionalized reduced temperatures: $\A = a_1(\vT - \vT_1^*)$ for the I-N transition, and $a = a_2(\vT - \vT_2^*)$ for the N-S transition, where the critical temperatures satisfy $\vT_1^* > \vT_2^*$.
% Theoretical analysis reveals that the nematic phase loses stability via a pitchfork bifurcation at $\vT = \vT_2^*$, and layered smectic structures only emerge when $\vT < \vT_2^*$.
In this test, we evaluate the convergence rates of the proposed scheme by computing the discrete errors for $\mathbf{Q}$, $u$ and $s$ in various norms.
For the temporal convergence test, we set the final time to $T=1$ and take $\tau=2^{-k}\tau_1$, where $\tau_1=2^{-4}$ and $k=0,1,\ldots,7$.
The benchmark solution at $T=1$ is computed using the time step $\tau=2^{-9}\tau_1$.
For the spatial convergence test, we set the final time to $T=1$, fix the time step at $\tau=10^{-3}$, and take $h=2^{-k}h_1$, where $h_1=2^{-5}$ and $k=0,1,\ldots,4$. The benchmark solution at $T=1$ is computed using the finer mesh size $h=2^{-5}h_1=2^{-10}$ with the same time step.

The initial conditions are set as follows:
\begin{equation*}
\begin{aligned}
\mathbf{Q}^0(x,y) &= S \left( \mathbf{n}\mathbf{n}^T - \frac{1}{2}\mathbf{I} \right),\quad \mathbf{n} = (\cos  q(x+y), \sin  q(x+y))^T,
\end{aligned}
\end{equation*}
where  $S$ is set to $ 0.25$ and  $u$ is initialized by $u^0(x,y) = \cos( q x)$.
\begin{table}[htbp]
    \centering
    \small
    \setlength{\tabcolsep}{3pt}
    \begin{tabular}{lccccc}
        \toprule
        \multirow{2}{*}{$\tau$} & \multicolumn{2}{c}{$\mathbf{Q}$} & \multicolumn{2}{c}{$u$} & $s$ \\
        \cmidrule(lr){2-3} \cmidrule(lr){4-5} \cmidrule(lr){6-6}
        & $\ell^2$-norm & $H_h^1$-norm & $\ell^2$-norm & $H_h^2$-norm & absolute error \\
        \midrule
        $2^{-4}$  & $1.56\mathrm{e}{-3}$ (--) & $1.70\mathrm{e}{-2}$ (--) & $3.78\mathrm{e}{-2}$ (--) & $1.00\mathrm{e}{0}$ (--) & $6.25\mathrm{e}{-6}$ (--) \\
        $2^{-5}$  & $6.72\mathrm{e}{-4}$ (1.22) & $7.35\mathrm{e}{-3}$ (1.21) & $1.79\mathrm{e}{-2}$ (1.08) & $4.79\mathrm{e}{-1}$ (1.07) & $4.20\mathrm{e}{-6}$ (0.57) \\
        $2^{-6}$  & $3.03\mathrm{e}{-4}$ (1.15) & $3.30\mathrm{e}{-3}$ (1.16) & $7.90\mathrm{e}{-3}$ (1.18) & $2.13\mathrm{e}{-1}$ (1.17) & $2.46\mathrm{e}{-6}$ (0.77) \\
        $2^{-7}$  & $1.43\mathrm{e}{-4}$ (1.09) & $1.54\mathrm{e}{-3}$ (1.10) & $3.48\mathrm{e}{-3}$ (1.18) & $9.37\mathrm{e}{-2}$ (1.19) & $1.32\mathrm{e}{-6}$ (0.89) \\
        $2^{-8}$  & $6.92\mathrm{e}{-5}$ (1.04) & $7.43\mathrm{e}{-4}$ (1.05) & $1.59\mathrm{e}{-3}$ (1.13) & $4.26\mathrm{e}{-2}$ (1.14) & $6.83\mathrm{e}{-7}$ (0.96) \\
        $2^{-9}$  & $3.40\mathrm{e}{-5}$ (1.02) & $3.64\mathrm{e}{-4}$ (1.03) & $7.57\mathrm{e}{-4}$ (1.07) & $2.02\mathrm{e}{-2}$ (1.08) & $3.45\mathrm{e}{-7}$ (0.98) \\
        $2^{-10}$ & $1.69\mathrm{e}{-5}$ (1.01) & $1.80\mathrm{e}{-4}$ (1.01) & $3.68\mathrm{e}{-4}$ (1.04) & $9.80\mathrm{e}{-3}$ (1.04) & $1.73\mathrm{e}{-7}$ (0.99) \\
        $2^{-11}$ & $8.41\mathrm{e}{-6}$ (1.01) & $8.98\mathrm{e}{-5}$ (1.01) & $1.82\mathrm{e}{-4}$ (1.02) & $4.83\mathrm{e}{-3}$ (1.02) & $8.69\mathrm{e}{-8}$ (1.00) \\
        \bottomrule
    \end{tabular}
    \caption{Temporal errors and convergence rates for $\mathbf{Q}$, $u$, and $s$ in various norms.}
    \label{tab_error_all_variables}
\end{table}
\begin{table}[htbp]
    \centering
    \small
    \setlength{\tabcolsep}{3pt}
    \begin{tabular}{lccccc}
        \toprule
        \multirow{2}{*}{$h$}
        & \multicolumn{2}{c}{$\mathbf{Q}$}
        & \multicolumn{2}{c}{$u$}
        & $s$ \\
        \cmidrule(lr){2-3}
        \cmidrule(lr){4-5}
        \cmidrule(lr){6-6}
        & $\ell^2$-norm
        & $H_h^1$-norm
        & $\ell^2$-norm
        & $H_h^2$-norm
        & absolute error \\
        \midrule
        $2^{-5}$
        & $1.38\mathrm{e}{-2}$ (--)
        & $1.11\mathrm{e}{-1}$ (--)
        & $2.75\mathrm{e}{-3}$ (--)
        & $2.58\mathrm{e}{-1}$ (--)
        & $1.55\mathrm{e}{-3}$ (--) \\

        $2^{-6}$
        & $3.58\mathrm{e}{-3}$ (1.94)
        & $3.01\mathrm{e}{-2}$ (1.88)
        & $7.77\mathrm{e}{-4}$ (1.82)
        & $8.67\mathrm{e}{-2}$ (1.57)
        & $3.33\mathrm{e}{-4}$ (2.22) \\

        $2^{-7}$
        & $9.03\mathrm{e}{-4}$ (1.99)
        & $7.69\mathrm{e}{-3}$ (1.97)
        & $2.00\mathrm{e}{-4}$ (1.96)
        & $2.32\mathrm{e}{-2}$ (1.90)
        & $7.83\mathrm{e}{-5}$ (2.09) \\

        $2^{-8}$
        & $2.26\mathrm{e}{-4}$ (2.00)
        & $1.93\mathrm{e}{-3}$ (1.99)
        & $5.03\mathrm{e}{-5}$ (1.99)
        & $5.91\mathrm{e}{-3}$ (1.98)
        & $1.93\mathrm{e}{-5}$ (2.02) \\

        $2^{-9}$
        & $5.66\mathrm{e}{-5}$ (2.00)
        & $4.84\mathrm{e}{-4}$ (2.00)
        & $1.26\mathrm{e}{-5}$ (2.00)
        & $1.48\mathrm{e}{-3}$ (1.99)
        & $4.79\mathrm{e}{-6}$ (2.01) \\
        \bottomrule
    \end{tabular}
    \caption{Spatial errors and convergence rates for $\mathbf{Q}$, $u$,
    and $s$ in various norms.}
    \label{tab_spatial_error_all_variables}
\end{table}

Tables~\ref{tab_error_all_variables} and \ref{tab_spatial_error_all_variables} present the temporal and spatial errors, respectively, together with the corresponding convergence rates for $\mathbf{Q}$, $u$, and $s$ measured in various norms.
These results are consistent with Theorem \ref{thm_global_error}.
% The comparison shows that the relaxation step suppresses the numerical tracking error of $s$ and improves its accuracy by several orders of magnitude, which reveals the significant numerical advantage of the proposed relaxed step.
% Meanwhile, the error magnitudes and convergence rates for $\mathbf{Q}$ and $u$ remain virtually unaltered, which indicates that the relaxation step does not compromise the overall accuracy of the scheme for the primary physical variables.

\subsection{Structure-preserving properties and self-assembly dynamics}
 In this test, we validate the structure-preserving properties of the proposed
scheme and illustrate its effectiveness in capturing the self-assembly
dynamics of the SmA phase.
The time step size is  set to $\tau=2^{-5}$ and the final time to $T=50$.
To trigger the phase transition and topological defect evolution, the components of the nematic tensor $\mathbf{Q}$ are initialized with random noise.
The initial conditions are set as follows:
\begin{equation} \label{eqb13}
    \mathbf{Q}^0(x, y) = \begin{pmatrix}
        \xi_1(x, y) & \xi_2(x, y) \\
        \xi_2(x, y) & -\xi_1(x, y)
    \end{pmatrix}, \quad
    u^0(x, y) = \cos( q x).
\end{equation}
% where $\xi_1(x,y)$ and $\xi_2(x,y)$ are independent random variables
% uniformly distributed on $[-0.5,0.5]$.
 \begin{figure}[htbp]
    \centering
		 \includegraphics[width=0.8\textwidth]{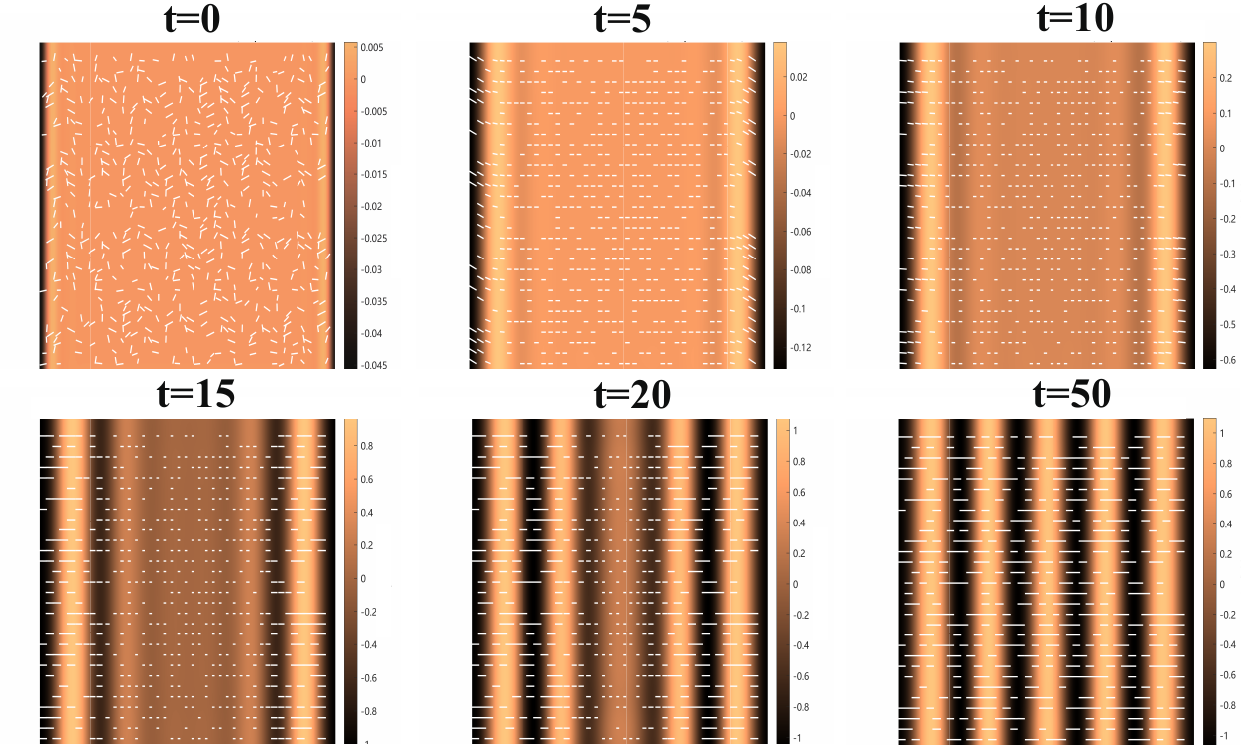}
	 \caption{Evolution of the SmA system  at $t=0, 5, 10, 15, 20$, and $t=50$. The local director field is visualized by white segments superimposed on the density field $u$, with warmer colors indicating the high-density smectic layers.}
		  \label{smectic}
 \end{figure}

Figure \ref{smectic} illustrates  a comprehensive visualization of the self-assembly process from an initial disordered state to a highly organized smectic phase.
The coarsening dynamics and the
eventual alignment of the director field indicate that the proposed scheme
captures the phase-ordering behavior of the SmA system.
% At the initial stage ($t=0$), the system starts with a random distribution of both the local director  and positional order, representing a nearly isotropic configuration.
% As the phase ordering  proceeds  ($t=5$ to $t=50$), the system undergoes a coarsening process, where the initially small and disordered domains merge and grow into larger, more coherent smectic layers.
%  By $t=50$, the system reaches a  steady state, where the density stripes exhibit high contrast, and the director field is uniformly oriented along the $x$-axis.

  \begin{figure}[htbp]
    \centering
		 \includegraphics[width=0.8\textwidth]{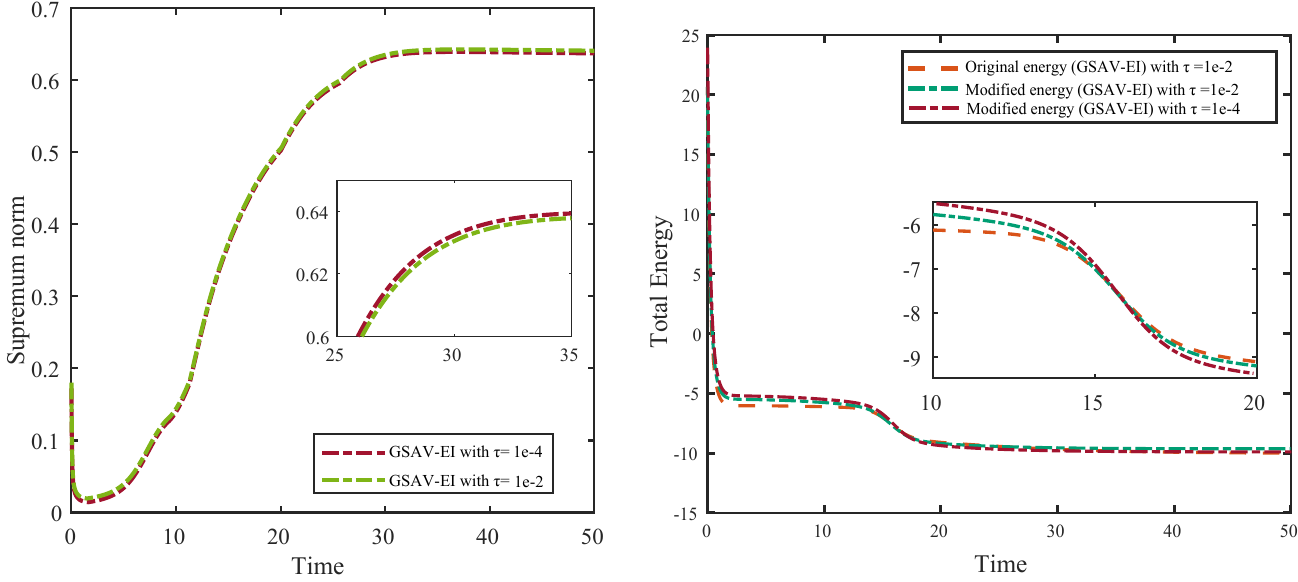}
	 \caption{Temporal evolution of the discrete $L^\infty$-norm (left) and the original and modified energies (right) for the SmA system under various time step sizes. }
		 \label{fig:energy_comparison}
 \end{figure}
Figure \ref{fig:energy_comparison}   illustrates the temporal evolution of the discrete $L^\infty$-norm and the energy trajectories.
Following an initial transient and rapid growth, the curves in the left
panel level off, indicating that $\|\mathbf Q\|_\infty$ remains bounded
throughout the simulation.
% %  Notably, the solution with $\tau = 10^{-2}$ does not exhibit any unbounded growth or numerical instability.
The right panel demonstrates that the original and modified energies for all tested cases exhibit a monotonic decay over time, which is consistent with Theorem \ref{them3_1}.
% The modified energy of the GSAV-EI scheme without the relaxation step (blue dotted line) significantly deviates from the original energy (orange solid line) and modified energy (red dashed line) over time, exhibiting a non-physical excessive dissipation.  In striking contrast, by incorporating the relaxation step, the modified energy closely tracks the original energy (red solid line). This strong alignment confirms the effectiveness of the relaxation mechanism.

  \begin{figure}[htbp]
		 \includegraphics[width=\textwidth]{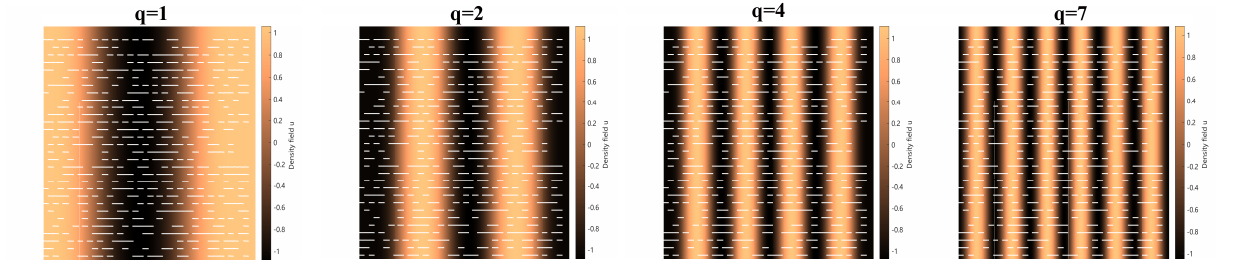}
	 \caption{Evolution of the SmA system at $t=50$ under various wave numbers $q$. }
            \label{fig:parameter_comparison}
 \end{figure}
 Figure \ref{fig:parameter_comparison} shows that the spatial frequency of the layers increases as $q$ varies over $1,2,4$, and $7$.
  This is consistent with the physical interpretation of $q$ as the intrinsic wave number dictating the preferred layer spacing in the smectic phase.
  These results demonstrate that our numerical scheme can accurately capture the influence of key physical parameters on the self-assembled structures in the SmA phase.

  \begin{figure}[htbp]
		 \includegraphics[width=\textwidth]{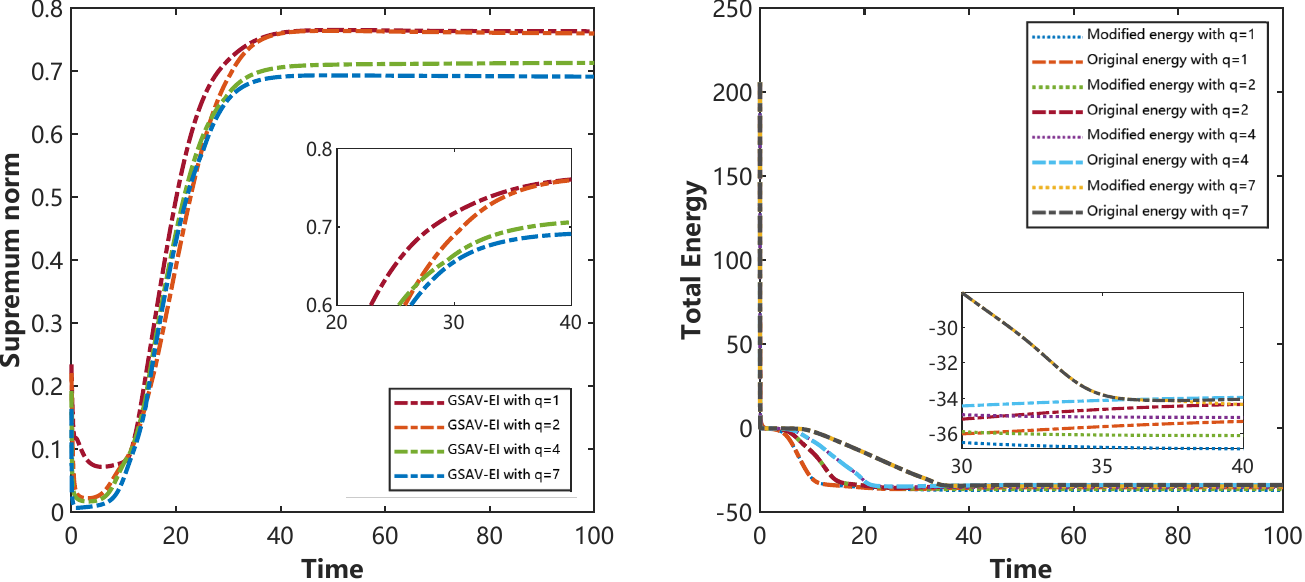}
	 \caption{Time evolution of the discrete $L^\infty$-norm (left) and the original and modified energies (right) for the SmA system with different values of the wave number $q$.}
		 \label{fig:energy_robustness}
 \end{figure}
Figure \ref{fig:energy_robustness} illustrates the temporal evolution of the discrete $L^\infty$-norm and the energy trajectories under varying wave numbers $q$.
The left panel shows that the $L^\infty$-norm of the solution remains uniformly bounded across all tested parameter combinations.
The right panel demonstrates that the original and modified energies for all tested cases exhibit a consistent, monotonic decay over time, with the modified energy closely tracking the original energy. This confirms that the proposed scheme  accurately captures the energy dissipation dynamics across a wide range of physical parameters.
\subsection{Adaptive time-stepping test}
In this test, we apply the proposed scheme together with an adaptive
time-stepping strategy, in which the time-step size $\tau_{n+1}$ is
dynamically adjusted according to the rate of change of the energy
functional \cite{Zhang2011}:
\begin{equation} \label{eq_adaptive_tau}
    \tau_{n+1} = \max \left\{ \tau_{\min}, \frac{\tau_{\max}}{\sqrt{1 + \varrho  |d_t \mathcal{E}_h^n|^2}} \right\},
\end{equation}
where $d_t\mathcal E_h^n=\frac{\mathcal E_h^{n+1}-\mathcal E_h^n}{\tau_n}$ and $\varrho > 0$ is a constant parameter.
The minimal and maximal time step sizes are set to $\tau_{\min} = 0.01$ and $\tau_{\max} = 0.2$, respectively, and we take $\varrho = 10^5$.
The initial conditions are set as follows:
\begin{equation*}
    \mathbf{Q}^0(x, y) = \begin{pmatrix}
        \xi_1(x, y) & \xi_2(x, y) \\
        \xi_2(x, y) & -\xi_1(x, y)
    \end{pmatrix}, \quad
    u^0(x,y)=\frac12\cos( qx)+\frac12\cos( qy).
\end{equation*}
% where $\xi_1(x, y)$ and $\xi_2(x, y)$ are independently generated random numbers  distributed in  $[-0.5, 0.5]$.

  \begin{figure}[htbp]
    \centering
		 \includegraphics[width=0.8\textwidth]{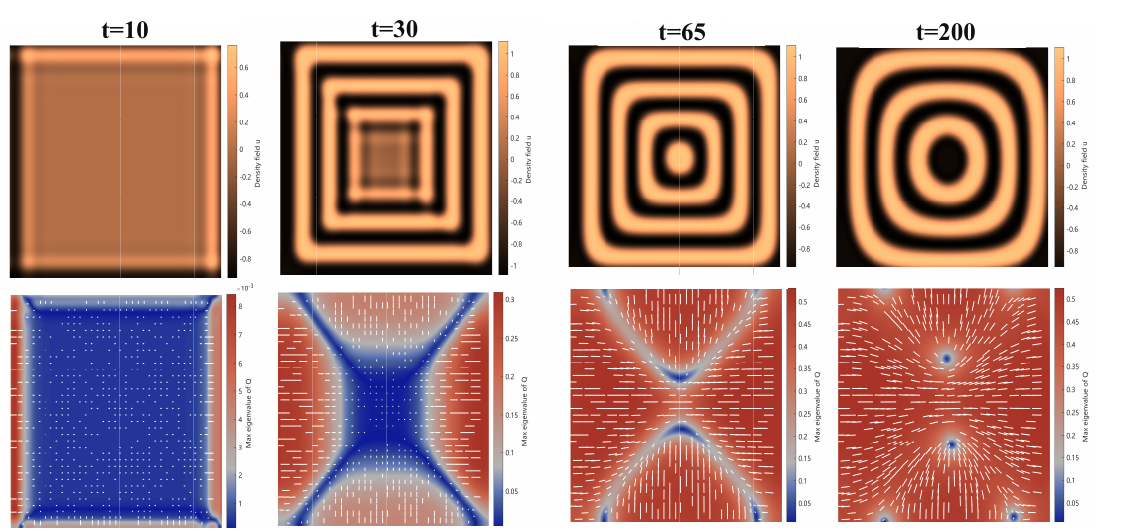}
	 \caption{Dynamic evolution of the SmA  system at $t=10, 30, 65$, and $200$. The top row displays the positional order $u$, while the bottom row shows the maximum eigenvalue of the $\mathbf{Q}$-tensor, superimposed with the director field $\mathbf{n}$ (white lines). }
		  \label{fig:smectic_evolution}
 \end{figure}
Figure~\ref{fig:smectic_evolution} illustrates the SmA self-assembly process with adaptive time-stepping, showing the positional order $u$ and the director field $\mathbf{n}$ (white lines) in the top and bottom rows, respectively.
 Starting from a disordered state, the system undergoes smectic layer
formation, defect evolution, and domain coarsening before approaching
a steady state with well-defined layers and a nearly uniform director
field.
This evolution illustrates the capability of the adaptive time-stepping
strategy to efficiently resolve the complex dynamics of the SmA phase.
  \begin{figure}[htbp]
    \centering
		 \includegraphics[width=0.8\textwidth]{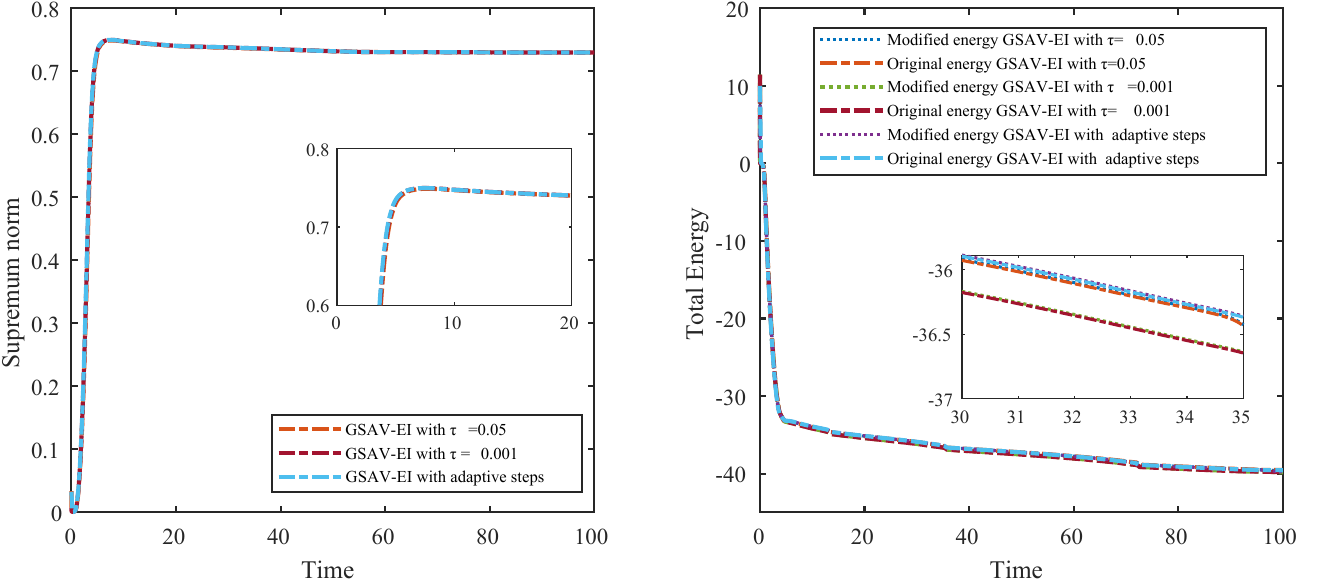}
	 \caption{Temporal evolution of the discrete $L^\infty$-norm (left) and the original
and modified energies (right) for the SmA system under the adaptive
GSAV-EI scheme and fixed-step GSAV-EI schemes with $\tau=0.05$ and
$\tau=0.001$.}
		  \label{fig:adaptive_energy}
 \end{figure}

Figure \ref{fig:adaptive_energy} presents the temporal evolution of the discrete $L^\infty$-norm and the energy trajectories for the SmA system under different numerical schemes and time-stepping strategies.
% The left panel shows that all tested schemes maintain a uniformly bounded $L^\infty$-norm throughout the simulation.
% Notably, the GSAV-LRI scheme (gray dashed line) exhibits a slightly lower steady-state norm level compared to the GSAV-EI scheme (blue solid line).
% This discrepancy arises because the left-rectangle rule in LRI truncates fine dynamics, whereas ETD's integration preserves high-fidelity trajectories.
The right panel demonstrates that the original and modified energies for all tested schemes exhibit a consistent, monotonic decay over time.
Together with the step size profiles in Fig. \ref{fig:adaptive_steps_profile}, the results
demonstrate that the adaptive strategy can  refine the step size during rapid topological transitions and enlarge it during smooth coarsening stages, thereby significantly reducing the overall computational cost without sacrificing physical fidelity.

 \begin{figure}[htbp]
    \centering
		 \includegraphics[width=0.8\textwidth]{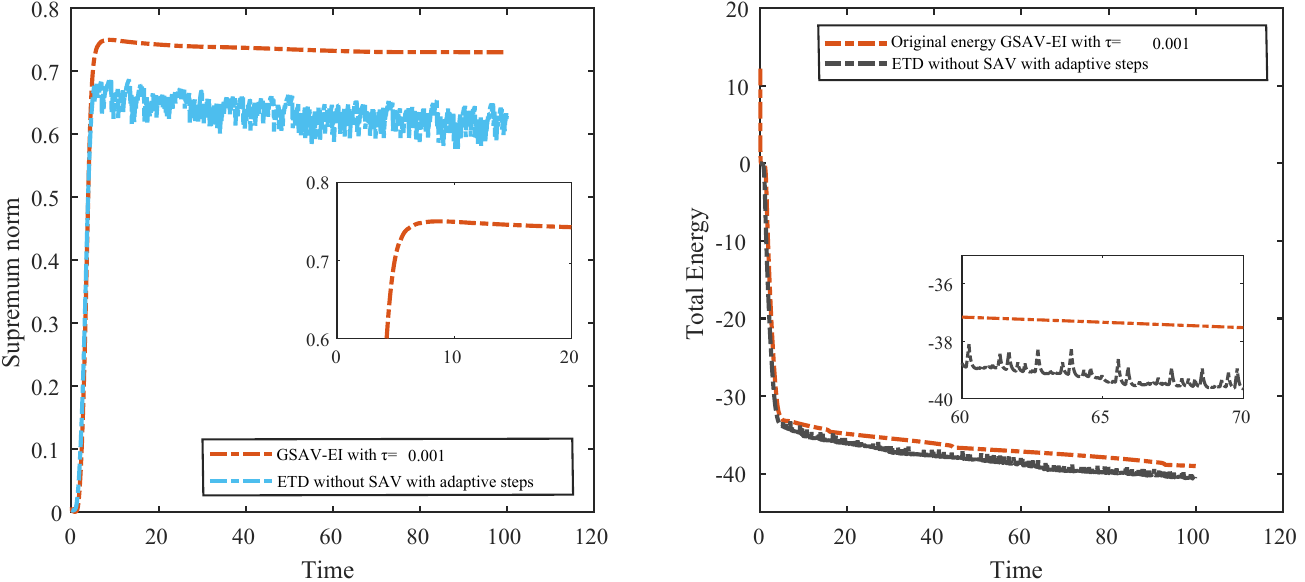}
	 \caption{Temporal evolution of the discrete $L^\infty$-norm (left) and the original and modified energies (right) for the SmA system, comparing the  GSAV-EI scheme against an adaptive  ETD scheme  without introducing SAV. }
		  \label{fig:stability_comparison}
 \end{figure}

Figure \ref{fig:stability_comparison} provides a direct comparison of the temporal evolution of the discrete $L^\infty$-norm and the energy trajectories for the SmA system under various numerical schemes.
% As shown in two panels,  the GSAV-EI   scheme yields  smooth norm evolutions and strictly dissipative energy trajectories.
In  contrast,  the standard ETD scheme without SAV (light blue dash-dotted line) suffers from severe high-frequency spurious oscillations and fails to maintain energy monotonicity, ultimately leading to numerical instability.
% Notably, as corroborated by the right panel, the energy of the GSAV-LRI scheme (gray dashed line) with adaptive time-stepping exhibits a slight, non-physical increase during $t \in [60, 70]$.
This comparison underscores the critical importance of the SAV mechanism
and  ETD integration in stabilizing the numerical solution  and preventing non-physical numerical blow-ups.
  \begin{figure}[htbp]
    \centering
		 \includegraphics[width=0.9\textwidth]{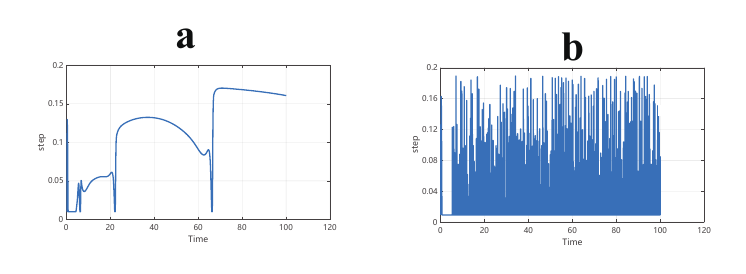}
	 \caption{Profiles of the adaptive time step sizes $\tau$ for  numerical schemes: (a) GSAV-EI, (b) the standard ETD scheme without SAV.}
		 \label{fig:adaptive_steps_profile}
 \end{figure}

 Figure \ref{fig:adaptive_steps_profile} presents the dynamic profiles of the adaptive time step sizes $\tau$ for different numerical schemes.
 The curve in panel (a)  shows a clear pattern of step size modulation that corresponds to the underlying physical dynamics, whereas panel (b) exhibits numerical instability.
 This behavior  further emphasizes the critical role of the SAV mechanism in stabilizing the numerical solution and enabling effective adaptive time-stepping during complex phase transition dynamics.
\subsection{Three-dimensional simulations}
For the 3D simulations,
we take $L_d=2 \pi$.
 The model parameters are chosen as
\begin{equation*}
\begin{aligned}
\A=-8,\quad \B=8,\quad \C=16,\quad a=-5,\quad b=0,\quad c=5,\quad
q=4,\quad K=0.1,
\end{aligned}
\end{equation*}
which yield $s_+=1$ and
drive the system into the SmA regime.
% The model parameters for the 3D simulations are chosen to ensure the system is deeply quenched into the smectic phase.
% Specifically, we set $\A = -8$, $\B = 8$, and $\C = 16$, which yields a stable uniaxial nematic background with a theoretical equilibrium scalar order parameter $s_+  = 1.0$. The parameters for the smectic energy are set to $a = -5$, $b = 0$, and $c = 5$, driving the symmetry-breaking phase transition. The wave number $q = 4$ and the elastic constant $K = 0.1$ are adopted to reliably resolve the layered smectic structures under reasonable 3D grid resolutions.
To initiate the phase transition and topological defect evolution, we
prescribe the  initial conditions:
\begin{equation*}
\begin{aligned}
    \mathbf{Q}^0(x, y, z) & = \begin{pmatrix}
        \xi_1(x, y, z) & \xi_2(x, y, z) & \xi_3(x, y, z) \\
        \xi_2(x, y, z) & \xi_4(x, y, z) & \xi_5(x, y, z) \\
        \xi_3(x, y, z) & \xi_5(x, y, z) & -\big(\xi_1(x, y, z) + \xi_4(x, y, z)\big)
    \end{pmatrix},\\
    u^0(x, y, z) &= \frac{1}{3} \big( \cos( q x) + \cos( q y) + \cos( q z) \big).
\end{aligned}
\end{equation*}
% where $\xi_i(x,y,z)$, $i=1,\ldots,5$, are independent uniform random variables on
% $[-0.1,0.1]$.

  \begin{figure}[htbp]
    \centering
		 \includegraphics[width=0.9\textwidth]{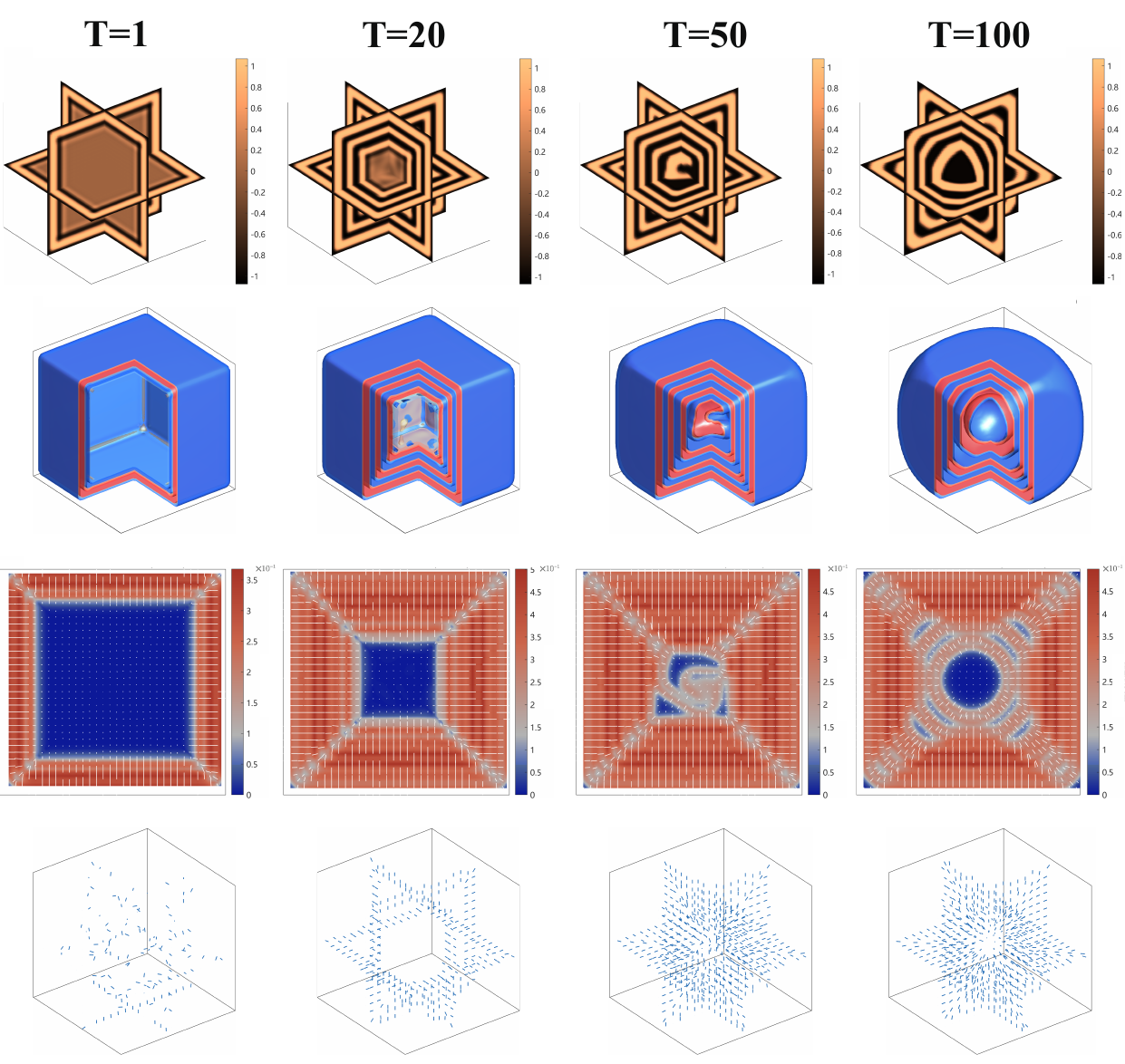}
	 \caption{Dynamic evolution of the 3D SmA phase at $t=1, 20, 50$, and $100$. From top to bottom, the rows display: (1) orthogonal 2D slices of $u$; (2) 3D cutaway isosurfaces of $u$ revealing internal nested layers; (3) the maximum eigenvalue of $\mathbf{Q}$ on the mid-plane, superimposed with the director field; and (4) the sparse 3D director field.}
		  \label{fig:3d_evolution}
 \end{figure}
Figure \ref{fig:3d_evolution} presents  a comprehensive visualization of the dynamic evolution of the 3D SmA phase at different time points, showcasing the intricate self-assembly process and the formation of complex smectic structures.
% The time evolution of the density field $u$ is visualized via orthogonal 2D slices in the first row and 3D cutaway isosurfaces in the second row, revealing the macroscopic smectic layers and their internal nested structures, respectively.
The cutaway view reveals the "onion-like" nested spherical layers, and the corresponding director field exhibits a  radial alignment pointing outward from a central $+1$ point defect.
% Furthermore, the third row overlays the local director field onto the mid-plane maximum eigenvalue of $\mathbf{Q}$ to reveal defect formation and order-disorder coexistence, while the fourth row illustrates complex spatial orientations via a sparse 3D director field.
%  This complex evolution vividly demonstrates the capability of our numerical scheme to accurately capture the intricate dynamics of 3D smectic layer formation and topological defect evolution.

  \begin{figure}[htbp]
    \centering
		 \includegraphics[width=0.8\textwidth]{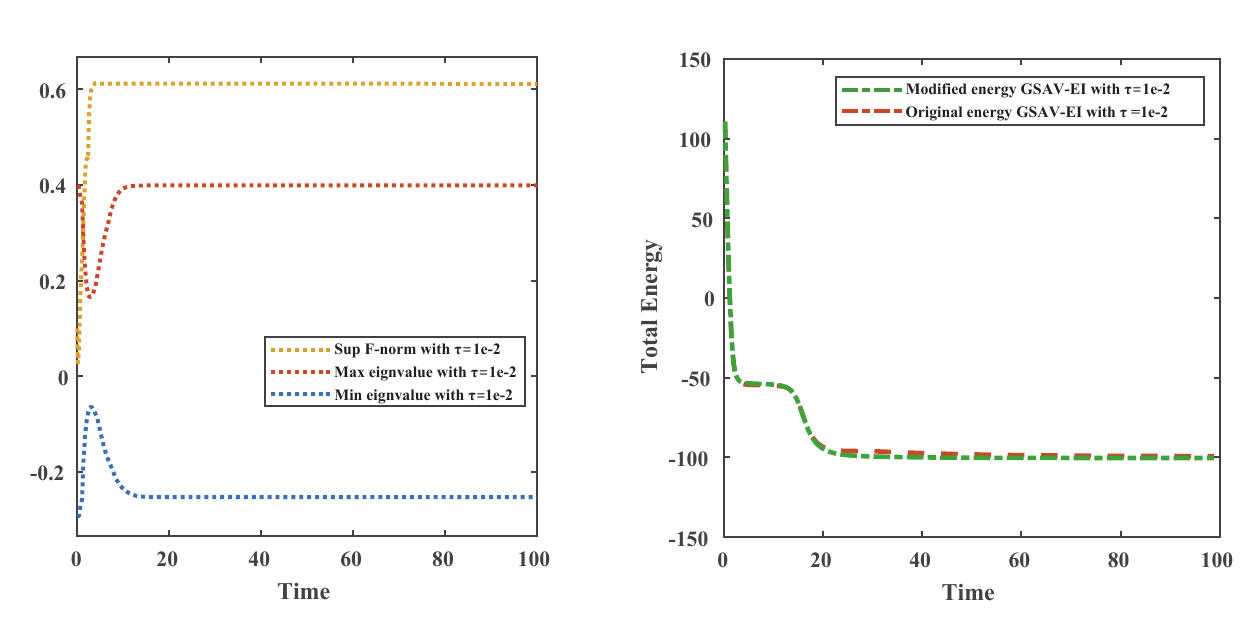}
	 \caption{Temporal evolution of the discrete $L^\infty$-norm of
$\mathbf{Q}$ and its maximum and minimum eigenvalues (left), together
with the original and modified energies (right), computed with
$\tau=0.01$.}
		  \label{fig:norm_energy_accuracy}
 \end{figure}

Figure \ref{fig:norm_energy_accuracy}  presents  the temporal evolution of the discrete $L^\infty$-norm of $\mathbf{Q}$ and its maximum and minimum eigenvalues, along with the original and modified energies for the 3D SmA system.
The results demonstrate that the discrete $L^\infty$-norm of $\mathbf{Q}$ remains uniformly bounded, and the eigenvalues of $\mathbf{Q}$ stay within the physically meaningful range $[-1/3,2/3]$, confirming the physical consistency of the numerical solution \cite{wang2021modelling}.

\section{Conclusion}\label{section7}
% {\color{red}In this paper, we have developed and analyzed a first-order, structure-preserving GSAV-EI scheme for the modified Landau-de Gennes model of Smectic-A liquid crystals.
% By reformulating the exponential-integrator discretization into a
% backward-Euler-type structure, the proposed method retains the advantages of the
% GSAV-EI framework while avoiding the severe time-step restriction required in
% the original analysis. We proved unconditional modified-energy stability and
%  derived
% optimal fully discrete error estimates.}

% {\color{red}To improve the physical consistency of long-time simulations, we introduced a
% relaxation step that reduces the discrepancy between the auxiliary modified
% energy and the original physical energy. The relaxation preserves the discrete
% energy dissipation law and substantially improves the accuracy of the auxiliary
% variable. Numerical experiments in two and three dimensions confirmed the
% theoretical convergence rate, demonstrated the monotone energy decay and
% uniform boundedness of the numerical solution, and captured the self-assembly
% dynamics and parameter dependence of SmA structures.}
In this paper, we developed a first-order, structure-preserving
numerical scheme for the SmA system. In particular, we  recast the time-discrete exponential integrator into a
quasi-implicit backward Euler-type form, thereby enabling a rigorous, fully discrete error analysis. By bridging these two frameworks, this reformulation opens up  new possibilities for the theoretical analysis of ETD-type schemes.

The present framework provides an efficient and robust numerical tool for
simulating SmA phase. Future work will include the development of
higher-order GSAV-EI schemes, the construction of methods that preserve the
physical eigenvalue constraints of the Q-tensor, and extensions to more general
liquid-crystal hydrodynamic models.

% \appendix
% \section{Proof of Lemma \ref{lem:continuous_regularity}}\label{sea}
% \section{Proof of Lemma \ref{thm_full_lipschitz}}\label{seb}


%  \bibliographystyle{amsplain}
% AMSRefs bibliography converted from references.bib.
% 49 entries; citation keys are unchanged.
% AMSRefs bibliography converted from references.bib.
% 49 entries; citation keys are unchanged.


\begin{bibdiv}
\begin{biblist}

\bib{Ball03052015}{article}{
   author={Ball, J. M.},
   author={Bedford, S. J.},
   title={Discontinuous order parameters in liquid crystal theories},
   journal={Mol. Cryst. Liq. Cryst.},
   volume={612},
   date={2015},
   pages={1--23},
   doi={10.1080/15421406.2015.1030571},
}

\bib{ball2011orientability}{article}{
   author={Ball, J. M.},
   author={Zarnescu, A.},
   title={Orientability and energy minimization in liquid crystal models},
   journal={Arch. Ration. Mech. Anal.},
   volume={202},
   date={2011},
   pages={493--535},
   doi={10.1007/s00205-011-0421-3},
}

\bib{baskaran2013convergence}{article}{
   author={Baskaran, A.},
   author={Lowengrub, J. S.},
   author={Wang, C.},
   author={Wise, S. M.},
   title={Convergence analysis of a second order convex splitting scheme for the modified phase field
   crystal equation},
   journal={SIAM J. Numer. Anal.},
   volume={51},
   date={2013},
   pages={2851--2873},
   doi={10.1137/120880677},
}

\bib{biscari2007landau}{article}{
   author={Biscari, P.},
   author={Calderer, M. C.},
   author={Terentjev, E. M.},
   title={Landau--de Gennes theory of isotropic-nematic-smectic liquid crystal transitions},
   journal={Phys. Rev. E},
   volume={75},
   date={2007},
   pages={051707},
   doi={10.1103/PhysRevE.75.051707},
}

\bib{chen1976landau}{article}{
   author={Chen, J.-H.},
   author={Lubensky, T. C.},
   title={{Landau--Ginzburg} mean-field theory for the nematic to smectic-{C} and nematic to
   smectic-{A} phase transitions},
   journal={Phys. Rev. A},
   volume={14},
   date={1976},
   pages={1202--1207},
   doi={10.1103/PhysRevA.14.1202},
}

\bib{deGennes1974}{book}{
   author={de Gennes, P.-G.},
   title={The Physics of Liquid Crystals},
   publisher={Clarendon Press},
   place={Oxford},
   date={1974},
}

\bib{du2018stabilized}{article}{
   author={Du, Q.},
   author={Ju, L.},
   author={Li, X.},
   author={Qiao, Z.},
   title={Stabilized linear semi-implicit schemes for the nonlocal {Cahn--Hilliard} equation},
   journal={J. Comput. Phys.},
   volume={363},
   date={2018},
   pages={39--54},
   doi={10.1016/j.jcp.2018.02.023},
}

\bib{du2019}{article}{
   author={{D}u, Q.},
   author={Ju, L.},
   author={Li, X.},
   author={Qiao, Z.},
   title={Maximum principle preserving exponential time differencing schemes for the nonlocal
   {Allen--Cahn} equation},
   journal={SIAM J. Numer. Anal.},
   volume={57},
   date={2019},
   pages={875--898},
   doi={10.1137/18M118236X},
}

\bib{du2021}{article}{
   author={D{u}, Q.},
   author={Ju, L.},
   author={Li, X.},
   author={Qiao, Z.},
   title={Maximum bound principles for a class of semilinear parabolic equations and exponential
   time-differencing schemes},
   journal={SIAM Rev.},
   volume={63},
   date={2021},
   pages={317--359},
   doi={10.1137/19M1243750},
}

\bib{du1991numerical}{article}{
   author={Du, Q.},
   author={Nicolaides, R. A.},
   title={Numerical analysis of a continuum model of phase transition},
   journal={SIAM J. Numer. Anal.},
   volume={28},
   date={1991},
   pages={1310--1322},
   doi={10.1137/0728069},
}

\bib{E1997modeling}{article}{
   author={E, W.},
   title={Nonlinear continuum theory of smectic-{A} liquid crystals},
   journal={Arch. Ration. Mech. Anal.},
   volume={137},
   date={1997},
   pages={159--175},
   doi={10.1007/s002050050026},
}

\bib{eyre1998unconditionally}{article}{
   author={Eyre, D. J.},
   title={Unconditionally gradient stable time marching the {Cahn--Hilliard} equation},
   journal={MRS Online Proc. Libr.},
   volume={529},
   date={1998},
   pages={39--46},
   doi={10.1557/PROC-529-39},
}

\bib{fei2018isotropic}{article}{
   author={Fei, M.},
   author={Wang, W.},
   author={Zhang, P.},
   author={Zhang, Z.},
   title={On the isotropic--nematic phase transition for the liquid crystal},
   journal={Peking Math. J.},
   volume={1},
   date={2018},
   pages={141--219},
   doi={10.1007/s42543-018-0005-3},
}

\bib{feng2013stabilized}{article}{
   author={Feng, X.},
   author={Tang, T.},
   author={Yang, J.},
   title={Stabilized {Crank--Nicolson}/{Adams--Bashforth} schemes for phase field models},
   journal={East Asian J. Appl. Math.},
   volume={3},
   date={2013},
   pages={59--80},
   doi={10.4208/eajam.200113.220213a},
}

\bib{furihata2001stable}{article}{
   author={Furihata, D.},
   title={A stable and conservative finite difference scheme for the {Cahn--Hilliard} equation},
   journal={Numer. Math.},
   volume={87},
   date={2001},
   pages={675--699},
   doi={10.1007/PL00005429},
}

\bib{han2015microscopic}{article}{
   author={Han, J.},
   author={Luo, Y.},
   author={Wang, W.},
   author={Zhang, P.},
   author={Zhang, Z.},
   title={From microscopic theory to macroscopic theory: a systematic study on modeling for liquid
   crystals},
   journal={Arch. Ration. Mech. Anal.},
   volume={215},
   date={2015},
   pages={741--809},
   doi={10.1007/s00205-014-0792-3},
}

% \bib{hu2016disclination}{article}{
%    author={Hu, Y.},
%    author={Qu, Y.},
%    author={Zhang, P.},
%    title={On the disclination lines of nematic liquid crystals},
%    journal={Commun. Comput. Phys.},
%    volume={19},
%    date={2016},
%    pages={354--379},
%    doi={10.4208/cicp.210115.180515a},
% }

\bib{huangGlobalWellposednessDynamical2015}{article}{
   author={Huang, J.},
   author={Ding, S.},
   title={Global well-posedness for the dynamical {Q}-tensor model of liquid crystals},
   journal={Sci. China Math.},
   volume={58},
   date={2015},
   pages={1349--1366},
   doi={10.1007/s11425-015-4990-8},
}

\bib{huang2024errorestimateorderenergy}{article}{
   author={Huang, J.},
   author={Li, X.},
   author={Ji, G.},
   title={Error estimate for the first order energy stable scheme of {Q}-tensor nematic model},
   journal={arXiv preprint arXiv:2412.08027},
   date={2024},
}

\bib{izzo2020landau}{article}{
   author={Izzo, D.},
   author={De Oliveira, M. J.},
   title={Landau theory for isotropic, nematic, smectic-{A}, and smectic-{C} phases},
   journal={Liq. Cryst.},
   volume={47},
   date={2020},
   pages={99--105},
   doi={10.1080/02678292.2019.1631968},
}

\bib{Ji2020BDF2}{article}{
   author={Ji, G.},
   title={A {BDF2} energy-stable scheme for a general tensor-based model of liquid crystals},
   journal={East Asian J. Appl. Math.},
   volume={10},
   date={2020},
   pages={57--71},
   doi={10.4208/eajam.291018.070619},
}

\bib{jiang2022improving}{article}{
   author={Jiang, M.},
   author={Zhang, Z.},
   author={Zhao, J.},
   title={Improving the accuracy and consistency of the scalar auxiliary variable ({SAV}) method with
   relaxation},
   journal={J. Comput. Phys.},
   volume={456},
   date={2022},
   pages={110954},
   doi={10.1016/j.jcp.2022.110954},
}

\bib{ju2022generalized}{article}{
   author={Ju, L.},
   author={Li, X.},
   author={Qiao, Z.},
   title={Generalized {SAV}-exponential integrator schemes for {Allen--Cahn} type gradient flows},
   journal={SIAM J. Numer. Anal.},
   volume={60},
   date={2022},
   pages={1905--1931},
   doi={10.1137/21M1446496},
}

\bib{liu2007dynamics}{article}{
   author={Liu, C.},
   author={Shen, J.},
   author={Yang, X.},
   title={Dynamics of defect motion in nematic liquid crystal flow: modeling and numerical simulation},
   journal={Commun. Comput. Phys.},
   volume={2},
   date={2007},
   pages={1184--1198},
   doi={10.4208/cicp.2007.v2.p1184},
}

\bib{liu2025maximum}{article}{
   author={Liu, Y.},
   author={Quan, C.},
   author={Wang, D.},
   title={On the maximum bound principle and energy dissipation of exponential time differencing
   methods for the matrix-valued {Allen--Cahn} equation},
   journal={IMA J. Numer. Anal.},
   volume={45},
   date={2025},
   pages={3342--3377},
   doi={10.1093/imanum/drae090},
}

% \bib{liu2024novel}{article}{
%    author={Liu, Z.},
%    author={Zhang, Y.},
%    author={Li, X.},
%    title={A novel energy-optimized technique of {SAV}-based ({EOP-SAV}) approaches for dissipative
%    systems},
%    journal={J. Sci. Comput.},
%    volume={101},
%    date={2024},
%    pages={38},
%    doi={10.1007/s10915-024-02677-0},
% }

\bib{majumdar2010equilibrium}{article}{
   author={Majumdar, A.},
   title={Equilibrium order parameters of nematic liquid crystals in the {Landau--de Gennes} theory},
   journal={European J. Appl. Math.},
   volume={21},
   date={2010},
   pages={181--203},
   doi={10.1017/S0956792509990210},
}

% \bib{mcmillan1971simple}{article}{
%    author={McMillan, W. L.},
%    title={Simple molecular model for the smectic-{A} phase of liquid crystals},
%    journal={Phys. Rev. A},
%    volume={4},
%    date={1971},
%    pages={1238--1240},
%    doi={10.1103/PhysRevA.4.1238},
% }

\bib{pazy2012semigroups}{book}{
   author={Pazy, A.},
   title={Semigroups of Linear Operators and Applications to Partial Differential Equations},
   series={Applied Mathematical Sciences},
   volume={44},
   publisher={Springer},
   place={New York},
   date={1983},
   doi={10.1007/978-1-4612-5561-1},
}

\bib{pevnyi2014modeling}{article}{
   author={Pevnyi, M. Y.},
   author={Selinger, J. V.},
   author={Sluckin, T. J.},
   title={Modeling smectic layers in confined geometries: Order parameter and defects},
   journal={Phys. Rev. E},
   volume={90},
   date={2014},
   pages={032507},
   doi={10.1103/PhysRevE.90.032507},
}

\bib{Zhang2011}{article}{
   author={Qiao, Z.},
   author={Zhang, Z.},
   author={Tang, T.},
   title={An adaptive time-stepping strategy for the molecular beam epitaxy models},
   journal={SIAM J. Sci. Comput.},
   volume={33},
   date={2011},
   pages={1395--1414},
   doi={10.1137/100812781},
}

\bib{shen2011spectral}{book}{
   author={Shen, J.},
   author={Tang, T.},
   author={Wang, L.-L.},
   title={Spectral Methods: Algorithms, Analysis and Applications},
   series={Springer Series in Computational Mathematics},
   volume={41},
   publisher={Springer},
   place={Berlin, Heidelberg},
   date={2011},
   doi={10.1007/978-3-540-71041-7},
}

\bib{shen2018scalar}{article}{
   author={Shen, J.},
   author={Xu, J.},
   author={Yang, J.},
   title={The scalar auxiliary variable ({SAV}) approach for gradient flows},
   journal={J. Comput. Phys.},
   volume={353},
   date={2018},
   pages={407--416},
   doi={10.1016/j.jcp.2017.10.021},
}

\bib{shen2019new}{article}{
   author={{S}hen, J.},
   author={Xu, J.},
   author={Yang, J.},
   title={A new class of efficient and robust energy stable schemes for gradient flows},
   journal={SIAM Rev.},
   volume={61},
   date={2019},
   pages={474--506},
   doi={10.1137/17M1150153},
}

\bib{shen2010numerical}{article}{
   author={Shen, J.},
   author={Yang, X.},
   title={Numerical approximations of {Allen--Cahn} and {Cahn--Hilliard} equations},
   journal={Discrete Contin. Dyn. Syst.},
   volume={28},
   date={2010},
   pages={1669--1691},
   doi={10.3934/dcds.2010.28.1669},
}

\bib{shi2025modified}{article}{
   author={Shi, B.},
   author={Han, Y.},
   author={Ma, C.},
   author={Majumdar, A.},
   author={Zhang, L.},
   title={A modified {Landau--de Gennes} theory for smectic liquid crystals: phase transitions and
   structural transitions},
   journal={SIAM J. Appl. Math.},
   volume={85},
   date={2025},
   pages={821--847},
   doi={10.1137/24M1682105},
}

\bib{wang2011energy}{article}{
   author={Wang, C.},
   author={Wise, S. M.},
   title={An energy stable and convergent finite-difference scheme for the modified phase field crystal
   equation},
   journal={SIAM J. Numer. Anal.},
   volume={49},
   date={2011},
   pages={945--969},
   doi={10.1137/090752675},
}

\bib{wang2021modelling}{article}{
   author={Wang, W.},
   author={Zhang, L.},
   author={Zhang, P.},
   title={Modelling and computation of liquid crystals},
   journal={Acta Numer.},
   volume={30},
   date={2021},
   pages={765--851},
   doi={10.1017/S0962492921000088},
}

\bib{wise2009energy}{article}{
   author={Wise, S. M.},
   author={Wang, C.},
   author={Lowengrub, J. S.},
   title={An energy-stable and convergent finite-difference scheme for the phase field crystal equation},
   journal={SIAM J. Numer. Anal.},
   volume={47},
   date={2009},
   pages={2269--2288},
   doi={10.1137/080738143},
}

\bib{xia2023variational}{article}{
   author={Xia, J.},
   author={Farrell, P. E.},
   title={Variational and numerical analysis of a {Q}-tensor model for smectic-{A} liquid crystals},
   journal={ESAIM Math. Model. Numer. Anal.},
   volume={57},
   date={2023},
   pages={693--716},
   doi={10.1051/m2an/2022083},
}

\bib{xia2024simple}{article}{
   author={Xia, J.},
   author={Han, Y.},
   title={Simple tensorial theory of smectic-{C} liquid crystals},
   journal={Phys. Rev. Res.},
   volume={6},
   date={2024},
   pages={033232},
   doi={10.1103/PhysRevResearch.6.033232},
}

\bib{xia2021structural}{article}{
   author={Xia, J.},
   author={MacLachlan, S.},
   author={Atherton, T. J.},
   author={Farrell, P. E.},
   title={Structural landscapes in geometrically frustrated smectics},
   journal={Phys. Rev. Lett.},
   volume={126},
   date={2021},
   pages={177801},
   doi={10.1103/PhysRevLett.126.177801},
}

\bib{xu2006stability}{article}{
   author={Xu, C.},
   author={Tang, T.},
   title={Stability analysis of large time-stepping methods for epitaxial growth models},
   journal={SIAM J. Numer. Anal.},
   volume={44},
   date={2006},
   pages={1759--1779},
   doi={10.1137/050628143},
}

\bib{xu2019efficient}{article}{
   author={Xu, Z.},
   author={Yang, X.},
   author={Zhang, H.},
   author={Xie, Z.},
   title={Efficient and linear schemes for anisotropic {Cahn--Hilliard} model using the
   stabilized-invariant energy quadratization ({S-IEQ}) approach},
   journal={Comput. Phys. Commun.},
   volume={238},
   date={2019},
   pages={36--49},
   doi={10.1016/j.cpc.2018.12.019},
}

\bib{yang2017linearly}{article}{
   author={Yang, X.},
   author={Han, D.},
   title={Linearly first- and second-order, unconditionally energy stable schemes for the phase field
   crystal model},
   journal={J. Comput. Phys.},
   volume={330},
   date={2017},
   pages={1116--1134},
   doi={10.1016/j.jcp.2016.10.020},
}

\bib{yang2020convergence}{article}{
   author={Yang, X.},
   author={Zhang, G.-D.},
   title={Convergence analysis for the invariant energy quadratization ({IEQ}) schemes for solving the
   {Cahn--Hilliard} and {Allen--Cahn} equations with general nonlinear potential},
   journal={J. Sci. Comput.},
   volume={82},
   date={2020},
   pages={55},
   doi={10.1007/s10915-020-01151-x},
}

% \bib{zhang2025novel}{article}{
%    author={Zhang, B.},
%    author={Zhou, C.},
%    author={Fu, H.},
%    title={A novel efficient generalized energy-optimized exponential {SAV} scheme with variable-step
%    {BDF}k method for gradient flows},
%    journal={Appl. Numer. Math.},
%    volume={210},
%    date={2025},
%    pages={39--63},
%    doi={10.1016/j.apnum.2024.12.005},
% }

% \bib{zhang2022generalized}{article}{
%    author={Zhang, Y.},
%    author={Shen, J.},
%    title={A generalized {SAV} approach with relaxation for dissipative systems},
%    journal={J. Comput. Phys.},
%    volume={464},
%    date={2022},
%    pages={111311},
%    doi={10.1016/j.jcp.2022.111311},
% }

\bib{zhao2017novel}{article}{
   author={Zhao, J.},
   author={Yang, X.},
   author={Gong, Y.},
   author={Wang, Q.},
   title={A novel linear second order unconditionally energy stable scheme for a hydrodynamic
   {Q}-tensor model of liquid crystals},
   journal={Comput. Methods Appl. Mech. Engrg.},
   volume={318},
   date={2017},
   pages={803--825},
   doi={10.1016/j.cma.2017.01.031},
}

\end{biblist}
\end{bibdiv}
\end{document}